\documentclass[a4paper, 11pt]{article}
\usepackage{amsmath,amssymb,amsthm,amscd}
\usepackage{cleveref}
\usepackage{cases}
\usepackage{mathrsfs}%Used for \mathscr
\usepackage{amsfonts}
\allowdisplaybreaks[4]
\usepackage[all]{xy}
\usepackage{comment}
\usepackage[top=30truemm,bottom=30truemm,left=15truemm,right=15truemm]{geometry}
\usepackage{tikz}
\usepackage{mathtools}%Used for Tikz
\usepackage{here}
\usepackage{url}

\usetikzlibrary{shapes,arrows,positioning}
\usetikzlibrary{shapes.misc}
\usetikzlibrary{decorations.markings}
\tikzset{cross/.style={cross out, draw=black, minimum size=2*(#1-\pgflinewidth), inner sep=0pt, outer sep=0pt},
cross/.default={1pt}}
\tikzset{->-/.style={decoration={
  markings,
  mark=at position #1 with {\arrow[scale=1.5]{>}}},postaction={decorate}}}

\tikzset{-<-/.style={decoration={
  markings,
  mark=at position #1 with {\arrow[scale=1.5]{<}}},postaction={decorate}}}

\tikzset{midstealth/.style={decoration={
  markings,
  mark=at position #1 with {\arrow{stealth}}},postaction={decorate}}}
\newtheoremstyle{break}
  {}%         
  {}%         
  {\itshape}
  {}%         
  {\bfseries}
  {.}%        
  {\newline}% 
  {}%         

 \makeatletter
    
    \@addtoreset{equation}{section}
  \makeatother

\theoremstyle{plain}
\newtheorem{thm}{Theorem}[section]
\newtheorem{lem}[thm]{Lemma}
\newtheorem{prop}[thm]{Proposition}
\newtheorem{cor}[thm]{Corollary}
\newtheorem{exa}[thm]{Example}

\newtheorem{rem}[thm]{Remark}

\newcommand{\conj}[1]{\overline{#1}}
\newcommand{\abs}[1]{\lvert#1\rvert}

\DeclareMathOperator{\Ker}{Ker}
\DeclareMathOperator{\Hom}{Hom}

\DeclareMathOperator{\Homo}{H}

\DeclareMathOperator{\sol}{Sol}
\DeclareMathOperator{\diag}{diag}

\DeclareMathOperator{\cone}{cone}
\newcommand{\barsigma}{\overline{\sigma}}

\newcommand{\R}{\mathbb{R}}

\newcommand{\C}{\mathbb{C}}
\newcommand{\DD}{\mathcal{D}}

\newcommand{\A}{\mathbb{A}}
\newcommand{\Q}{\mathbb{Q}}
\newcommand{\Z}{\mathbb{Z}}

\newcommand{\lef }{\left\{ }
\newcommand{\righ }{\right\} }
\newcommand{\Gm}{\mathbb{G}_m}

\newcommand{\ii}{\sqrt{-1}}
\newcommand{\s}{\sigma}

\newcommand{\bs}{\barsigma}

\newcommand{\oz}{\mathring{\zeta}}

\newcommand*{\transp}[2][-3mu]{\ensuremath{\mskip1mu\prescript{\smash{\mathrm t\mkern#1}}{}{\mathstrut#2}}}

\newcommand{\tpi}{2\pi \sqrt{-1}}
\newcommand{\ot}{\otimes}
\DeclareMathOperator{\lcm}{lcm}

\newcommand{\tC}{\tilde{C}}
\newcommand{\bC}{\bar{C}}
\newcommand{\X}{X}
\newcommand{\tX}{\tilde{X}}
\newcommand{\bX}{\bar{X}}
\newcommand{\uu}{u}
\newcommand{\tu}{\tilde{u}}
\newcommand{\bu}{\bar{u}}
\newcommand{\bM}{\bar{M}}
\newcommand{\Ihgoto}{I_{h,1}}
\newcommand{\bIhgoto}{\bar{I}_{h,1}}
\crefname{thm}{Theorem}{Theorems}
\crefname{bthm}{Theorem}{Theorems}
\crefname{bprop}{Proposition}{Propositions}
\crefname{blem}{Lemma}{Lemmata}
\crefname{prop}{Proposition}{Propositions}
\crefname{lem}{Lemma}{Lemmata}
\crefname{bcor}{Corollary}{Corollary}
\crefname{cor}{Corollary}{Corollary}
\crefname{defn}{Definition}{Definitions}
\crefname{rem}{Remark}{Remarks}
\crefname{exa}{Example}{Exmples}
\title{Homology and cohomology intersection numbers of GKZ systems}
\author{
Yoshiaki Goto\footnote{ General Education, Otaru University of Commerce, 3-5-21 Midori, Otaru 047-8501, Japan.\newline\indent e-mail: \texttt{goto@res.otaru-uc.ac.jp}}
\and
Saiei-Jaeyeong Matsubara-Heo\footnote{ Graduate School of Science, Kobe  University, 1-1 Rokkodai, Nada-ku, Kobe 657-8501, Japan.\newline\indent e-mail: \texttt{saiei@math.kobe-u.ac.jp}}}

\begin{document}

\date{}
\maketitle

\begin{abstract}
\noindent
We describe the homology intersection form associated to regular holonomic GKZ systems in terms of the combinatorics of regular triangulations. Combining this result with the twisted period relation, we obtain a formula of cohomology intersection numbers in terms of a Laurent series. We show that the cohomology intersection number depends rationally on the parameters. We also prove a conjecture of F. Beukers and C.Verschoor on the signature of the monodromy invariant hermitian form. This is a continuation of the previous work \cite{MatsubaraEuler}.
\end{abstract}
%%%%%%%%%%%%%%%%%%%%%%%%%%
%%%%%% goto  %%%%%%%%%%%%%
%%%%%%%%%%%%%%%%%%%%%%%%%%
\section{Introduction}
The study of hypergeometric integrals has a long history. The starting point of modern approach to hypergeometric integrals is the work of K.Aomoto (\cite{Aomoto1}, \cite{Aomoto2}, \cite{AomotoLesEquation}). He systematically employed the viewpoint of twisted (co)homology group which is an alias for (co)homology group with local system coefficients. We consider an integral of the following form:
\begin{equation}\label{EulerInt}
f_{\Gamma}(z)=\int_\Gamma h_{1}(x)^{-\gamma_1}\cdots h_{k}(x)^{-\gamma_k}x^{c}\omega,
\end{equation}
where $h_l(x;z)=h_{l,z^{(l)}}(x)=\sum_{j=1}^{N_l}z_j^{(l)}x^{{\bf a}^{(l)}(j)}$ ($l=1,\dots,k$) are Laurent polynomials in torus variables $x=(x_1,\dots,x_n)$, $\gamma_l\in\C$ and $c={}^t(c_1,\dots,c_n)\in\C^{n\times 1}$ are parameters, $x^c=x_1^{c_1}\dots x_n^{c_n}$, $\Gamma$ is a suitable integration cycle, and $\omega$ is an algebraic $n$-form which has poles along $D=\{ x\in\C^n\mid x_1\dots x_n h_1(x)\dots h_k(x)=0\}$. As a function of the independent variable $z=(z_j^{(l)})_{j,l}$, the integral (\ref{EulerInt}) defines a hypergeometric function. We call the integral (\ref{EulerInt}) Euler integral. We can naturally define the twisted de Rham cohomology group associated to the Euler integral (\ref{EulerInt}). Let $\Gm^n={\rm Specm}\;\C[x_1^\pm,\dots,x_n^\pm]$ be an algebraic torus. We put $V_z=\Gm^n\setminus D$. We can define a suitable integrable connection $\nabla_x:\mathcal{O}_{V_z}\rightarrow\Omega^1_{V_z}$ and a local system $\mathcal{L}_z$ so that the Euler integral (\ref{EulerInt}) can be regarded as a result of a twisted period pairing $\Homo_n\left( V_z^{an};\mathcal{L}_z\right)\times\Homo_{\rm dR}^n\left( V_z;(\mathcal{O}_{V_z},\nabla_x)\right)\ni ([\Gamma],[\omega])\mapsto \int_{\Gamma}h_{1}(x)^{-\gamma_1}\cdots h_{k}(x)^{-\gamma_k}x^{c}\omega\in\C$. Here, the algebraic de Rham cohomology group $\Homo_{\rm dR}^*\left( V_z;(\mathcal{O}_{V_z},\nabla_x)\right)$ is defined as the hypercohomology group $\mathbb{H}^*\left( V_z;(\cdots\overset{\nabla_x}{\rightarrow}\Omega_{V_z}^\bullet\overset{\nabla_x}{\rightarrow}\cdots)\right)$ where $\Omega_{V_z}^p$ is the sheaf of algebraic $p$-forms on $V_z$. We call the homology group $\Homo_n\left( V_z^{an};\mathcal{L}_z\right)$ the twisted homology group. Under a genericity assumption on the parameters $\gamma_l$ and $c$, we have the vanishing result $\Homo_{\rm dR}^m\left( V_z;(\mathcal{O}_{V_z},\nabla_x)\right)=0$ $(m\neq n)$.
We can also define perfect pairings $\langle\bullet,\bullet\rangle_{ch}:\Homo_{\rm dR}^n\left( V_z;(\mathcal{O}_{V_z},\nabla_x)\right)\times \Homo_{\rm dR}^{n}\left( V_z;(\mathcal{O}_{V_z}^\vee,\nabla_x^\vee)\right)\rightarrow\C$ and $\langle\bullet,\bullet\rangle_{h}:\Homo_n\left( V_z^{an};\mathcal{L}_z\right)\times \Homo_{n}\left( V_z^{an};\mathcal{L}_z^\vee\right)\rightarrow\C$ which are called the cohomology and homology intersection form respectively. Here, the superscript $\vee$ stands for the dual. Inspired by the works of Aomoto, many people developed the intersection theory of twisted (co)homology groups. They discovered that the explicit formulas of (co)homology intersection numbers play important roles in the theory of hypergeometric integrals: twisted homology intersection numbers can be used to determine the image of the period map of a family of algebraic varieties (\cite{DeligneMostow}, \cite{MSY},\cite{YoshidaLove}); cohomology intersection numbers appear as important candidates of a family of functional relations called quadratic relations (\cite{ChoMatsumoto}, \cite{MatsumotoIntersection}). When the Laurent polynomials $h_l$ are all linear polynomials, various techniques of computing (co)homology intersection numbers have been developed and applied to hypergeometric integrals of particular types. They are, generalized hypergeometric functions ${}_pF_{p-1}$, Appell-Lauricella's $F_A,F_B,F_D$ functions, Aomoto-Gelfand functions, and Selberg integrals (\cite{ChoMatsumoto}, \cite{Goto}, \cite{KitaYoshida1}, \cite{KitaYoshida2}, \cite{MatsumotoIntersection}, \cite{MatsumotoYoshida}, \cite{Mim}, \cite{MimachiYoshida}). For Lauricella's $F_C$ functions,  (co)homology intersection numbers are computed in \cite{Goto2}, \cite{Goto3}.

In \cite{MatsubaraEuler}, the second author developed a systematic method of constructing a basis of the twisted homology group associated to a class of integrals which includes Euler integral (\ref{EulerInt}). The aim of this paper is to compute the homology intersection numbers of these cycles for Euler integral (\ref{EulerInt}). In order to explain the result of \cite{MatsubaraEuler} in our setting, let us recall the definition of GKZ system. GKZ system is a special holonomic system introduced by I.M.Gelfand, M.I.Graev, M.M.Kapranov, and A.V.Zelevinsky (\cite{GelfandGraevZelevinsky}, \cite{GKZToral}). The system is determined by two inputs: a $d\times n$ ($d<n$) integer matrix $A=({\bf a}(1)|\cdots|{\bf a}(n))$ and a parameter vector $\delta\in\C^{d}$. GKZ system $M_A(\delta)$ is defined by
\begin{subnumcases}{M_A(\delta):}
E_i\cdot f(z)=0 &($i=1,\dots, d$)\label{EulerEq}\\
\Box_u\cdot f(z)\hspace{-0.8mm}=0& $\left( 0\neq u={}^t(u_1,\dots,u_{n})\in L_A=\Ker(A\times:\Z^{n\times 1}\rightarrow\Z^{d\times 1})\right)$,\label{ultrahyperbolic}
\end{subnumcases}
where $E_i$ and $\Box_u$ are differential operators defined by 

\begin{equation}\label{HGOperators}
E_i=\sum_{j=1}^{n}a_{ij}z_j\frac{\partial}{\partial z_j}+\delta_i,\;\;\;
\Box_u=\prod_{u_j>0}\left(\frac{\partial}{\partial z_j}\right)^{u_j}-\prod_{u_j<0}\left(\frac{\partial}{\partial z_j}\right)^{-u_j}.
\end{equation}

\noindent
Throughout this paper, we assume an additional condition that the column vectors of $A$ generate the lattice $\Z^{d\times 1}$. Let us write $\DD_{\A^n}$ for the Weyl algebra on $\A^n$ and write $H_A(\delta)$ for the left ideal of $\DD_{\A^n}$ generated by the operators (\ref{HGOperators}). We also call the left $\DD_{\A^n}$-module $M_A(\delta)=\DD_{\A^n}/H_A(\delta)$ GKZ system. The fundamental property of GKZ system $M_A(\delta)$ is that it is always holonomic (\cite{Adolphson}), which implies that the stalk of the sheaf of holomorphic solutions at a generic point is finite dimensional. It is proved in \cite{GKZEuler} that (\ref{EulerInt}) is a solution of $M_A(\delta)$ with particular choices of the $(n+k)\times N$ matrix $A$ and parameters $\delta$ where $N=N_1+\dots+N_k$. In view of the important result of R.Hotta (\cite{Hotta}), GKZ system is regular holonomic if and only if all the solutions of it admit Euler integral representation when the parameter $\delta$ satisfies a certain genericity condition. See also \cite[Theorem 5.9]{FernandezFernandez} and \cite[COROLLARY 3.16]{SchulzeWalther}. Therefore, we expect that the theory of GKZ system tells us some information of the integral (\ref{EulerInt}).

An important combinatorial structure of GKZ system is the secondary fan ${\rm Fan}(A)$ (\cite{DeLoeraRambauSantos}, \cite[Chapter 7]{GKZbook}). Each open cone of ${\rm Fan}(A)$ corresponds to a combinatorial operation $T$ called a regular triangulation. It was discussed in \cite{FernandezFernandez},  \cite{GKZToral} and \cite{SST} that each regular triangulation $T$ gives rise to a particular basis $\Phi_T$ of solutions of $M_A(\delta)$. Each element of $\Phi_T$ can be expressed in terms of hypergeometric series in several variables. Let $\sol_{M_A(\delta),z}$ be the vector space of local solutions of $M_A(\delta)$ at the point $z$. Under a genericity assumption on the parameters $\delta$, one has a canonical isomorphism $\int:\Homo_n\left( V_z^{an};\mathcal{L}_z\right)\tilde{\rightarrow}\sol_{M_A(\delta),z}$ at a generic point $z$ (\cite[Theorem 2.10]{GKZEuler}, see also \cite[Theorem 3.14]{MatsubaraEuler}). In \cite[\S 6]{MatsubaraEuler}, the second author constructed a basis $\Gamma_T$ of the twisted homology group $\Homo_n\left( V_z^{an};\mathcal{L}_z\right)$ so that the transition matrix between two bases $\int(\Gamma_T)$ and $\Phi_T$ has an explicit formula. Replacing the parameter $\delta$ by $-\delta$, we also obtain a basis $\check{\Gamma}_T$ of $\Homo_n\left( V_z^{an};\mathcal{L}_z^\vee\right)$. When the regular triangulation $T$ is unimodular, one can obtain a closed formula of any homology intersection number $\langle\gamma,\gamma^\vee\rangle_h$ with $\gamma\in\Gamma_T$ and $\gamma^\vee\in\check{\Gamma}_T$ (\cite[Theorem 7.5]{MatsubaraEuler}). As an immediate consequence of the twisted analogue of Riemann-Hodge bilinear relation, one has an expansion formula of a cohomology intersection number in terms of Laurent series in $z$ (\cite[Theorem 8.1]{MatsubaraEuler}). This formula was later applied to constructing an algorithm of computing a closed form of cohomology intersection numbers in \cite{MatsubaraTakayama}. 

The purpose of this paper is to remove the assumption that the regular triangulation $T$ is unimodular in the theorems mentioned above. The main result of the paper is the following 

\newpage

\begin{thm}\label{thm:Intro}
Take any regular triangulation $T$, ${\bf a}, {\bf a}^\prime\in\Z^{n\times 1}$ and ${\bf b}, {\bf b}^\prime\in\Z^{k\times 1}$. Suppose that $\gamma_l\notin\Z$ and $\delta$ is very generic with respect to $T$. Then, the following identity holds:
\begin{align}
&\frac{\langle x^{\bf a}h^{\bf b}\frac{dx}{x},x^{\bf a^\prime}h^{\bf b^\prime}\frac{dx}{x} \rangle_{ch}}{(2\pi\ii)^n}\nonumber\\
=&(-1)^{|{\bf b}|+|{\bf b^\prime}|}(\gamma-{\bf b})_{\bf b}(-\gamma-{\bf b^\prime})_{\bf b^\prime}\gamma_1\cdots\gamma_k\times\nonumber\\
&\sum_{\s\in T}\sum_{[A_{\bs}{\bf k}]\in\Z^{(n+k)\times 1}/\Z A_\s}\frac{(-1)^{|{\bf k}|}}{|\det A_\s|}\frac{\pi^{n+k}}{\sin\pi A_\s^{-1}(\delta+A_{\bs}{\bf k})}
\varphi_{\s,{\bf k}_{\bf a,b}}\Big( z;\substack{\gamma-{\bf b}\\ c+{\bf a}}\Big)
\varphi_{\s,{\bf k}^\prime_{\bf a^\prime,b^\prime}}\Big( z;\substack{-\gamma-{\bf b^\prime}\\ -c+{\bf a^\prime}}\Big).
\end{align}
Here, ${\bf k}_{\bf a,b}$ (resp. ${\bf k}^\prime_{\bf a^\prime,b^\prime}$) denotes the element such that 
$[A_{\bs}{\bf k}+
\begin{pmatrix}
{\bf b}\\
-{\bf a}
\end{pmatrix}
]=
[A_{\bs}{\bf k}_{\bf a,b}]$ (resp. $[-A_{\bs}{\bf k}+
\begin{pmatrix}
{\bf b^\prime}\\
-{\bf a^\prime}
\end{pmatrix}
]=
[A_{\bs}{\bf k}^\prime_{\bf a^\prime,b^\prime}]$) holds in the group $\Z^{(n+k)\times 1}/\Z A_\s$.
\end{thm}

\noindent
The notation will be explained in \S\ref{section:TPR}. We only note that the right-hand side is a convergent Laurent series in $z$ and cohomology classes of the form $\left[x^{\bf a}h^{\bf b}\frac{dx}{x}\right]$ span the twisted cohomology group $\Homo_{\rm dR}^n\left( V_z;(\mathcal{O}_{V_z},\nabla_x)\right)$. When we can evaluate the left-hand side, the theorem above gives a non-trivial functional relation of hypergeometric functions.

Theorem \ref{thm:Intro} reveals the combinatorial nature of the cohomology intersection form of GKZ system. We fix bases $\{\phi_i\}_i$ of $\Homo_{\rm dR}^n\left( V_z;(\mathcal{O}_{V_z},\nabla_x)\right)$ and $\{\psi_i\}_i$ of $\Homo_{\rm dR}^{n}\left( V_z;(\mathcal{O}_{V_z}^\vee,\nabla_x^\vee)\right)$ which depend rationally on the variable $z$. We set $I_{ch}:=(\langle\phi_i,\psi_j\rangle_{ch})_{i,j}$ and call it the cohomology intersection matrix. It was pointed out in \cite[Theorem 3.4]{MatsubaraTakayama} that the secondary fan ${\rm Fan}(A)$ is a refinement of the normal fan $N(g)$ of the common denominator $g$ of entries of $I_{ch}$. On the other hand, there is an injective correspondence from maximal cones of $N(g)$ to Laurent expansions of $I_{ch}$ (\cite[Chap. 6, Corollary 1.8]{GKZbook}). Therefore, for a maximal cone $C$ of $N(g)$, there is a regular triangulation $T$ such that the Laurent expansion of $I_{ch}$ corresponding to $C$ is given by Theorem \ref{thm:Intro}. The deduction of Theorem \ref{thm:Intro} relies on detailed computation of the homology intersection form $\langle\bullet,\bullet\rangle_h$ (Theorem \ref{thm:SigmaIntersectionMatrix1}). A byproduct of this theorem is a combinatorial formula of the signature of the  invariant hermitian form of monodromy representation of a regular holonomic GKZ system originally conjectured by F. Beukers and C. Verschoor in a different context (\cite{Verschoor}).

Let us summarize the content of this paper. We discuss our results in \S2. In \S2.1 and \S2.2, we recall some basic results of Euler integral representation and the definition of the cohomology intersection form. In \S2.3, we state our main result Theorem \ref{thm:Intro} as Theorem \ref{TheQuadraticRelation}. The proof of Theorem \ref{TheQuadraticRelation} is based on Theorem \ref{thm:SigmaIntersectionMatrix1} whose proof is given in \S3. In \S2.4, we give some immediate consequences of our main result. In \S\ref{subsec:BV}, we discuss a proof of the conjecture of F. Beukers and C. Verschoor. The last section is devoted to the step-by-step proof of Theorem \ref{thm:SigmaIntersectionMatrix1}. The proof is reduced to the study of the case when $A$ is a square matrix. In \S3.1-4, we discuss the reduced cases. Based on these preparations, we give a proof of Theorem \ref{thm:SigmaIntersectionMatrix1} in \S3.5.

\vspace{2.5mm}

\noindent
\underline{\bf Notation} For any commutative ring $R$ and for any pair of finite sets $I$ and $J$, $R^{I\times J}$ denotes the set of matrices with entries in $R$ whose rows (resp. columns) are indexed by $I$ (resp. $J$). For any univariate function $F$ and for any vector $w={}^t(w_1,\dots,w_n)\in\C^{n\times 1}$, we define $F(w)$ by $F(w)=F(w_1)\cdots F(w_n)$. For any vector ${\bf k}=(k_1,\dots,k_n)\in\Z^{n}$, we set $|{\bf k}|=k_1+\cdots+k_n$.

\section{The twisted period relation for $\Gamma$-series}\label{section:TPR}

\subsection{Euler integral representations of regular holonomic GKZ systems}
In this subsection, we recall the relation between the integral (\ref{EulerInt}) and GKZ systems. We set $N=N_1+\dots+N_k,$ and $\delta=\transp (\gamma_1,\dots,\gamma_k,c_1,\dots,c_n)\in\C^{(n+k)\times 1}$. We put $A_l=({\bf a}^{(l)}(1)|\dots|{\bf a}^{(l)}(N_l))$ and define an $(n+k)\times N$ matrix $A$ by

\begin{equation}\label{CayleyConfigu}
A
=
\left(
\begin{array}{ccc|ccc|c|ccc}
1&\cdots&1&0&\cdots&0&\cdots&0&\cdots&0\\
\hline
0&\cdots&0&1&\cdots&1&\cdots&0&\cdots&0\\
\hline
 &\vdots& & &\vdots& &\ddots& &\vdots& \\
\hline
0&\cdots&0&0&\cdots&0&\cdots&1&\cdots&1\\
\hline
 &A_1& & &A_2& &\cdots & &A_k& 
\end{array}
\right).
\end{equation}

\noindent
Then, it is easy to see that (\ref{EulerInt}) is a solution of the GKZ system $M_A(\delta)$ (\cite[Theorem 2.7]{GKZEuler}). If $N_l=1$ for some $l$, the Laurent polynomial $h_l$ is a monomial $x^{\bf a}$ and we have a factorization $h_l(x)^{-\gamma_l}x^c=x^{c-\gamma_l{\bf a}}$. Therefore, in the followings, we may assume that $N_l\geq 2$ for any $l=1,\dots,k$. We further assume that the column vectors of $A$ generate the lattice $\Z^{(n+k)\times 1}$.

Let $z\in\C^N$ be a nonsingular point in the sense of \cite[Definition 3.7]{MatsubaraEuler}. We set 
\begin{equation}
V_z=\left\{ x\in(\Gm)^n\mid h_{1,z^{(1)}}(x)\cdots h_{k,z^{(k)}}(x)\neq 0\right\}
\end{equation}
and define an integrable connection $\nabla_x=d_x-\sum_{l=1}^k\gamma_l\frac{d_xh_l}{h_l}\wedge+\sum_{i=1}^nc_i\frac{dx_i}{x_i}\wedge:\mathcal{O}_{V_z}\rightarrow\Omega^1_{V_z}$. Here, $d_x$ is the exterior derivative on $V_z$. The dual connection of $(\mathcal{O}_{V_z},\nabla_x)$ is denoted by $(\mathcal{O}_{V_z}^\vee,\nabla_x^\vee)$. The local system of flat sections of $(\mathcal{O}_{V_z^{an}}^\vee,\nabla_x^{an\vee})$ is denoted by the symbol $\mathcal{L}_z$. Here, the superscript $an$ stands for the analytification. Let us write $\Phi_z$ for the multivalued function $h_{1,z^{(1)}}(x)^{-\gamma_1}\cdots h_{k,z^{(k)}}(x)^{-\gamma_k}x^c$. We define the twisted homology group $\Homo_k(V_z^{an};\mathcal{L}_z)$ as the $(-k)$-th cohomology group of the complex $\cdots\overset{\partial}{\rightarrow}\mathcal{C}_p(V_z^{an};\mathcal{L}_z)\overset{\partial}{\rightarrow}\mathcal{C}_{p-1}(V_z^{an};\mathcal{L}_z)\overset{\partial}{\rightarrow}\cdots$. Here, $\mathcal{C}_p(V_z^{an};\mathcal{L}_z)$ is the set of singular $p$-chains with coefficients in the local system $\mathcal{L}_z$, which is supposed to be at the $(-p)$-th position of the complex and $\partial$ is the boundary operator(\cite[\S 2.1.6]{AomotoKita}). Note that any singular $p$-chain $\Gamma$ takes the form $\Gamma=\sum_{j}a_j\Delta_j\otimes\Phi_z\restriction_{\Delta_j}$ where $a_j$ are complex numbers, $\Delta_j$ are continuous maps from a $p$-simplex to $V_z^{an}$ and $\Phi_z\restriction_{\Delta_j}$ are a section of $\mathcal{L}_z$ on the image of $c_j$. We call any element $\Gamma\in \mathcal{Z}_p(V_z^{an};\mathcal{L}_z):=\Ker\left(\partial:\mathcal{C}_p(V_z^{an};\mathcal{L}_z)\rightarrow\mathcal{C}_{p-1}(V_z^{an};\mathcal{L}_z)\right)$ a twisted cycle.

When $\delta$ is non-resonant in the sense of \cite[2.9]{GKZEuler} and $\gamma_l\notin\Z$, we have a canonical isomorphism 
\begin{equation}\label{eqn:Integration}
\Homo_n(V_z^{an},\mathcal{L}_z)\tilde{\rightarrow}\sol_{M_A(\delta),z},
\end{equation}
which is given by $\Homo_n(V_z^{an},\mathcal{L}_z)\ni[\Gamma]\mapsto \int_\Gamma \Phi_z\frac{dx}{x}$ (\cite[Theorem 2.10]{GKZEuler}, \cite[Theorem 3.13]{MatsubaraEuler}). Here, $\sol_{M_A(\delta),z}$ is the stalk of the solution sheaf of $M_A(\delta)$ at the point $z$. While this correspondence is highly transcendental, the solution space $\sol_{M_A(\delta),z}$ has a combinatorial structure when $z$ is close to a special point in a suitable toric compactification.

We briefly recall the construction of a basis of solutions of GKZ system in terms of $\Gamma$-series following the exposition of M.-C. Fern\'andez-Fern\'andez (\cite{FernandezFernandez}). For any vector $v\in\C^{N\times 1}$ such that $Av=-\delta,$ we put
\begin{equation}\label{GammaSeries1}
\varphi_v(z)=\displaystyle\sum_{u\in L_A}\frac{z^{u+v}}{\Gamma(1+u+v)}.
\end{equation}
It can readily be seen that $\varphi_{v}(z)$ is a formal solution of $M_A(\delta)$ (\cite{GKZToral}). We call (\ref{GammaSeries1}) a $\Gamma$-series solution of $M_A(\delta)$.
For any subset $\tau\subset\{1,\dots,N\}$, $A_\tau$ denotes the matrix given by the columns of $A$ indexed by $\tau.$ In the following, we take $\sigma\subset\{1,\dots,N\}$ such that the cardinality $|\sigma|$ is equal to $n+k$ and $\det A_\sigma\neq 0.$ Taking a vector ${\bf k}\in\Z^{\bar{\sigma}\times 1},$ we put
\begin{equation}
v_\sigma^{\bf k}=
\begin{pmatrix}
-A_{\sigma}^{-1}(\delta+A_{\bar{\sigma}}{\bf k})\\
{\bf k}
\end{pmatrix},
\end{equation}
where $\bs$ denotes the complement $\{ 1,\dots,N\}\setminus \s$. Then, by a direct computation, we have
\begin{equation}\label{seriesphi}
\varphi_{\s,{\bf k}}(z;\delta)\overset{\rm def}{=}\varphi_{v_\sigma^{\bf k}}(z)=
z_\sigma^{-A_\sigma^{-1}\delta}
\sum_{{\bf k+m}\in\Lambda_{\bf k}}\frac{(z_\sigma^{-A_\sigma^{-1}A_{\bar{\sigma}}}z_{\bar{\sigma}})^{\bf k+m}}{\Gamma({\bf 1}_\sigma-A_\sigma^{-1}(\delta+A_{\bar{\sigma}}({\bf k+m}))){\bf (k+m)!}},
\end{equation}
where $\Lambda_{\bf k}$ is given by
\begin{equation}\label{lambdak}
\Lambda_{\bf k}=\Big\{{\bf k+m}\in\Z^{\bar{\sigma}\times 1}_{\geq 0}\mid A_{\bar{\sigma}}{\bf m}\in\Z A_\sigma\Big\}.
\end{equation}
By \cite[Lemma 3.1,3.2, Remark 3.4.]{FernandezFernandez}, a complete set of representatives $\{ [A_{\barsigma}{\bf k}(i)]\}_{i=1}^{r_\s}$ of the finite abelian group $\Z^{(n+k)\times 1}/\Z A_\sigma$ induces a decomposition $\Z^{\barsigma\times 1}_{\geq 0}=\bigsqcup_{j=1}^{r_\s}\Lambda_{{\bf k}(j)}.$  Therefore, we can observe that $\{\varphi_{\sigma,{\bf k}(i)}(z;\delta)\}_{i=1}^{r_\s}$ is a set of $r_\s$ linearly independent formal solutions of $M_A(\delta)$ unless $\varphi_{\sigma,{\bf k}(i)}(z;\delta)=0$ for some $i$. In order to ensure that $\varphi_{\s,{\bf k}(i)}(z;\delta)$ does not vanish, we say that a parameter vector $\delta$ is very generic with respect to $\sigma$ if $A_\sigma^{-1}(\delta+A_{\bar{\sigma}}{\bf m})$ does not contain any integral entry for any ${\bf m}\in\mathbb{Z}_{\geq 0}^{\bar{\sigma}\times 1}.$ As is well-known in the literature, under a genericity condition, we can construct a basis of holomorphic solutions of  GKZ system $M_A(\delta)$ consisting of $\Gamma$-series with the aid of a regular triangulation. Let us recall the definition of regular triangulation. In general, for any subset $\sigma$ of $\{1,\dots,N\},$ $\cone(\sigma)$ denotes the positive span of the column vectors $\{{\bf a}(1),\dots,{\bf a}(N)\}$ of $A$, i.e., $\cone(\sigma)=\sum_{i\in\sigma}\R_{\geq 0}{\bf a}(i).$ We often identify a subset $\sigma\subset\{1,\dots,N\}$ with the corresponding set of vectors $\{{\bf a}(i)\}_{i\in\sigma}$ or with the set $\cone(\s)$. A collection $T$ of subsets of $\{1,\dots,N\}$ is called a triangulation if $\{\cone(\sigma)\mid \sigma\in T\}$ is a set of cones in a simplicial fan whose support equals $\cone(A)$. We regard $\Z^{1\times N}$ as the dual lattice of $\Z^{N\times 1}$ via the standard dot product. Let $\pi_A:\Z^{1\times N}\rightarrow L_A^\vee$ be the dual morphism of the natural inclusion $L_A\hookrightarrow \Z^{N\times 1}$ where $L_A^\vee$ is the dual lattice $\Hom_{\Z}(L_A,\Z)$. By abuse of notation, we continue to write $\pi_A$ for the linear map $\pi_A\underset{\Z}{\otimes}{\rm id}_{\R}:\R^{1\times N}\rightarrow L_A^\vee\underset{\Z}{\otimes}\R$ where ${\rm id}_{\R}:\R\rightarrow\R$ is the identity map. Then, for any generic choice of a vector $\omega\in\R^{1\times N},$ we can define a triangulation $T(\omega)$ as follows: A subset $\sigma\subset\{1,\dots,N\}$ belongs to $T(\omega)$ if there exists a vector ${\bf n}\in\R^{1\times (n+k)}$ such that ${\bf n}\cdot{\bf a}(i)=\omega_i$ if $i\in\sigma$ and ${\bf n}\cdot{\bf a}(j)<\omega_j$ if $ j\in\barsigma.$ A triangulation $T$ is called a regular triangulation if $T=T(\omega)$ for some $\omega\in\R^{1\times N}.$ For a fixed regular triangulation $T$, we say that the parameter vector $\delta$ is very generic if it is very generic with respect to any $\sigma\in T$. It is easy to see that if $\delta$ is very generic, $\delta$ must be non-resonant in the sense of \cite[2.9]{GKZEuler}. Now suppose that $\delta$ is very generic. We set 
\begin{equation}
W_\sigma=\left\{z\in(\C^*)^N\mid {\rm abs}\left(z_\sigma^{-A_\sigma^{-1}{\bf a}(j)}z_{j}\right)<R, \text{for all } j\in\bs\right\},
\end{equation}
where $R>0$ is a small positive real number and abs stands for the absolute value. With these terminologies, the following result is a special case of \cite[Theorem 6.7.]{FernandezFernandez}.

\begin{prop}
Suppose that a matrix of the form (\ref{CayleyConfigu}) is given. Fix a regular triangulation $T$ and assume $\delta$ is very generic. Then, the set 
$\Phi_T=\bigcup_{\sigma\in T}\left\{ \varphi_{\sigma,{\bf k}(i)}(z;\delta)\right\}_{i=1}^{r_\s}$ 
is a basis of holomorphic solutions of $M_A(\delta)$ on $W_{T}\overset{def}{=}\bigcap_{\sigma\in T}W_\sigma\neq\varnothing$.
\end{prop}

For a regular triangulation $T$, we set
\begin{equation}
C_T=\Big\{ \omega\in\R^{1\times N}\mid T(\omega)=T\Big\}.
\end{equation}

\noindent
We cite a fundamental result of Gelfand, Kapranov, and Zelevinsky (\cite[Chapter 7, Proposition 1.5.]{GKZbook},\cite[Theorem 5.2.11.]{DeLoeraRambauSantos}).

\begin{thm}[\cite{GKZbook},\cite{DeLoeraRambauSantos}]
There exists a projective fan ${\rm Fan}(A)$ in $\R^{1\times N}$ whose maximal cones are precisely $\{ C_T\}_{T: \text{regular triangulation}}$. The fan ${\rm Fan}(A)$ is called the secondary fan.
\end{thm}

\begin{rem}
Let $F$ be a fan obtained by applying the projection $\pi_A$ to each cone of ${\rm Fan}(A)$. By definition, each cone of ${\rm Fan}(A)$ is a pull-back of a cone of $F$ through the projection $\pi_A$. Therefore, the fan $F$ is also called the secondary fan.
\end{rem}

The basis $\{ \Gamma_{\s,\tilde{\bf k}}\}_{[\tilde{\bf k}]\in\Z^{\s\times 1}/\Z \transp A_{\s}}$ of $\Homo_n(V_z^{an},\mathcal{L}_z)$ constructed in \cite[\S6]{MatsubaraEuler} can be related to the basis $\Phi_T$ through the isomorphism (\ref{eqn:Integration}). We set $N_0=1$ and $I_l=\{ N_0+\dots+N_{l-1},\dots,N_1+\dots+N_l\}$ $(l=1,\dots,k)$. For any $(n+k)$-simplex $\s$, we define $\s^{(l)}$ by $\s\cap I_l$. Let us cite \cite[Theorem 6.5]{MatsubaraEuler} in our setting.

\begin{thm}\label{thm:fundamentalthm}
Take a regular triangulation $T$. Assume that the parameter vector $\delta$ is very generic and that $\gamma_l\notin\Z$ for any $l=1,\dots,k$. Then, if one puts
\begin{equation}
f_{\s,\tilde{\bf k}(j)}(z;\delta)=\frac{1}{(2\pi\ii)^{n+k}}\int_{\Gamma_{\s,\tilde{\bf k}(j)}} h_{1,z^{(1)}}(x)^{-\gamma_1}\cdots h_{k,z^{(k)}}(x)^{-\gamma_k}x^{c}\frac{dx}{x},
\end{equation}
$\displaystyle\bigcup_{\s\in T}\{ f_{\sigma,\tilde{\bf k}(j)}(z)\}_{j=1}^{r_\s}$ is a basis of solutions of $M_A(\delta)$ on the non-empty open set $W_T$, where $\{\tilde{\bf k}(j)\}_{j=1}^{r_\s}$ is a complete system of representatives of $\Z^{\s\times 1}/\Z\transp{A}_\s$. Moreover, for each $\sigma\in T,$ one has a transformation formula 
\begin{equation}
\begin{pmatrix}
f_{\sigma,\tilde{\bf k}(1)}(z;\delta)\\
\vdots\\
f_{\sigma,\tilde{\bf k}(r_\s)}(z;\delta)
\end{pmatrix}
=
T_\sigma
\begin{pmatrix}
\varphi_{\sigma,{\bf k}(1)}(z;\delta)\\
\vdots\\
\varphi_{\sigma,{\bf k}(r_\s)}(z;\delta)
\end{pmatrix}.
\end{equation}
Here, $T_\sigma$ is an $r_\s\times r_\s$ matrix given by 
\begin{align}
T_\sigma=&\sqrt{r_\s}C_\s(\gamma)
\diag\Big( \exp\left\{
2\pi\ii\transp{
\tilde{\bf k}(i)
}
A_{\s}^{-1}\delta
\right\}\Big)_{i=1}^{r_\s}
U_\s.
\end{align}
Here, $C_\s(\gamma)$ is a complex number defined by
\begin{equation}
C_\s(\gamma)=
\frac{{\rm sgn}(A,\s)\displaystyle\prod_{l:|\s^{(l)}|>1}e^{-\pi\ii(1-\gamma_l)}\displaystyle\prod_{l:|\s^{(l)}|=1}e^{-\pi\ii\gamma_l}}{\det A_\s\Gamma(\gamma_1)\cdots\Gamma(\gamma_k)\displaystyle\prod_{l:|\s^{(l)}|=1}(1-e^{-2\pi\ii\gamma_l})}
\end{equation}
with $\text{sgn} (A,\s)=\pm 1$ and $U_\s$ is a unitary $r_\s\times r_\s$ matrix given by 
\begin{equation}
U_\s=\frac{1}{\sqrt{r_\s}}\Big(
\exp\left\{
2\pi\ii\transp{
\tilde{\bf k}(i)
}
A_{\s}^{-1}A_{\bs}{\bf k}(j)
\right\}\Big)_{i,j=1}^{r_\s}.
\end{equation}

\noindent
The precise formula of $\text{sgn} (A,\s)$ can be found in \cite[\S6]{MatsubaraEuler}. In particular, if $z$ is nonsingular in the sense of \cite[Definition 3.7]{MatsubaraEuler}, $\Gamma_T=\displaystyle\bigcup_{\s\in T}\left\{ \left[\Gamma_{\s,\tilde{\bf k}(j)}\right]\right\}_{j=1}^{r_\s}$ is a basis of the twisted homology group $\Homo_{n}(V_z^{an},\mathcal{L}_z).$
\end{thm}

\subsection{Generalities on homology and cohomology intersection pairing}
Before we state our main result, we need to recall the twisted period relation. We define the algebraic de Rham cohomology group $\Homo_{\rm dR}^*\left( V_z;(\mathcal{O}_{V_z},\nabla_x)\right)$ as the hypercohomology group $\mathbb{H}^*\left( V_z;(\cdots\overset{\nabla_x}{\rightarrow}\Omega_{V_z}^\bullet\overset{\nabla_x}{\rightarrow}\cdots)\right)$. Since $V_z$ is an affine variety, we can compute $\Homo_{\rm dR}^*\left( V_z;(\mathcal{O}_{V_z},\nabla_x)\right)$ by the formula $\Homo^*\left( \cdots\overset{\nabla_x}{\rightarrow}\Omega_{V_z}^\bullet(V_z)\overset{\nabla_x}{\rightarrow}\cdots\right)$. Let us write $\Homo^n_{dR,c}\left( V_z^{an};(\mathcal{O}_{V_z^{an}},\nabla_x^{an})\right)$ for the analytic de Rham cohomology group with compact support. By Poincar\'e-Verdier duality, the bilinear pairing 
\begin{equation}
\begin{array}{ccc}
\Homo^n_{dR,c}\left( V_z^{an};(\mathcal{O}_{V_z^{an}},\nabla_x^{an})\right)\times\Homo^n_{dR}\left( V_z^{an};(\mathcal{O}_{V_z^{an}}^\vee,\nabla_x^{an\vee})\right)&\rightarrow&\mathbb{C}\\
\rotatebox{90}{$\in$}& &\rotatebox{90}{$\in$}\\
([\phi],[\psi])&\mapsto&\int_{V_z^{an}}\phi\wedge\psi
\end{array}
\end{equation}
is perfect. We say that the regularization condition is satisfied if the canonical morphism $\Homo^n_{dR,c}\left( V_z^{an};(\mathcal{O}_{V_z^{an}},\nabla_x^{an})\right)$ $\rightarrow$ $\Homo^n_{dR}\left( V_z^{an};(\mathcal{O}_{V_z^{an}},\nabla_x^{an})\right)$ is an isomorphism. In the following, we always assume that the regularization condition is satisfied. Note that the regularization condition implies the purity $\Homo^m_{dR}\left( V_z;(\mathcal{O}_{V_z},\nabla_x)\right)=0$ $(m\neq n)$. When $\delta$ is non-resonant and $\gamma_l\notin\Z$, the regularization condition is true (\cite[2.15]{GKZEuler}, see also \cite[Theorem 2.12]{MatsubaraEuler}). Since $(\mathcal{O}_{V_z},\nabla_x)$ is a regular connection, the canonical morphism $\Homo^n_{dR}\left( V_z;(\mathcal{O}_{V_z},\nabla_x)\right)\rightarrow\Homo^n_{dR}\left( V_z^{an};(\mathcal{O}_{V_z^{an}},\nabla_x^{an})\right)$ is always an isomorphism by the Deligne-Grothendieck comparison theorem (\cite[Corollaire 6.3]{Deligne}). Therefore, we have a canonical isomorphism ${\rm reg}:\Homo^n_{dR}\left( V_z;(\mathcal{O}_{V_z},\nabla_x)\right)\rightarrow\Homo^n_{dR,c}\left( V_z^{an};(\mathcal{O}_{V_z^{an}},\nabla_x^{an})\right)$. Note that the Poincar\'e dual of the isomorphism ${\rm reg}$ is called a regularization map in the theory of special functions (\cite[\S 3.2]{AomotoKita}). Finally, we define the cohomology intersection form $\langle\bullet,\bullet\rangle_{ch}$ between algebraic de Rham cohomology groups by the formula
\begin{equation}\label{eqn:ACIF}
\begin{array}{cccc}
\langle\bullet,\bullet\rangle_{ch}:&\Homo^n_{dR}\left( V_z;(\mathcal{O}_{V_z},\nabla_x)\right)\times\Homo^n_{dR}\left( V_z;(\mathcal{O}_{V_z}^\vee,\nabla_x^{\vee})\right)&\rightarrow&\mathbb{C}\\
&\rotatebox{90}{$\in$}& &\rotatebox{90}{$\in$}\\
&([\phi],[\psi])&\mapsto&\int_{V_z^{an}}{\rm reg}(\phi)\wedge\psi.
\end{array}
\end{equation}
The value above is called the cohomology intersection number of
$[\phi]$ and $[\psi]$. By abuse of notation, we often write $\langle\phi,\psi\rangle_{ch}$ for $\langle[\phi],[\psi]\rangle_{ch}$. Through Poincar\'e duality $\Homo^n_{dR}\left( V_z^{an};(\mathcal{O}_{V_z^{an}},\nabla_x^{an})\right)\simeq \Homo_n(V_z^{an};\mathcal{L}_z^\vee)$ and $\Homo^n_{dR}\left( V_z^{an};(\mathcal{O}_{V_z^{an}}^\vee,\nabla_x^{an\vee})\right)\simeq \Homo_n(V_z^{an};\mathcal{L}_z)$, we can define a perfect pairing 
\begin{equation}\label{eqn:HIP}
\langle\bullet,\bullet\rangle_{h}:\Homo_n(V_z^{an};\mathcal{L}_z)\times\Homo_n(V_z^{an};\mathcal{L}_z^\vee)\rightarrow\mathbb{C}
\end{equation}
called the homology intersection pairing. By abuse of notation, we write $\langle\gamma,\gamma^\vee\rangle_{h}$ for $\langle[\gamma],[\gamma^\vee]\rangle_{h}$. Again by Poincar\'e duality, we can define twisted period pairings $\Homo_n(V_z^{an};\mathcal{L}_z)\times\Homo^n_{dR}\left( V_z;(\mathcal{O}_{V_z},\nabla_x)\right)\ni ([\Gamma],[\phi])\mapsto \int_{\Gamma} \Phi_z\phi\in\C$ and $\Homo_n(V_z^{an};\mathcal{L}_z^\vee)\times\Homo^n_{dR}\left( V_z;(\mathcal{O}_{V_z}^\vee,\nabla_x^\vee)\right)\ni ([\Gamma^\vee],[\psi])\mapsto \int_{\Gamma^\vee} \Phi_z^{-1}\psi\in\C$.

Now, we are ready to state the twisted period relation. Let us fix any nonsingular $z\in\C^N$ and consider four bases $\{[\phi_i]\}_{i=1}^r\subset\Homo^n_{dR}\left( V_z;(\mathcal{O}_{V_z},\nabla_x)\right)$, $\{[\psi_i]\}_{i=1}^r\subset\Homo_{dR}^n\left( V_z;(\mathcal{O}_{V_z}^\vee,\nabla_x^\vee)\right)$, $\{ [\gamma_i]\}_{i=1}^r\subset \Homo_n(V_z^{an};\mathcal{L}_z)$, and $\{ [\gamma_i^\vee]\}_{i=1}^r\subset\Homo_n(V_z^{an};\mathcal{L}_z^{\vee})$. Since $\langle\bullet,\bullet\rangle_h$ and $\langle\bullet,\bullet\rangle_{ch}$ are perfect pairings, intersection matrices $I_{ch}=(\langle \phi_i, \psi_j\rangle_{ch})_{i,j}$ and $I_h=(\langle \gamma_i,\gamma_j^\vee\rangle_h)_{i,j}$ are both invertible. On the other hand, period matrices $P=(\int_{\gamma_j} \Phi_z\phi_i)_{i,j}$ and $P^\vee=(\int_{\gamma^\vee_j}\Phi_z^{-1}\psi_i)_{i,j}$ are also well-defined and invertible. The twisted period relation \cite[Theorem 2]{ChoMatsumoto} is a transcendental analogue of the Riemann-Hodge bilinear relation:
\begin{equation}
I_{ch}=P{}^t I_h^{-1}{}^tP^\vee.
\end{equation}
If we set $I_h^{-1}=(C^{ij})_{i,j}$, we have 
\begin{equation}\label{GeneralQR}
\langle \phi,\psi\rangle_{ch}=\sum_{i,j}\left( \int_{\gamma_i}\Phi_z\phi\right)C^{ji}\left( \int_{\gamma_j^\vee}\Phi_z^{-1}\psi\right)
\end{equation}
for any cohomology classes $[\phi]\in\Homo^n_{dR}\left( V_z;(\mathcal{O}_{V_z},\nabla_x)\right)$ and $[\psi]\in\Homo_{dR}^n\left( V_z;(\mathcal{O}_{V_z}^\vee,\nabla_x^\vee)\right)$.

\subsection{Main results}\label{sec:2.3}

Taking a regular triangulation $T$, we consider bases $\{ \left[\Gamma_{\s,\tilde{\bf k}}\right]\}_{[\tilde{\bf k}]\in\Z^{\s\times 1}/\Z \transp A_{\s}}$ of $\Homo_n(V_z^{an},\mathcal{L}_z)$ and $\{ \left[\check{\Gamma}_{\s,\tilde{\bf k}}\right]\}_{[\tilde{\bf k}]\in\Z^{\s\times 1}/\Z \transp A_{\s}}$ of $\Homo_n(V_z^{an};\mathcal{L}_z^{\vee})$. We claim that the homology intersection matrix $I_h$ with respect to these bases has a closed form. In view of \cite[Proposition 7.1]{MatsubaraEuler}, we see that $I_h$ is a direct sum $I_h=\oplus_{\s\in T}I_{\s,h}$ where the summand $I_{\s,h}$ is defined by $I_{\s,h}=(\langle\Gamma_{\s,\tilde{\bf k}_1},\check{\Gamma}_{\s,\tilde{\bf k}_2}\rangle_h)_{[\tilde{\bf k}_1],[\tilde{\bf k}_2]\in\Z^{\s\times 1}/\Z \transp A_{\s}}$. For any vector $v\in\C^{(n+k)\times 1}$, the $i$($\in\s$)-th entry of $A_\s^{-1}v\in\C^{\s\times 1}$ is denoted by the symbol $p_{\s i}(v)$. For any $(n+k)$-simplex $\s$ and $[{\bf k}]\in\Z^{(n+k)\times 1}/\Z A_\s$, we set
\begin{equation}\label{eq:Eigen1}
H_{A_\s,{\bf k}}(\delta)=\prod_{l:|\s^{(l)}|>1}\left(1-e^{2\pi\ii\gamma_l}\right)\prod_{i\in\s^{(l)}}\left( 1-e^{-2\pi\ii p_{\s i}(\delta+{\bf k})}\right).
\end{equation}

%Let $\{{\bf e}_i\}_{i\in\s}$ be the standard basis of $\Z^{\s\times 1}$. 

\noindent
We have the following theorem whose proof will be given in \S\ref{section:Lifts}.
\begin{thm}\label{thm:SigmaIntersectionMatrix1}
\begin{equation}\label{SigmaIntersectionMatrix1}
I_{\s,h}=\diag\left( e^{2\pi\ii{}^t\tilde{\bf k}A_\s^{-1}\delta}\right)_{[\tilde{\bf k}]\in\Z^{\s\times 1}/\Z \transp A_{\s}}U_\s\diag\left( H_{A_\s,{\bf k}}(\delta)\right)_{[{\bf k}]\in\Z^{(n+k)\times 1}/\Z A_\s}U_\s^*\diag\left( e^{-2\pi\ii{}^t\tilde{\bf k}A_\s^{-1}\delta}\right)_{[\tilde{\bf k}]\in\Z^{\s\times 1}/\Z \transp A_{\s}}.
\end{equation}
\end{thm}

For any pair of complex numbers $\alpha,\beta$ such that $\alpha+\beta\notin\{ 0,-1,-2,\dots\}$, we set $(\alpha)_{\beta}=\frac{\Gamma(\alpha+\beta)}{\Gamma(\alpha)}$. In general, for any pair of complex vectors $\alpha=(\alpha_1,\dots,\alpha_n)$, $\beta=(\beta_1,\dots,\beta_n)$ such that $\alpha_i+\beta_i\notin\{ 0,-1,-2,\dots\}$, we set $(\alpha)_{\beta}=(\alpha_1)_{\beta_1}\cdots (\alpha_n)_{\beta_n}$. For any vectors ${\bf a}={}^t(a_1,\dots,a_n)\in\Z^{n\times 1}$ and ${\bf b}={}^t(b_1,\dots,b_k)\in\Z^{k\times 1}$, we set $x^{\bf a}h^{\bf b}=x_1^{a_1}\cdots x_n^{a_n}h_1^{b_1}\cdots h_k^{b_k}$. We are in a position to state and prove the main
\begin{thm}\label{TheQuadraticRelation}
Take any regular triangulation $T$, ${\bf a}, {\bf a}^\prime\in\Z^{n\times 1}$ and ${\bf b}, {\bf b}^\prime\in\Z^{k\times 1}$. Suppose that $\gamma_l\notin\Z$ and $\delta$, $\left(\substack{\gamma-{\bf b}\\ c+{\bf a}}\right)$ and $\left(\substack{-\gamma-{\bf b^\prime}\\ -c+{\bf a^\prime}}\right)$ are very generic with respect to $T$. Then, the following identity holds:
\begin{align}
&\frac{\langle x^{\bf a}h^{\bf b}\frac{dx}{x},x^{\bf a^\prime}h^{\bf b^\prime}\frac{dx}{x} \rangle_{ch}}{(2\pi\ii)^n}\nonumber\\
=&(-1)^{|{\bf b}|+|{\bf b^\prime}|}(\gamma-{\bf b})_{\bf b}(-\gamma-{\bf b^\prime})_{\bf b^\prime}\gamma_1\cdots\gamma_k\times\nonumber\\
&\sum_{\s\in T}\sum_{[A_{\bs}{\bf k}]\in\Z^{(n+k)\times 1}/\Z A_\s}\frac{(-1)^{|{\bf k}|}}{|\det A_\s|}\frac{\pi^{n+k}}{\sin\pi A_\s^{-1}(\delta+A_{\bs}{\bf k})}
\varphi_{\s,{\bf k}_{\bf a,b}}\Big( z;\substack{\gamma-{\bf b}\\ c+{\bf a}}\Big)
\varphi_{\s,{\bf k}^\prime_{\bf a^\prime,b^\prime}}\Big( z;\substack{-\gamma-{\bf b^\prime}\\ -c+{\bf a^\prime}}\Big).\label{eqn:CIF}
\end{align}
Here, ${\bf k}_{\bf a,b}$ (resp. ${\bf k}^\prime_{\bf a^\prime,b^\prime}$) denotes the element such that 
$[A_{\bs}{\bf k}+
\begin{pmatrix}
{\bf b}\\
-{\bf a}
\end{pmatrix}
]=
[A_{\bs}{\bf k}_{\bf a,b}]$ (resp. $[-A_{\bs}{\bf k}+
\begin{pmatrix}
{\bf b^\prime}\\
-{\bf a^\prime}
\end{pmatrix}
]=
[A_{\bs}{\bf k}^\prime_{\bf a^\prime,b^\prime}]$) holds in the group $\Z^{(n+k)\times 1}/\Z A_\s$.
\end{thm}

\begin{proof}
We compute the left-hand side of (\ref{eqn:CIF}) by means of (\ref{GeneralQR}). Taking a complete system of representatives $\{[\tilde{\bf k}(1)],\dots,[\tilde{\bf k}(r_\s)]\}$ of $\Z^{\s\times 1}/\Z\transp A_\s$, we set 
\begin{equation}
P_\s=\transp{\left( \int_{\Gamma_{\s,\tilde{\bf k}(1)}}\Phi_z x^{\bf a}h^{\bf b}\frac{dx}{x},\dots,\int_{\Gamma_{\s,\tilde{\bf k}(r_\s)}}\Phi_z x^{\bf a}h^{\bf b}\frac{dx}{x}\right)}
\end{equation}
and
\begin{equation}
P^\vee_\s=\transp{\left( \int_{\Gamma^\vee_{\s,\tilde{\bf k}(1)}}\Phi^{-1} x^{\bf a^\prime}h^{\bf b^\prime}\frac{dx}{x},\dots,\int_{\Gamma^\vee_{\s,\tilde{\bf k}(r_\s)}}\Phi^{-1} x^{\bf a^\prime}h^{\bf b^\prime}\frac{dx}{x}\right)}.
\end{equation}

\noindent
The formula (\ref{GeneralQR}) in our setting takes the form
\begin{equation}\label{OurQR}
\langle x^{\bf a}h^{\bf b}\frac{dx}{x},x^{\bf a^\prime}h^{\bf b^\prime}\frac{dx}{x} \rangle_{ch}
=
\sum_{\s\in T}\transp P_\s\transp I_{\s,h}^{-1}P^\vee_\s.
\end{equation}

\noindent
On the other hand, if we take a complete set of representatives $\{ [A_{\bs}{\bf k}(1)],\dots,[A_{\bs}{\bf k}(r_\s)]\}$ of $\Z^{(n+k)\times 1}/\Z A_\s$, we set $\Phi_{\s,1}=\transp{\left(\varphi_{\s,{\bf k}(1)}\left( z;\substack{\gamma-{\bf b}\\ c+{\bf a}}\right),\dots,\varphi_{\s,{\bf k}(r_\s)}\left( z;\substack{\gamma-{\bf b}\\ c+{\bf a}}\right)\right)}$. By Theorem \ref{thm:fundamentalthm}, we have

\begin{align}
&\frac{1}{(2\pi\ii)^{n+k}}P_\s\nonumber\\
=&
\sqrt{r_\s}C_\s(\gamma-{\bf b})\diag\left( 
\exp\left\{
2\pi\ii{}^t\tilde{\bf k}(1)A_\s^{-1}
\begin{pmatrix}
\gamma-{\bf b}\\
c+{\bf a}
\end{pmatrix}
\right\}
,\dots,
\exp\left\{
2\pi\ii{}^t\tilde{\bf k}(r_\s)A_\s^{-1}
\begin{pmatrix}
\gamma-{\bf b}\\
c+{\bf a}
\end{pmatrix}
\right\}\right)
U_{\s}
\Phi_{\s,1}.\label{Integral1}
\end{align}

\noindent
In the same way, setting 
$
\Phi_{\s,2}=
\transp{\left(
\varphi_{\s,{\bf k}(1)}\left( z;\substack{-\gamma-{\bf b^\prime}\\ -c+{\bf a^\prime}}\right)
,\dots,
\varphi_{\s,{\bf k}(r_\s)}\left( z;\substack{-\gamma-{\bf b^\prime}\\ -c+{\bf a^\prime}}\right)
\right)}$, we have
\begin{align}
&\frac{1}{(2\pi\ii)^{n+k}}P_\s^\vee\nonumber\\
=&
\sqrt{r_\s}
C_\s(-\gamma-{\bf b^\prime})\diag\left( 
\exp\left\{
2\pi\ii{}^t\tilde{\bf k}(1)A_\s^{-1}
\begin{pmatrix}
-\gamma-{\bf b^\prime}\\
-c+{\bf a^\prime}
\end{pmatrix}
\right\}
,\dots,
\exp\left\{
2\pi\ii{}^t\tilde{\bf k}(r_\s)A_\s^{-1}
\begin{pmatrix}
-\gamma-{\bf b^\prime}\\
-c+{\bf a^\prime}
\end{pmatrix}
\right\}
\right)
U_{\s}
\Phi_{\s,2}.\label{Integral2}
\end{align}

\noindent
Substituting (\ref{Integral1}), (\ref{Integral2}), and
\begin{equation}
I_{\s,h}^{-1}=\diag\left( e^{2\pi\ii{}^t\tilde{\bf k}A_\s^{-1}\delta}\right)_{\tilde{\bf k}}U_{\s}\diag\left( H_{A_\s,A_{\bs}{\bf k}}(\delta)^{-1}\right)_{\bf k}U^*_{\s}\diag\left( e^{-2\pi\ii{}^t\tilde{\bf k}A_\s^{-1}\delta}\right)_{\tilde{\bf k}}
\end{equation}
to (\ref{OurQR}), we obtain (\ref{eqn:CIF}). Note that $U_\s^{-1}=U_\s^*$.
\end{proof}

\begin{rem}
While the signature $(-1)^{|{\bf k}|}$ may depend on the choice of the representative $[A_{\bs}{\bf k}]\in\Z^{(n+k)\times 1}/\Z A_\s$, the fraction 
$\frac{(-1)^{|{\bf k}|}}{\sin\pi A_\s^{-1}(\delta+A_{\bs}{\bf k})}$ does not. To show this, let us take two integer vectors ${\bf k}_1,{\bf k}_2$ such that $A_{\bs}({\bf k}_1-{\bf k}_2)=A_\s{\bf k}_\s$ for some ${\bf k}_\s$. Since $|{\bf k}_1-{\bf k}_2|=|{\bf k}_\s|$, we have
\begin{equation}
\frac{(-1)^{|{\bf k}_1|}}{\sin\pi A_\s^{-1}(\delta+A_{\bs}{\bf k}_1)}=\frac{(-1)^{|{\bf k}_2|}(-1)^{|{\bf k}_\s|}}{\sin\pi \{A_\s^{-1}(\delta+A_{\bs}{\bf k}_2)+{\bf k}_\s\}}=\frac{(-1)^{|{\bf k}_2|}}{\sin\pi A_\s^{-1}(\delta+A_{\bs}{\bf k}_2)}.
\end{equation}

\end{rem}

\begin{exa}
We consider a $3\times 5$ matrix 
\begin{equation}
A=
\begin{pmatrix}
1&1&1&1&1\\
0&1&0&2&0\\
0&0&1&0&2
\end{pmatrix}
\end{equation}
which corresponds to the Euler integral of the form
\begin{equation}
\int_\Gamma (z_1+z_2x+z_3y+z_4x^2+z_5y^2)^{-\gamma}x^{c_1}y^{c_2}\frac{dxdy}{xy}.
\end{equation}

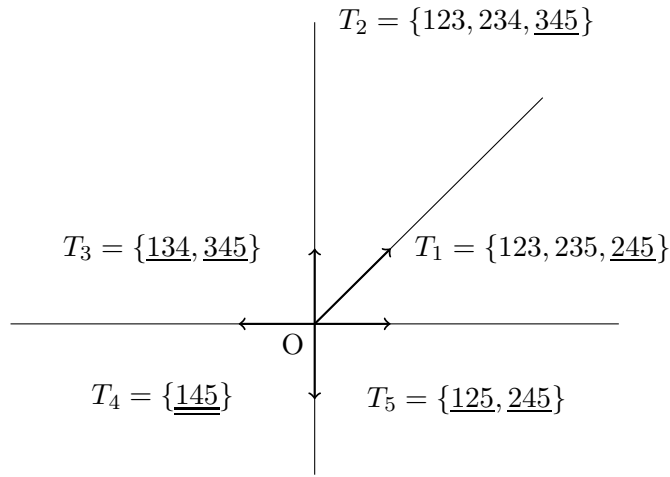
\begin{figure}[H]
\begin{center}
\begin{tikzpicture}%[domain=0:3, samples=100, very thick] % 定義域、点の数、線幅
\draw (0,0) node[below left]{O}; % 原点、0でも、above, below, left, rightで位置指定
  % 位置指定はanchor=north, south, east, westでも可能
\draw[thick, ->] (0,0)--(1,0);
\draw[thick, ->] (0,0)--(0,1);
\draw[thick, ->] (0,0)--(1,1);
\draw[thick, ->] (0,0)--(-1,0);
\draw[thick, ->] (0,0)--(0,-1);
\draw[-] (0,0)--(4,0);
\draw[-] (0,0)--(0,4);
\draw[-] (0,0)--(3,3);
\draw[-] (0,0)--(-4,0);
\draw[-] (0,0)--(0,-2);
\node at (3,1){$T_1=\{ 123,235,\underline{245}\}$};
\node at (2,4){$T_2=\{ 123,234,\underline{345}\}$};
\node at (-2,1){$T_3=\{ \underline{134},\underline{345}\}$};
\node at (-2,-1){$T_4=\{ \underline{\underline{145}}\}$};
\node at (2,-1){$T_5=\{ \underline{125},\underline{245}\}$};
\end{tikzpicture}
\caption{Projected image of the secondary fan in $\R^{1\times 2}\simeq L_A^\vee$}%\label{SecondaryFanOfG1}
\end{center}
\end{figure}

\noindent
Using the package {\rm ns\_twistedlog.rr}\footnote{http://www.math.kobe-u.ac.jp/OpenXM/Current/doc/asir-contrib/en/ns\_twistedlog-html/ns\_twistedlog-en.html} of Risa/Asir (\cite{risa-asir}), we see that the basis of the cohomology groups $\Homo^2_{dR}\left( V_z;(\mathcal{O}_{V_z},\nabla_x)\right)$ and $\Homo^2_{dR}\left( V_z;(\mathcal{O}_{V_z}^\vee,\nabla_x^\vee)\right)$ is given by $\{ [\frac{dx\wedge dy}{xy}],[\frac{dx\wedge dy}{y}],[\frac{ydx\wedge dy}{x}],[dx\wedge dy]\}$. 
The intersection number $\langle \frac{dx\wedge dy}{xy},\frac{dx\wedge dy}{xy}\rangle_{ch}$ can easily be evaluated by the method of \cite{MatsumotoIntersection} and is equal to 
\begin{equation}
(2\pi\ii)^2\left(\frac{1}{c_1c_2}+\frac{1}{(2\gamma-c_1-c_2)c_2}+\frac{1}{c_1(2\gamma-c_1-c_2)}\right)=(2\pi\ii)^2\frac{2\gamma}{(2\gamma-c_1-c_2)c_1c_2}.
\end{equation}

Let us take the regular triangulation $T_4$. Theorem \ref{TheQuadraticRelation} for the cohomology intersection number $\langle \frac{dx\wedge dy}{xy},\frac{dx\wedge dy}{xy}\rangle_{ch}$ is 
\begin{align}
  &\frac{2\gamma}{(2\gamma-c_1-c_2)c_1c_2}\nonumber\\
=&\frac{\gamma\pi^3}{4}\left\{
\frac{1}{\sin\pi(\frac{2\gamma-c_1-c_2}{2})\sin\pi\frac{c_1}{2}\sin\pi\frac{c_2}{2}}
\varphi_{145,\scalebox{0.6}{$\begin{pmatrix} 0\\ 0\end{pmatrix}$}}(z;\delta)\varphi_{145,\scalebox{0.6}{$\begin{pmatrix} 0\\ 0\end{pmatrix}$}}(z;-\delta)\right.\nonumber\\
&-
\frac{1}{\sin\pi(\frac{2\gamma-c_1-c_2+1}{2})\sin\pi\frac{(c_1+1)}{2}\sin\pi\frac{c_2}{2}}
\varphi_{145,\scalebox{0.6}{$\begin{pmatrix} 1\\ 0\end{pmatrix}$}}(z;\delta)\varphi_{145,\scalebox{0.6}{$\begin{pmatrix} 1\\ 0\end{pmatrix}$}}(z;-\delta)\nonumber\\
&-
\frac{1}{\sin\pi(\frac{2\gamma-c_1-c_2+1}{2})\sin\pi\frac{c_1}{2}\sin\pi\frac{(c_2+1)}{2}}
\varphi_{145,\scalebox{0.6}{$\begin{pmatrix} 0\\ 1\end{pmatrix}$}}(z;\delta)\varphi_{145,\scalebox{0.6}{$\begin{pmatrix} 0\\ 1\end{pmatrix}$}}(z;-\delta)\nonumber\\
&\left.+
\frac{1}{\sin\pi(\frac{2\gamma-c_1-c_2+2}{2})\sin\pi\frac{(c_1+1)}{2}\sin\pi\frac{(c_2+1)}{2}}
\varphi_{145,\scalebox{0.6}{$\begin{pmatrix} 1\\ 1\end{pmatrix}$}}(z;\delta)\varphi_{145,\scalebox{0.6}{$\begin{pmatrix} 1\\ 1\end{pmatrix}$}}(z;-\delta)\right\}.
\end{align}

On the other hand, if we take the regular triangulation $T_1$, we obtain 
\begin{align}
  &\frac{2\gamma}{(2\gamma-c_1-c_2)c_1c_2}\nonumber\\
=&\gamma\pi^3\left\{
\frac{1}{\sin\pi(\gamma-c_1-c_2)\sin\pi c_1\sin\pi c_2}
\varphi_{123,\scalebox{0.6}{$\begin{pmatrix} 0\\ 0\end{pmatrix}$}}(z;\delta)\varphi_{123,\scalebox{0.6}{$\begin{pmatrix} 0\\ 0\end{pmatrix}$}}(z;-\delta)\right.\nonumber\\
&+
\frac{1}{\sin\pi c_1\sin\pi (2\gamma-2c_1-c_2)\sin\pi (c_1+c_2-\gamma)}
\varphi_{235,\scalebox{0.6}{$\begin{pmatrix} 0\\ 0\end{pmatrix}$}}(z;\delta)\varphi_{235,\scalebox{0.6}{$\begin{pmatrix} 0\\ 0\end{pmatrix}$}}(z;-\delta)\nonumber\\
&+
\frac{1}{2\sin\pi(2\gamma-c_1-c_2)\sin\pi \frac{2c_1+c_2-2\gamma}{2}\sin\pi \frac{c_2}{2}}
\varphi_{245,\scalebox{0.6}{$\begin{pmatrix} 0\\ 0\end{pmatrix}$}}(z;\delta)\varphi_{245,\scalebox{0.6}{$\begin{pmatrix} 0\\ 0\end{pmatrix}$}}(z;-\delta)\nonumber\\
&-
\left.
\frac{1}{2\sin\pi(2\gamma-c_1-c_2+1)\sin\pi \frac{2c_1+c_2-2\gamma-1}{2}\sin\pi \frac{(c_2+1)}{2}}
\varphi_{245,\scalebox{0.6}{$\begin{pmatrix} 0\\ 1\end{pmatrix}$}}(z;\delta)\varphi_{245,\scalebox{0.6}{$\begin{pmatrix} 0\\ 1\end{pmatrix}$}}(z;-\delta)\right\}.
\end{align}

The cohomology intersection number $\langle \frac{dx\wedge dy}{z_1+z_2x+z_3y+z_4x^2+z_5y^2}, \frac{dx\wedge dy}{xy}\rangle_{ch}$ is zero. If we take the regular triangulation $T_4$, we obtain
\begin{align}
0=&
\frac{1}{\sin\pi(\frac{2\gamma-c_1-c_2}{2})\sin\pi\frac{c_1}{2}\sin\pi\frac{c_2}{2}}
\varphi_{145,\scalebox{0.6}{$\begin{pmatrix} 1\\ 1\end{pmatrix}$}}(z;\delta+{\bf 1})\varphi_{145,\scalebox{0.6}{$\begin{pmatrix} 0\\ 0\end{pmatrix}$}}(z;-\delta)\nonumber\\
&-
\frac{1}{\sin\pi(\frac{2\gamma-c_1-c_2+1}{2})\sin\pi\frac{(c_1+1)}{2}\sin\pi\frac{c_2}{2}}
\varphi_{145,\scalebox{0.6}{$\begin{pmatrix} 0\\ 1\end{pmatrix}$}}(z;\delta+{\bf 1})\varphi_{145,\scalebox{0.6}{$\begin{pmatrix} 1\\ 0\end{pmatrix}$}}(z;-\delta)\nonumber\\
&-
\frac{1}{\sin\pi(\frac{2\gamma-c_1-c_2+1}{2})\sin\pi\frac{c_1}{2}\sin\pi\frac{(c_2+1)}{2}}
\varphi_{145,\scalebox{0.6}{$\begin{pmatrix} 1\\ 0\end{pmatrix}$}}(z;\delta+{\bf 1})\varphi_{145,\scalebox{0.6}{$\begin{pmatrix} 0\\ 1\end{pmatrix}$}}(z;-\delta)\nonumber\\
&+
\frac{1}{\sin\pi(\frac{2\gamma-c_1-c_2+2}{2})\sin\pi\frac{(c_1+1)}{2}\sin\pi\frac{(c_2+1)}{2}}
\varphi_{145,\scalebox{0.6}{$\begin{pmatrix} 0\\ 0\end{pmatrix}$}}(z;\delta+{\bf 1})\varphi_{145,\scalebox{0.6}{$\begin{pmatrix} 1\\ 1\end{pmatrix}$}}(z;-\delta).
\end{align}

\noindent
Here, we have put
$
\delta+{\bf 1}=
\begin{pmatrix}
\gamma+1\\
c_1+1\\
c_2+1
\end{pmatrix}.
$
\end{exa}

\subsection{Some consequences of Theorem \ref{TheQuadraticRelation}}
We discuss some consequences on the cohomology intersection form derived from Theorem \ref{TheQuadraticRelation}. Namely, we explain that Theorem \ref{TheQuadraticRelation} can be used to generalize \cite[Theorem 3.2 and Theorem 3.6]{MatsubaraTakayama}. 

Let us recall notation of \cite[\S2]{MatsubaraTakayama}. For any field extension $K\subset\C$ of $\Q$, we set $\Gm(K)^n={\rm Spec}\left( K[x_1^\pm,\dots,x_n^\pm]\right)$. We also put $\A^N={\rm Spec}\ \C[z_j^{(l)}]$. We write $E_A(z)$ for the product of principal $A_{\Gamma}$-discriminants for any face $\Gamma$ of ${\rm New}(A)$ (\cite[Chapter 9, Definition 1.2]{GKZbook}). Here, ${\rm New}(A)$ is the convex hull of column vectors of $A$. The complement $U=\A^N\setminus \{ E_A(z)=0\} $ is called the Newton non-degenerate locus (\cite[LEMMA 3.3]{Adolphson}). Since $E_A(z)$ is a  polynomial with coefficients in $\Q$, we can also consider a reduced scheme $U(K)$ defined over $K$ whose base change to $\C$ is isomorphic to $U$. Let us denote by $D_{U(K)}$ the ring of differential operators on $U(K)$. Any element $P$ of $D_{U(K)}$ is a finite sum $P=\frac{1}{E_A(z)^l}\sum_\alpha a_\alpha(z)\partial^\alpha$ where $a_\alpha(z)$ is a polynomial with coefficients in $K$ and $l$ is an integer. We define a reduced divisor $D$ of $\Gm(K)^n\times U(K)$ by $D=\bigcup_{l=1}^k\{ h_{l,z^{(l)}}(x)=0\}$. We define the symbol $\int^0_{\pi}(\mathcal{O}_{X(K)},\nabla)\restriction_{U(K)}$ by
\begin{equation}\label{eqn:GM}
\mathbb{H}^n\left( (\Gm(K))^n_x\times U(K);\left(\cdots\overset{\nabla_x}{\rightarrow}
\Omega^\bullet_{\Gm(K)^n\times U(K)/U(K)}\left(*D\right)\overset{\nabla_x}{\rightarrow}\right)\cdots\right).
\end{equation}
Here, $\Omega^p_{(\Gm(K))^n_x\times U(K)/U(K)}(*D)$ is the sheaf of relative differential $p$-forms with poles along $D$ with respect to the projection $\Gm(K)^n\times U(K)\rightarrow U(K)$.  We can also define the dual object $\int^0_{\pi}(\mathcal{O}_{X(K)},\nabla^\vee)\restriction_{U(K)}$ by replacing $\nabla_x$ by $\nabla_x^\vee$ in (\ref{eqn:GM}). Moreover, for any $[\phi]\in \int^0_{\pi}(\mathcal{O}_{X(K)},\nabla)\restriction_{U(K)}$, we set 
\begin{equation}\label{GMAction}
\nabla^{GM}\phi=d_z\phi-\displaystyle\sum_{j,l}\gamma_l\frac{x^{{\bf a}^{(l)}(j)}}{h_{l,z^{(l)}}(x)}dz_j^{(l)}\wedge\phi.
\end{equation}
Then, $\left(\int^0_{\pi}(\mathcal{O}_{X(K)},\nabla)\restriction_{U(K)},\nabla^{GM}\right)$ is an integrable connection on $U(K)$. We can also define the dual connection $\nabla^{\vee GM}$ by replacing $\gamma_l$ by $-\gamma_l$ in (\ref{GMAction}). When $\delta$ is non-resonant in the sense of \cite[2.9]{GKZEuler} and $\gamma_l\notin\Z$, there is a perfect bilinear pairing
\begin{equation}\label{eqn:RCIF}
\int^0_{\pi}(\mathcal{O}_{X(\C)},\nabla)\restriction_{U}\times\int^0_{\pi}(\mathcal{O}_{X(\C)},\nabla^\vee)\restriction_{U}\rightarrow\mathcal{O}_{U}
\end{equation}
whose stalk at any point $z\in U$ is the cohomology intersection form (\ref{eqn:ACIF}). By abuse of notation, the pairing of (\ref{eqn:RCIF}) is denoted by $\langle\bullet,\bullet\rangle_{ch}$. We define a field extension $\Q(\delta)$ of $\Q$ by $\Q(\delta)=\Q(\gamma_1,\dots,\gamma_k,c_1,\dots,c_n)$. The following statement is a generalization of  \cite[Theorem 3.6]{MatsubaraTakayama}.

\begin{thm}\label{thm:Coefficients}
Suppose that $\delta$ is non-resonant and $\gamma_l\notin\Z$ for any $l=1,\dots,k$. Then, the normalized cohomology intersection pairing $B=\frac{\langle\bullet,\bullet\rangle_{ch}}{(2\pi\ii)^n}$ defines a perfect bilinear pairing 
\begin{equation}
B:\ \int^0_{\pi}(\mathcal{O}_{X(\Q(\delta))},\nabla)\restriction_{U(\Q(\delta))}\times\int^0_{\pi}(\mathcal{O}_{X(\Q(\delta))},\nabla^\vee)\restriction_{U(\Q(\delta))}\rightarrow\mathcal{O}_{U(\Q(\delta))}.
\end{equation}
Moreover, for any local sections $\phi$ of $\int^0_{\pi}(\mathcal{O}_{X(\Q(\delta))},\nabla)\restriction_{U(\Q(\delta))}$ and $\psi$ of $\int^0_{\pi}(\mathcal{O}_{X(\Q(\delta))},\nabla^\vee)\restriction_{U(\Q(\delta))}$, the equality
\begin{equation}\label{Compatibility}
dB(\phi,\psi)=B(\nabla^{GM}\phi,\psi)+B(\phi,\nabla^{\vee GM}\psi)
\end{equation}
holds.
\end{thm}

\noindent
For the proof of Theorem \ref{thm:Coefficients}, we only need to replace \cite[Theorem 8.1]{MatsubaraEuler} in the proof of \cite[Theorem 3.5, 3.6]{MatsubaraTakayama} by Theorem \ref{TheQuadraticRelation} of this paper. In the same way, we can also extend the algorithm of \cite{MatsubaraTakayama}.

\begin{thm}
Given a matrix $A$ as in (\ref{CayleyConfigu}).
When parameters are non-resonant, $\gamma_l\notin\Z$ and moreover the set of series solutions $\Phi_T$ for some regular triangulation $T$ is linearly independent, 
the intersection matrix of
the twisted cohomology group of the GKZ system associated to the matrix $A$
can be algorithmically determined.
\end{thm}

\subsection{A conjecture of F. Beukers and C. Verschoor}\label{subsec:BV}

Let us consider a $d\times n$ ($d<n$) integer matrix $A$ and a parameter vector $\delta=\transp{(\delta_1,\dots,\delta_d)}\in\R^{d\times 1}$. We assume that $A$ is homogeneous, i.e., all the column vectors of $A$ lie on a hyperplane that does not go through the origin. We fix a base point $z\in U$. The purpose of this subsection is to show the following theorem conjectured by F. Beukers and C. Verschoor, which turns out to be a corollary of Theorem \ref{thm:SigmaIntersectionMatrix1}.
\begin{thm}\label{thm:BV}
Assume that $\delta$ is very generic with respect to a regular triangulation $T$. Then, the signature of the monodromy invariant hermitian form of $\sol_{M_A(\delta),z}$ is given by those of
\begin{equation}\label{eqn:BV}
\sin\pi A^{-1}_{\s}(\delta+{\bf k})=\prod_{i\in\s}\sin\pi p_{\s i}(\delta+{\bf k})
\end{equation}
where $\s$ runs over $T$ and $[{\bf k}]$ runs over a complete system of representatives of a finite abelian group $\Z^{d\times 1}/\Z A_{\s}$.
\end{thm}

Although it seems it is well-known that the solution space of GKZ system $M_A(\delta)$ with non-resonant real parameter $\delta$ admits a monodromy invariant hermitian form, we give a proof of this basic fact. Since $A$ is homogeneous, we may assume that $A$ is of the form (\ref{CayleyConfigu}) with $k=1$ after a suitable change of a basis of $\Z^{d\times 1}$. Therefore, we change the notation and keep using that of preceding subsections. For example, we write $\delta=\transp{(\gamma,c_1,\dots,c_n)}$ with $\gamma\in\R$, $c\in\R^{n\times 1}$ and $n=d-1$. We assume that $\delta$ is non-resonant and $\gamma\notin\Z$. We write $\pi:V:=(\Gm)^n\times U\setminus D\rightarrow U$ for the composition of the inclusion $V\hookrightarrow (\Gm)^n\times U$ and the canonical projection $(\Gm)^n\times U\rightarrow U$. We define an integrable connection $\nabla$ by $\nabla=d-\gamma\frac{dh_1}{h_1}\wedge+\sum_{i=1}^nc_i\frac{dx_i}{x_i}\wedge:\mathcal{O}_{V}\rightarrow\Omega^1_{V}$. Here, $d$ stands for the exterior derivative on $V$. Let $\mathcal{L}$ be the dual local system of the sheaf of flat sections of $\nabla^{an}$. In order to simplify the notation, let us put $K:=\Q(e^{2\pi\ii\gamma},e^{2\pi\ii c_1},\dots,e^{2\pi\ii c_n})$. Note that $K$ is closed under complex conjugate. Then, it is easy to see that $\mathcal{L}$ is defined over $K$ and its restriction to any fiber $\pi^{-1}(z)$ is identical to $\mathcal{L}_z$. By fixing a base point $(x,z)\in V$ and a basis $\Phi_{(x,z)}\in \mathcal{L}_{(x,z)}$, we identify the local system $\mathcal{L}$ with its associated $\pi_1(V^{an},(x,z))$-representation on the vector space $\mathcal{L}_{(x,z)}$. We write $\bar{\mathcal{L}}$ for the complex conjugate $K$-local system of $\mathcal{L}$. Note that we have a natural $\underline{\Q}_{V^{an}}$-linear isomorphism ${\rm conj}:\mathcal{L}\rightarrow\bar{\mathcal{L}}$. We define a $\pi_1(V^{an},(x,z))$-invariant unitary metric $(\bullet,\bullet):\mathcal{L}\times\bar{\mathcal{L}}\rightarrow\underline{K}_{V^{an}}$ by the formula $(a\Phi_{(x,z)},\overline{b\Phi_{(x,z)}}):=a\bar{b}$ for any $a,b\in K$. We naturally obtain an induced isomorphism ${\rm iden}:\bar{\mathcal{L}}\rightarrow\mathcal{L}^\vee$ defined by $\langle {\rm iden}({\rm conj}(v)),w\rangle=(w,{\rm conj}(v))$ for any $v,w\in\mathcal{L}$. Here, $\langle\bullet,\bullet\rangle:\mathcal{L}\times\mathcal{L}^\vee\rightarrow \underline{K}_{V^{an}}$ is the natural duality pairing. Applying the functor $R^n\pi^{an}_{!}$ to these isomorphisms, we obtain isomorphisms 
\begin{equation}
\xymatrix{
R^n\pi^{an}_{!}\mathcal{L}\ar[r]^-{\rm conj} &R^n\pi^{an}_{!}\bar{\mathcal{L}}\ar[r]^-{\rm iden} \ar[d]&R^n\pi^{an}_{!}\mathcal{L}^\vee\\
 &\overline{R^n\pi^{an}_{!}\mathcal{L}}&
}.
\end{equation}
Note that $R^n\pi^{an}_{!}\mathcal{L}$ is a local system on $U^{an}$. In particular, for any choice of a local frame $\{ \Gamma_1,\dots,\Gamma_r\}$ of $R^n\pi^{an}_{!}\mathcal{L}$, we can construct a local frame $\{ \bar{\Gamma}_1,\dots,\bar{\Gamma}_r\}$ of $R^n\pi^{an}_{!}\mathcal{L}^\vee$ so that $\bar{\Gamma}_i={\rm iden}\circ{\rm conj}(\Gamma_i)$. By using the isomorphism $\left( R^n\pi_{!}\mathcal{L}\right)_z\simeq \Homo_{c}^n\left( V_z^{an};\mathcal{L}_z\right)\simeq\Homo_n\left( V_z^{an};\mathcal{L}_z\right)$, we have the induced $\Q$-linear map ${\rm iden}\circ{\rm conj}:\Homo_n\left( V_z^{an};\mathcal{L}_z\right)\rightarrow\Homo_n\left( V_z^{an};\mathcal{L}^\vee_z\right)$. We set $\Gamma=(\Gamma_1,\dots,\Gamma_r)$. For any $\rho\in\pi_1(U^{an},z)$ in $U^{an}$, we define an $r\times r$ $K$-matrix $M^\Gamma_\rho$ by the identity $\rho_*\Gamma:=(\rho_*\Gamma_1,\dots,\rho_*\Gamma_r)=\Gamma M^\Gamma_\rho$ where $\rho_*\Gamma_i$ is the parallel transport of $\Gamma_i$ along $\rho$. Since ${\rm iden}\circ{\rm conj}$ is a morphism of sheaves, we have ${\rm iden}\circ{\rm conj}\left(\rho_*(\Gamma)\right)=\rho_*\left( {\rm iden}\circ{\rm conj}(\Gamma)\right)$. The left-hand side of this identity gives $\bar{\Gamma} \overline{M^\Gamma_\rho}$ while the right-hand side gives $\bar{\Gamma}M^{\bar{\Gamma}}_\rho$ where $\bar{\Gamma}=(\bar{\Gamma}_1,\dots,\bar{\Gamma}_r)$.  Thus, we obtain $\overline{M^\Gamma_\rho}=M^{\bar{\Gamma}}_\rho$.

In view of \cite[2.15 THEOREM]{GKZEuler} (see also \cite[Theorem 2.12 and Remark 2.13]{MatsubaraEuler}), we have a canonical isomorphism $R^d\pi_{!}\mathcal{L}^\vee\tilde{\rightarrow}R^d\pi_{*}\mathcal{L}^\vee$. We can define a $K$-sesquilinear pairing $(\bullet,\bullet)_h$ which is monodromy invariant and perfect by
\begin{equation}\label{eqn:2.43}
\begin{array}{cccc}
(\bullet,\bullet)_h :&\Homo_n\left( V_z^{an};\mathcal{L}_z\right)\times \Homo_n\left( V_z^{an};\mathcal{L}_z\right)&\rightarrow&K\\
&\rotatebox{90}{$\in$}&&\rotatebox{90}{$\in$}\\
&([\Gamma_1],[\Gamma_2])&\mapsto&\langle[\Gamma_1],{\rm iden}\circ{\rm conj}([\Gamma_2])\rangle_h.
\end{array}
\end{equation}
By definition, we can find a non-zero constant $C\in K\cap\R$ such that the intersection number $([\Gamma_1],[\Gamma_2])_h$ of two cycles $\Gamma_1=\sum a_j \Delta_j\otimes\Phi_z\restriction_{\Delta_j}$ and $\Gamma_2=\sum a_j^\prime \Delta_j^\prime\otimes\Phi_z\restriction_{\Delta_j^\prime}$ is given by
\begin{equation}
\sum_{j,k}a_j\overline{a_k^\prime}I_{loc}(\Delta_j,\Delta_k^\prime)\frac{\Phi_z\restriction_{\Delta_j}\bar{\Phi}_z\restriction_{\Delta_k^\prime}}{|\Phi_z|^2}
\end{equation}
multiplied by $C$. Here, $I_{loc}$ is the local intersection multiplicity. Note that the ratio $\frac{\Phi_z\restriction_{\Delta_j}\bar{\Phi}_z\restriction_{\Delta_k^\prime}}{|\Phi_z|^2}$ belongs to $K$. By this formula, we can observe that the formula $([\Gamma_1],[\Gamma_2])_h=(-1)^n\overline{([\Gamma_1],[\Gamma_2])_h}$ holds. We set $I_h=(([\Gamma_i],[\Gamma_j])_h)$. Since (\ref{eqn:2.43}) is a stalk of a morphism $(\bullet,\bullet)_h:R^n\pi^{an}_!\mathcal{L}\otimes_K\overline{R^n\pi^{an}_!\mathcal{L}}\rightarrow \underline{K}_{U^{an}}$, we have
\begin{equation}
I_h=\transp{M}^\Gamma_\rho I_hM_\rho^{\bar{\Gamma}}=\transp{M}^\Gamma_\rho I_h\overline{M_\rho^{\Gamma}}
\end{equation}
for any $\rho\in\pi_1(U^{an},z)$. In sum, we obtain a well-known

\begin{prop}
Suppose that $\delta$ is non-resonant, real and $\gamma\notin\Z$. Then, $(\ii)^n(\bullet,\bullet)_h$ defines a monodromy invariant non-degenerate Hermitian form on $\Homo_n\left( V_z^{an};\mathcal{L}_z\right)$ with values in $(\ii)^nK$.
\end{prop}

\begin{rem}
The irreducibility of $R^n\pi^{an}_!\mathcal{L}$ ensures the uniqueness of the monodromy invariant non-degenerate Hermitian form up to a non-zero scalar multiplication.
\end{rem}

\begin{rem}
As was pointed out in \cite[Remark 2.13]{MatsubaraEuler} the canonical morphism $\Homo_n(V_z^{an};\mathcal{L}_z)\rightarrow\Homo_n^{lf}(V_z^{an};\mathcal{L}_z)$ may fail to be an isomorphism when $\gamma\in\Z$ even if $\delta$ is non-resonant. 
\end{rem}

\noindent
\begin{proof}[Proof of Theorem \ref{thm:BV}]
Let us first discuss the case when $\delta$ is very generic and $\gamma\notin\Z$. By our construction of the twisted cycles $\Gamma_{\s,\tilde{\bf k}}$ and $\check{\Gamma}_{\s,\tilde{\bf k}}$ in \S\ref{section:Lifts} or in \cite[\S6]{MatsubaraEuler}\footnote{The construction of $\check{\Gamma}_{\s,\tilde{\bf k}}$ is to replace $\delta$ by $-\delta$ in the construction of $\Gamma_{\s,\tilde{\bf k}}$.}, it is easy to see that $[\check{\Gamma}_{\s,\tilde{\bf k}}]$ is equal to ${\rm iden}\circ{\rm conj}\left([\Gamma_{\s,\tilde{\bf k}}]\right)$ up to constant multiplication by $C\in K\cap\R$ which depends on neither $\s$ nor $\tilde{\bf k}$. We fix a regular triangulation $T$ and suppose that the parameter $\delta$ is very generic with respect to $T$. Let us fix a simplex $\s\in T$ and take complete systems of representatives $\{ [A_{\bs}{\bf k}(i)]\}_{i=1}^{r_\s}$ of $\Z^{(n+1)\times 1}/\Z A_\s$ and $\{ [\tilde{\bf k}(i)]\}_{i=1}^{r_\s}$ of $\Z^{\s\times 1}/\Z \transp{A_\s}$. We define row vectors $\Gamma_\s$ and $\Phi_\s$ by $\Gamma_\s:=([\Gamma_{\s,\tilde{\bf k}(1)}],\dots,[\Gamma_{\s,\tilde{\bf k}(r_\s)}])$ and $\Phi_\s:=(\varphi_{\s,{\bf k}(1)}(z;\delta),\dots,\varphi_{\s,{\bf k}(r_\s)}(z;\delta))$. Under the identification $\Homo_n(V_z^{an};\mathcal{L}_z)\simeq\sol_{M_A(\delta),z}$, we have a relation $\Gamma_\s=\sqrt{r_\s}C_\s(\gamma)\Phi_\s\transp{U_\s}\diag\Big( \exp\left\{
2\pi\ii\transp{
\tilde{\bf k}(i)
}
A_{\s}^{-1}\delta
\right\}\Big)_{i=1}^{r_\s}$. Note that $C_\s(\gamma)=\frac{\text{sgn} (A,\s)e^{-\pi\ii(1-\gamma)}}{\det A_\s \Gamma(\gamma)}$ in our setting. Now it is easy to see that the matrix $\Big(\left( \varphi_{\s,{\bf k}(i)}(z;\delta),\varphi_{\s,{\bf k}(j)}(z;\delta)\right)_h\Big)_{i,j}^{r_\s}$ is a diagonal matrix of which diagonal entries are 
\begin{equation}\label{eqn:DE}
-\pi\frac{\Gamma(\gamma)}{\Gamma(1-\gamma)}\times r_\s(2\ii)^{n+2}\sin\pi A_\s^{-1}(\delta+A_{\bs}{\bf k}(j))
\end{equation}
in view of Theorem \ref{thm:SigmaIntersectionMatrix1}. Note that we used an identity $\gamma+|{\bf k}|=\sum_{i=1}^{r_\s}p_{\s i}(\delta+A_{\bs}{\bf k})$. Thus, if we take $\Phi_T:=\cup_{\s\in T}\Phi_\s$ as a basis of $\Homo_n(V_z^{an};\mathcal{L}_z)$, the Gram matrix $I_{T,h}$ of $(\bullet,\bullet)_h$ is a diagonal matrix with diagonal entries (\ref{eqn:DE}), hence the theorem is proved when $\gamma\notin\Z$.

We discuss the case when $\delta$ is still very generic but $\gamma\in\Z$. Let $\rho\in\pi_1(U^{an},z)$ and $M_\rho(\delta)$ be the monodromy matrix with respect to the basis $\Phi_T$. We claim that $M_\rho(\delta)$ is a holomorphic function of $\delta$ when $\delta$ is very generic. Indeed, when $\delta=\delta_0$ and $\delta_0$ is still very generic, we can take a finite set $\{ [\partial^\alpha]\}_\alpha$ as a $\C(z)$-basis of $\C(z)\otimes_{\C[z]}M_A(\delta_0)$ (\cite[\S6.1]{Takayama}). Writing $\Phi_T=(\varphi_1(z;\delta),\dots,\varphi_r(z;\delta))$, we consider a matrix $\Psi(z;\delta):=\left( \partial^\alpha\varphi_i(z;\delta)\right)_{\alpha,i}$. By construction, we have $\det \Psi (z;\delta_0)\neq 0$. Since each entry of $\Psi(z;\delta)$ is a holomorphic function of $\delta$, the inverse matrix $\Psi(z;\delta)^{-1}$ is also a holomorphic function near $\delta=\delta_0$. Thus, the equality $M_\rho(\delta)=\Psi(z;\delta)^{-1}\rho_*\Psi(z;\delta)$ shows that $M_\rho(\delta)$ is holomorphic near $\delta=\delta_0$. 

Now we claim that the monodromy invariant hermitian form does exist when $\delta$ is very generic. By the monodromy invariance of the sesquilinear form $(\bullet,\bullet)_h$, we see that the formula ${}^tM_\rho(\delta)I_{T,h}(\delta)\overline{M_\rho(\delta)}=I_{T,h}(\delta)$ is true for any $\delta$ very generic and $\gamma\notin\Z$. Thus, we have ${}^tM_\rho(\delta)f(\delta)I_{T,h}(\delta)\overline{M_\rho(\delta)}=f(\delta)I_{T,h}(\delta)$ for any function $f(\delta)$. Taking into account that $M_\rho(\delta)$ depends holomorphically on $\delta$ and the formula (\ref{eqn:DE}), we can conclude that the monodromy invariant hermitian form exists when $\delta$ is very generic of which signature is given by those of (\ref{eqn:BV}). 
\end{proof}

\begin{exa}
We only mention an example which comes from the period integral of a family of K3 surfaces studied in \cite{MSY}. Let ${\bf e}_1,\dots,{\bf e}_5$ be the standard basis of $\Z^{5\times 1}$. We set $A=({\bf e}_i+{\bf e}_j)_{i=1,2,3\ j=4,5}$ and $\delta=\transp{(1/2,1/2,1/2,-1/2,-1/2)}$. The corresponding GKZ system $M_A(\delta)$ is referred to as Gra\ss man hypergeometric system $E(3,6)$. It is not difficult to see that the period integral discussed in \cite{MSY} corresponds to $E(3,6)$. Theorem \ref{thm:BV} shows that the signature of this system is $(4,2)$ which is in concordance with the signature of the transcendental lattice (\cite[Propositions 2.1.5, 2.2.1, 2.3.1]{MSY}). We computed the signature by employing a computer algebra system Risa/Asir (\cite{risa-asir}).
\end{exa}

\section{Lifts of a regularized cycle and their intersection numbers}\label{section:Lifts}
In this section, we give a proof of Theorem \ref{thm:SigmaIntersectionMatrix1}. 
%%%%%%%%%%%%%%%%%%%%%%%%%%%%%%%%%%%%%%%%%%%%%%%%%%%%%%%%%%%%%%%%%%%%%%%%%%%%%%%
% \input G_part.tex %
\subsection{Preparation}
First, we discuss the intersection numbers of twisted cycles when $A$ is a square matrix.
In this subsection, we collect some fundamental properties. In the following, for $x=(x_1,\dots,x_n)\in(\C^*)^n$, $B=({\bf b}(1)|\dots|{\bf b}(m))\in \Z^{n\times m}$,
and ${\bf b}=\transp (b_1,\dots ,b_n)\in \C^{n\times 1}$, 
we set $x^B:=(x^{{\bf b}(1)},\dots,x^{{\bf b}(m)})$ and $e^{2\pi\ii {\bf b}}x :=(e^{2\pi\ii b_1}x_1,\dots,e^{2\pi\ii b_n}x_n)$. By abuse of notation, we often write ${\bf k}\in\Z^{n\times 1}/\Z B$ for the relation $[{\bf k}]\in\Z^{n\times 1}/\Z B$.
 %%% If you don't use it, please delete this { }. 

Let $M=({\bf m}(1)|\cdots|{\bf m}(n))$ be an $n\times n$ integer matrix whose determinant is not zero. 
We set 
$\tX =\Big\{ (x_1,\dots ,x_n) \in (\C^{*})^n \mid  1-\sum_{i=1}^n x^{{\bf m}(i)}  \neq 0 \Big\}$, 
$\X =\Big\{ (\xi_1,\dots ,\xi_n) \in (\C^{*})^n \mid 1-\sum_{i=1}^n \xi_i  \neq 0 \Big\}$, 
% \begin{align*}
%   \tX &=\Big\{ (x_1,\dots ,x_n) \in (\C^{*})^n \mid  1-\sum_{i=1}^n x^{{\bf m}(i)}  \neq 0 \Big\}, \\
%   \X &=\Big\{ (\xi_1,\dots ,\xi_n) \in (\C^{*})^n \mid 1-\sum_{i=1}^n \xi_i  \neq 0 \Big\}, 
% \end{align*}
and consider a covering map $p_{M} : \tX \to \X$ defined by 
$p_M(x)=x^M$. 
% \begin{align*}
%   p_M(x)=x^M, \quad \textrm{i.e.,}\quad
%   p_M(x_1,\dots ,x_n)=(x^{{\bf m}(1)},\dots ,x^{{\bf m}(n)}). 
% \end{align*}
% With this notation,
For $\tilde{\bf k}\in \Z^{n\times 1}/\Z \transp M$, we have a deck transformation 
$g^{(\tilde{\bf k})}: \tX \to \tX$ defined by 
$g^{(\tilde{\bf k})} (x) = e^{\tpi \transp M^{-1}\tilde{\bf k}}x$.
% \begin{align*}
%   g^{(\tilde{\bf k})}: \tX \to \tX ;\quad 
%   g^{(\tilde{\bf k})} (x) = e^{\tpi \transp M^{-1}\tilde{\bf k}}x .
% \end{align*}
In this way, we can identify the abelian group $\Z^{n\times 1}/\Z \transp M$ with the deck transformation group ${\rm Gal}(\tilde{X}/X)$. 
Note that the degree of the covering map $p_{M}$ is 
$\abs{\Z^{n\times 1}/\Z \transp M}=\abs{\det M}$.
For any parameters $\alpha_0,\alpha_1,\dots,\alpha_n\in\C \setminus \Z$, we set 
$\alpha =\transp{(\alpha_1,\dots,\alpha_n)}\in \C^{n\times 1}$ and  
\begin{align}
  \uu(\xi )
  &=\xi^{M^{-1}\alpha} \Big(1-\sum_{i=1}^n \xi_i \Big)^{\alpha_0}, \\
  \tu(x)=\uu (p_{M}(x))
  &=x^{\alpha} \Big(1-\sum_{i=1}^n x^{{\bf m}(i)} \Big)^{\alpha_0} 
    =\prod_{i=1}^n x_i^{\alpha_i} \cdot \Big(1-\sum_{i=1}^n x^{{\bf m}(i)} \Big)^{\alpha_0} .
\end{align}
Let $\mathcal{L}=\C \uu$ be the local system of flat sections of
$(\mathcal{O}_{X^{an}},d_{\xi}-d_{\xi}\log u \wedge)$, or equivalently
the locally constant sheaf on $X$ defined by the multi-valued function $\uu$. 
% Before we discuss the construction,

We first prove the purity of the twisted homology group $\Homo_n(\tX ;p_M^{-1}\mathcal{L})$. We set $\alpha^\prime=\transp{(-\alpha_0,\alpha_1,\dots,\alpha_n)}$ and define an $(n+1)\times (n+1)$ matrix $M^\prime$ by
$
M^\prime =
({\bf m}^\prime(0)|\cdots|{\bf m}^\prime(n))=
\left(
\begin{array}{cccc}
1&1&\cdots&1\\
\hline
{\bf O}&{\bf m}(1)&\cdots&{\bf m}(n)
\end{array}
\right).
$
We say that $\alpha^\prime$ is non-resonant if for any proper subset $\Gamma\subsetneq\{ 0,\dots,n\}$, $\alpha^\prime\notin{\rm span}_{\C}\{ {\bf m}^\prime(i)\mid i\in \Gamma\}+\Z^{(n+1)\times 1}$. 
%%%%%%%%%%%%%%%%%%%%%%%%%%%%%%%%%%%%%%%%%%%%%%%%%%%%%%%%%%%%%%%%%%%%%%%
%\textcolor{red}{
\begin{prop}\label{prop:purity-dim}
Suppose that the parameter vector $\alpha^\prime$ is non-resonant and $\alpha_0\notin\Z$. Then, we have the purity
\begin{equation}
\Homo_l\left(\tX ;p_M^{-1}\mathcal{L}\right)=
\begin{cases}
0&(l\neq n)\\
\C^{|\det M|}&(l=n)
\end{cases}
\end{equation}
and the regularization condition is true, i.e., the canonical  morphism ${\rm can}:\Homo_l\left(\tX ;p_M^{-1}\mathcal{L}\right)\rightarrow\Homo_l^{lf}\left(\tX ;p_M^{-1}\mathcal{L}\right)$ is an isomorphism for any $l$, where the symbol $\Homo^{lf}_*$ denotes the locally finite (or Borel-Moore) homology group.  
\end{prop}

\begin{proof}
We define a map $j:(\C^*)^{n+1}\rightarrow \C^{n+1}$ by $j(x_0,\dots,x_n)=(x_0,x_0x^{{\bf m}(1)},\dots,x_0x^{{\bf m}(n)})$. By the proof of \cite[Theorem 2.12]{MatsubaraEuler} and Poincar\'e duality, we only have to prove that the canonical morphism $j_!(\C x_0^{-\alpha_0}x^\alpha)\rightarrow \R j_*(\C x_0^{-\alpha_0}x^\alpha)$ is an isomorphism in order to prove the purity and the regularization condition. Since $j$ defines a covering map onto its image, we see that for any $\oz\in\C^{n+1}$ in the image of $j$, $\R j_*(\C x_0^{-\alpha_0}x^\alpha)_{\oz}$ is concentrated in degree $0$ and $j_!(\C x_0^{-\alpha_0}x^\alpha)_{\oz} \tilde{\rightarrow}  j_*(\C x_0^{-\alpha_0}x^\alpha)_{\oz}$. On the other hand, if $\oz\in\C^{n+1}$ does not lie in the image of $j$, we have $j_!(\C x_0^{-\alpha_0}x^\alpha)_{\oz}=0$. Therefore, it is enough to show the vanishing $\R j_*(\C x_0^{-\alpha_0}x^\alpha)_{\oz}=0$. This is equivalent to showing the vanishing $\Homo^l(j^{-1}(D);\C x_0^{-\alpha_0}x^\alpha)=0$ for any $l$ and for any small neighborhood $D$ of $\oz$. We may assume that there is a proper subset $\Gamma$ of $\{ 0,\dots,n\}$ so that $\oz_j\neq 0$ if and only if $j\in\Gamma$. By homotopy, it is enough to prove that the local system $\C x_0^{-\alpha_0}x^\alpha$ does not have eigenvalue $1$ along any fundamental loop on $j^{-1}(\{ \zeta\in\C^{n+1}\mid \zeta_j=\oz_j(j\in\Gamma),\zeta_j\neq 0(j\notin\Gamma)\})$. Without loss of generality, we can set $\oz_j=1$ for $j\notin\Gamma$. Now we consider the factorization $j=\iota\circ p$ where $\iota:(\C^*)^{n+1}\rightarrow \C^{n+1}$ is the inclusion and $p:(\C^*)^{n+1}\rightarrow(\C^*)^{n+1}$ is a covering map defined by $p(x_0,\dots,x_n)=j(x_0,\dots,x_n)$. We set $N_1:=N_2:=\Z^{n+1}$. Then, it is easy to see that the morphism $\varphi:N_2\ni (a_0,\dots,a_n)\mapsto \sum_{i=0}^na_i{\bf m}^\prime(i)\in N_1$ of abelian groups induces the morphism $p$. We also set $N_\Gamma=\{ (a_0,\dots,a_n)\in N_2\mid a_i=0 (i\in \Gamma)\}$. Then, we obtain the standard identification $j^{-1}(\{ \zeta\in\C^{n+1}\mid \zeta_j=1(j\in\Gamma),\zeta_j\neq 0(j\notin\Gamma)\})=p^{-1}(\{ \zeta\in(\C^*)^{n+1}\mid \zeta_j=1(j\in\Gamma)\})={\rm Specm}\C[N_1/\varphi(N_\Gamma)].$ Note that the inclusion $p^{-1}(\{ \zeta\in(\C^*)^{n+1}\mid \zeta_j=1(j\in\Gamma)\})\hookrightarrow (\C^*)^{n+1}={\rm Specm}\C[N_1]$ is induced from the projection $\varpi:N_1\rightarrow N_1/\varphi(N_\Gamma)$. Then, the local system $\C x_0^{-\alpha_0}x^\alpha$ has an exponent $(\C\otimes \varpi)(\alpha^\prime)$ on $p^{-1}(\{ \zeta\in(\C^*)^{n+1}\mid \zeta_j=1(j\in\Gamma)\})$. Therefore, it does not have eigenvalue $1$ along a fundamental loop if and only if $(\C\otimes \varpi)(\alpha^\prime)\not\in\C \varphi(N_\Gamma)+N_1$, which is precisely the non-resonant condition. Finally, we see that $\dim_{\C}\Homo_n\left(\tX ;p_M^{-1}\mathcal{L}\right)=|\det M|$ due to the identity of the Euler characteristic $\chi(\tilde{X})=|\det M|\cdot \chi(X)$ and the fact $\chi(X)=(-1)^n$. 
\end{proof}

% We set $\bar{\Gamma}=\{0,\dots,n\}\setminus\Gamma$. 
%}
%%%%%%%%%%%%%%%%%%%%%%%%%%%%%%%%%%%%%%%%%%%%%%%%%%%%%%%%%%%%%%%%%%%%%%%
Under the non-resonant condition above, we can also see that
the canonical morphism $\Homo_n\left(\X ;\mathcal{L}\right)\rightarrow\Homo_n^{lf}\left(\X ;\mathcal{L}\right)$
is an isomorphism.
For an element of $\Homo_n^{lf}\left(\X ;\mathcal{L}\right)$, its inverse image through the canonical morphism is
called the regularization.
% For details, refer to \cite[Section 3.2]{AomotoKita}. 
Let $C\in \mathcal{Z}_n(\X;\mathcal{L})$ be a twisted cycle such that
the homology class $[C]\in \Homo_n (\X;\mathcal{L})$ is the regularization of
the element of $\Homo_n^{lf}\left(\X ;\mathcal{L}\right)$ defined by the bounded chamber 
\begin{align}
  \label{BoundedChamber}
  \Big\{ (\xi_1,\dots ,\xi_n) \in \R^n \mid \xi_1>0, \dots ,\xi_n>0, 1-\sum_{i=1}^n \xi_i  >0 \Big\} .
\end{align}

We write $(p_{M})_{*}:\Homo_n(\tX ;p_M^{-1}\mathcal{L})\rightarrow\Homo_n(X ;\mathcal{L})$ for the adjoint of the pull-back map $p_{M}^*:\Homo^n(X ;\mathcal{L}^\vee)\rightarrow\Homo^n(\tX ;p_M^{-1}\mathcal{L}^\vee)$. Namely, $(p_{M})_{*}$ is characterized by the equality $\int_{\s}p_{M}^*\omega=\int_{(p_{M})_{*}\s}\omega$ for any $\s\in\Homo_n(\tX ;p_M^{-1}\mathcal{L})$ and $\omega\in\Homo^n(X ;\mathcal{L}^\vee)$.
In fact, $(p_M)_{*}$ is defined on the level of 
$\mathcal{Z}_n(\tX ;p_M^{-1}\mathcal{L})$, 
and its expression for each twisted cycle is explained in the proof of Proposition \ref{prop:lift-intersection} (2) below. 
In \S\ref{subsection:diagonal} and \S\ref{subsection:not-diagonal}, we construct twisted cycles 
$\tC^{(\tilde{\bf k})} \in \mathcal{Z}_n(\tX ;p_M^{-1}\mathcal{L})$ 
($\tilde{\bf k} \in \Z^{n\times 1}/\Z \transp M$)
such that $(p_M)_{*} (\tC^{(\tilde{\bf k})}) =C$,
and evaluate their intersection numbers. 

\begin{prop}
  \label{prop:lift-intersection}
  For $h_1(\xi),\dots ,h_l(\xi) \in \C [\xi_1 ,\dots ,\xi_n]$, 
  $\alpha_1,\dots,\alpha_l\in\C \setminus \Z$, 
  we set $\uu(\xi)=h_1(\xi)^{\alpha_1}\cdots h_l(\xi)^{\alpha_l}$ and  
  $\mathcal{L}=\C \uu$, which are defined on 
  $X=\lef \xi \in \C^n \mid h_1(\xi)\cdots h_l(\xi) \neq 0\righ$.  
  Let $p:\tX \to X$ be a locally biholomorphic covering map of degree $r$,
  where $\tX$ is a complex manifold. 
  \begin{enumerate}
  \item[(1)] %\label{prop:lift-intersection-1}
    For a twisted cycle $\tilde{\s}\in \mathcal{Z}_n(\tX ;p^{-1}\mathcal{L})$, 
    we set $\s =p_{*}(\tilde{\s}) \in \mathcal{Z}_n(\X;\mathcal{L})$. 
    Then, there exist $r$ twisted cycles 
    $\tilde{\s}^{(1)}(=\tilde{\s}),\tilde{\s}^{(2)},\dots ,\tilde{\s}^{(r)}\in \mathcal{Z}_n(\tX ;p^{-1}\mathcal{L})$ 
    such that $p_{*}(\tilde{\s}^{(k)})=\s$ as twisted cycles 
    (not as elements of the twisted homology group) 
    and the support of them are mutually distinct.
  \item[(2)] For twisted cycles $\tilde{\s}\in \mathcal{Z}_n(\tX ;p^{-1}\mathcal{L})$ and 
    $\tilde{\tau}\in \mathcal{Z}_n(\tX ;p^{-1}\mathcal{L}^{\vee})$, 
    we set $\s =p_{*}(\tilde{\s}) \in \mathcal{Z}_n(\X;\mathcal{L})$ and 
    $\tau =p_{*}(\tilde{\tau}) \in \mathcal{Z}_n(\X;\mathcal{L}^{\vee})$. 
    Let $\tilde{\tau}^{(1)},\dots ,\tilde{\tau}^{(r)}\in \mathcal{Z}_n(\tX ;p^{-1}\mathcal{L}^{\vee})$ 
    be the twisted cycles obtained in (1). 
    Then we have 
    \begin{align}
      \langle \s ,\tau \rangle_h =
      \sum_{k=1}^r \langle \tilde{\s} , \tilde{\tau}^{(k)} \rangle_h , 
    \end{align}
    where $\langle \bullet,\bullet \rangle_h$ on the left-hand (resp. right-hand) side is 
    the intersection form on $\Homo_n(\X;\mathcal{L})$ 
    (resp. $\Homo_n(\tX ;p^{-1}\mathcal{L})$). 
  \end{enumerate}
\end{prop}
The twisted cycles $\tilde{\s}^{(1)},\dots ,\tilde{\s}^{(r)}$ in (1) are called 
the \textit{lifts} of $\s$. 
\begin{proof}
  \begin{enumerate}
  \item[(1)] Consider the deck transformations. 
  \item[(2)] We take an open covering $\{ U_{\lambda} \}_{\lambda}$ of $X$ such that 
    we have $p^{-1}(U_{\lambda})=\bigsqcup_{k=1}^r \tilde{U}_{\lambda,k}$ and each 
    $p|_{\tilde{U}_{\lambda,k}} : \tilde{U}_{\lambda,k} \to U_{\lambda}$  
    is biholomorphic.
    % \begin{align*}
    %   p^{-1}(U_{\lambda})=\bigsqcup_{k=1}^r \tilde{U}_{\lambda,k} ,\quad 
    %   p|_{\tilde{U}_{\lambda,k}} : \tilde{U}_{\lambda,k} \to U_{\lambda} \ 
    %   \text{is biholomorphic.}
    % \end{align*}
    We may assume that $\s$ and $\tau$ are expressed as 
    \begin{align}
      \label{eq:cycle-expression-proof}
      \s = \sum_{i} \mu_i \cdot \Delta_i \ot u_{\Delta_i} ,\quad 
      \tau = \sum_{j} \mu'_j \cdot \Delta'_j \ot u_{\Delta'_j}^{-1}
      \qquad (\mu_i,\mu'_j\in \C),
    \end{align}
    where $\Delta_i$ and $\Delta'_j$ are simply connected and 
    included in some $U_{\lambda}$, and  
    the intersection $\Delta_i \cap \Delta'_j$ is at most one point 
    at which they intersect transversally. 
    By \cite{KitaYoshida1}, the intersection number is evaluated as 
    \begin{align}
      \label{eq:homology-intersection}
      \langle \s ,\tau \rangle_h
      =\sum_{\Delta_i \cap \Delta'_j =\{ \xi \} }
      \mu_i \mu'_j \cdot
      I_{loc}(\Delta_i , \Delta'_j) \cdot 
      % (\Delta_i , \Delta'_j)_{\textrm{topological}} \cdot 
      u_{\Delta_i}(\xi) \cdot u_{\Delta'_j}(\xi)^{-1} .
    \end{align}
    % where $(\Delta_i , \Delta'_j)_{\textrm{topological}}$ indicates the topological intersection number. 
    Let $\tilde{\s}=\tilde{\s}^{(1)},\tilde{\s}^{(2)},\dots ,\tilde{\s}^{(r)}$ and 
    $\tilde{\tau}^{(1)},\dots ,\tilde{\tau}^{(r)}$ be the lifts of $\s$ and $\tau$, respectively. 
    We can express them as 
    \begin{align}
      \tilde{\s}^{(k)} = \sum_{i} {\tilde{\mu}_{i}}^{(k)} \cdot 
      \tilde{\Delta}_{i}^{(k)} \ot \tu_{\tilde{\Delta}_{i}^{(k)}} ,\quad 
      \tilde{\tau}^{(k)} = \sum_{j} \tilde{\mu}'_{j}{}^{(k)} \cdot 
      \tilde{\Delta}'_{j}{}^{(k)} \ot \tu_{\tilde{\Delta}'_{j}{}^{(k)}}^{-1} , 
    \end{align}
    where 
    $p^{-1} (\Delta_i) = \bigsqcup_{k=1}^r \tilde{\Delta}_{i}^{(k)}$ and 
    $p^{-1} (\Delta'_j) = \bigsqcup_{k=1}^r \tilde{\Delta}'_{j}{}^{(k)}$. 
    % \begin{align*}
    %   p^{-1} (\Delta_i) = \bigsqcup_{k=1}^r \tilde{\Delta}_{i}^{(k)} ,\quad 
    %   p^{-1} (\Delta'_j) = \bigsqcup_{k=1}^r \tilde{\Delta}'_{j}{}^{(k)}. 
    % \end{align*}
    For an intersection point $\{ \xi \}=\Delta_i \cap \Delta'_j$ in (\ref{eq:homology-intersection}), 
    there exist a unique $x\in \tX$ and a unique $k_x \in \{1,\dots,r\}$ such that 
    $p(x)=\xi$ and $x\in \tilde{\Delta}_{i}^{(1)} \cap \tilde{\Delta}'_{j}{}^{(k_x)}$. 
    Conversely, all the intersection points of $\tilde{\s}^{(1)}$ and 
    % $\tilde{\tau}^{(1)},\dots ,\tilde{\tau}^{(r)}$
    $\tilde{\tau}^{(k)}$'s 
    are obtained in this way. 
    Thus, we have 
    \begin{align}
      \sum_{k=1}^r \langle \tilde{\s}^{(1)} , \tilde{\tau}^{(k)} \rangle_h
      =\sum_{\substack{\Delta_i \cap \Delta'_j =\{ \xi \} \\ x \in \tilde{\Delta}_{i}^{(1)} \ \textrm{and} \ p(x)=\xi} }
      \tilde{\mu}_{i}^{(1)} \tilde{\mu}'_{j}{^{(k_x)}} \cdot 
      I_{loc}(\tilde{\Delta}_{i}^{(1)} , \tilde{\Delta}'_{j}{}^{(k_x)})
      % (\tilde{\Delta}_{i}^{(1)} , \tilde{\Delta}'_{j}{}^{(k_x)})_{\textrm{topological}}
      \cdot \tu_{\tilde{\Delta}_{i}^{(1)}}(x) \cdot \tu_{\tilde{\Delta}'_{j}{}^{(k_x)}}(x)^{-1}  .
    \end{align}
    Since each $p|_{\tilde{\Delta}_{i}^{(1)}}$ is biholomorphic, it follows that 
    $I_{loc}(\tilde{\Delta}_{i}^{(1)} , \tilde{\Delta}'_{j}{}^{(k_x)})=I_{loc}(\Delta_i , \Delta'_j)$.
    % $(\tilde{\Delta}_{i}^{(1)} , \tilde{\Delta}'_{j}{}^{(k_x)})_{\textrm{topological}}
    % =(\Delta_i , \Delta'_j)_{\textrm{topological}}$.
    % \begin{align*}
    %   (\tilde{\Delta}_{i}^{(1)} , \tilde{\Delta}'_{j}{}^{(k_x)})_{\textrm{topological}}
    %   =(\Delta_i , \Delta'_j)_{\textrm{topological}} .
    % \end{align*}
    By comparing (\ref{eq:cycle-expression-proof}) with 
    $\s = p_{*}(\tilde{\s}^{(1)}) 
    =\sum_{i} {\tilde{\mu}_{i}}^{(1)} \cdot 
    \Delta_{i} \ot \tu_{\tilde{\Delta}_{i}^{(1)}}\circ p|_{\tilde{\Delta}_{i}^{(1)}}^{-1}$, 
    % \begin{align*}
    %   \s = p_{*}(\tilde{\s}^{(1)}) 
    %   =\sum_{i} {\tilde{\mu}_{i}}^{(1)} \cdot 
    %   \Delta_{i} \ot \tu_{\tilde{\Delta}_{i}^{(1)}}\circ p|_{\tilde{\Delta}_{i}^{(1)}}^{-1} , 
    % \end{align*}
    we obtain 
    \begin{align}
      \tilde{\mu}_i^{(1)} = \mu_i \cdot \frac{u_{\Delta_i}}{\tu_{\tilde{\Delta}_i} \circ p|_{\tilde{\Delta}_{i}^{(1)}}^{-1}} ,
    \end{align}
    and hence 
    \begin{align}
      \tilde{\mu}_{i}^{(1)} \cdot \tu_{\tilde{\Delta}_{i}^{(1)}}(x)
      =\mu_i \cdot \frac{u_{\Delta_i}(\xi)}{\tu_{\tilde{\Delta}_{i}^{(1)}} \circ p|_{\tilde{\Delta}_{i}^{(1)}}^{-1}(\xi)} 
      \cdot \tu_{\tilde{\Delta}_{i}^{(1)}}(x)
      =\mu_i \cdot \frac{u_{\Delta_i}(\xi)}{\tu_{\tilde{\Delta}_{i}^{(1)}}(x)} 
      \cdot \tu_{\tilde{\Delta}_{i}^{(1)}}(x)
      =\mu_i \cdot u_{\Delta_i}(\xi) .
    \end{align}
    Similarly, we have 
    $\tilde{\mu}'_{j}{}^{(k_x)} \cdot \tu_{\tilde{\Delta}_{j}^{(k_x)}}(x)^{-1}=\mu'_j \cdot u_{\Delta_j}(\xi)^{-1}$. 
    Therefore, we obtain
    \begin{align}
      \sum_{k=1}^r \langle \tilde{\s} , \tilde{\tau}^{(k)} \rangle_h  
      =\sum_{k=1}^r \langle \tilde{\s}^{(1)} , \tilde{\tau}^{(k)} \rangle_h
      =\sum_{\Delta_i \cap \Delta'_j =\{ \xi \} }
      \mu_i \mu'_j \cdot
      I_{loc}(\Delta_i , \Delta'_j)
      % (\Delta_i , \Delta'_j)_{\textrm{topological}}
      \cdot u_{\Delta_i}(\xi) \cdot u_{\Delta'_j}(\xi)^{-1}
      =\langle \s ,\tau \rangle_h .
    \end{align}
  \end{enumerate}
\end{proof}

\subsection{Diagonal cases}\label{subsection:diagonal}
First, we consider the case when the matrix $M$ is diagonal. 
We assume $M=D=\diag (m_1,\dots ,m_n)$, where each $m_i$ is a positive integer. 
Then we have
$\tX =\Big\{ (x_1,\dots ,x_n) \in (\C^{*})^n \mid  1- \sum_{i=1}^n x_i^{m_i} \neq 0\Big\}$, 
$\X =\Big\{ (\xi_1,\dots ,\xi_n) \in (\C^{*})^n \mid  1-\sum_{i=1}^n \xi_i  \neq 0\Big\}$, 
$\tu(x)
=\prod_{i=1}^n x_i^{\alpha_i} \cdot \Big(1-\sum_{i=1}^n x_i^{m_i} \Big)^{\alpha_0}$, 
$\uu(\xi )
=\prod_{i=1}^n \xi_i^{\frac{\alpha_i}{m_i}} \cdot \Big(1-\sum_{i=1}^n \xi_i \Big)^{\alpha_0}$, 
% \begin{align*}
%   &\tX =\Big\{ (x_1,\dots ,x_n) \in (\C^{*})^n \mid  1- \sum_{i=1}^n x_i^{m_i} \neq 0\Big\}, &
%   & \X =\Big\{ (\xi_1,\dots ,\xi_n) \in (\C^{*})^n \mid  1-\sum_{i=1}^n \xi_i  \neq 0\Big\}, \\
%   &\tu(x)
%   =\prod_{i=1}^n x_i^{\alpha_i} \cdot \Big(1-\sum_{i=1}^n x_i^{m_i} \Big)^{\alpha_0}, &
%   & \uu(\xi )
%   =\prod_{i=1}^n \xi_i^{\frac{\alpha_i}{m_i}} \cdot \Big(1-\sum_{i=1}^n \xi_i \Big)^{\alpha_0}
% \end{align*}
and $p_D(x_1,\dots ,x_n)=(x_1^{m_1},\dots ,x_n^{m_n})$. 
The covering map $p_D$ is of degree $m_1 \cdots m_n$. 

We review the construction of the twisted cycle $C$ whose homology class is the regularization 
of the bounded chamber (\ref{BoundedChamber}). 
For details, the readers may refer to \cite[\S3.2.4]{AomotoKita}. 
For a sufficiently small positive number $\varepsilon$, let $\triangle$ be a closed simplex 
$\Big\{ (\xi_1,\dots ,\xi_n) \in \R^n \mid \xi_1 \geq \varepsilon, \dots ,\xi_n\geq \varepsilon, 
1-\sum_{i=1}^n \xi_i  \geq \varepsilon \Big\}$.
% \begin{align*}
%   \Big\{ (\xi_1,\dots ,\xi_n) \in \R^n \mid \xi_1 \geq \varepsilon, \cdots ,\xi_n\geq \varepsilon, 
%   1-\sum_{i=1}^n \xi_i  \geq \varepsilon \Big\} .
% \end{align*}
For $j=0, \dots ,n$, 
let $T_j$ be the $\varepsilon$-neighborhood of the divisor $(\xi_j=0)$ 
(when $j=0$, replace ``$(1-\sum_{i=1}^n \xi_i=0)$"), 
$l_j$ the face of $\triangle$ given by $\triangle \cap \conj{T_j}$, and 
$S_j$ the circle with radius $\varepsilon$ such that $\partial T_j =S_j \times (\xi_j=0)$. 
The twisted cycle $C$ is expressed as 
\begin{align}
  % \label{eq:reguralization-expression1}
  C=\triangle \ot \uu +\sum_{\emptyset \neq I \subsetneq \{ 0,\dots ,n \}}
  \left( \prod_{i\in I} \frac{1}{e^{\tpi \frac{\alpha_i}{m_i}}-1} \cdot 
  \left(\left( \bigcap_{i\in I}l_i \right) \times \prod_{i\in I}S_i \right)\ot \uu
  \right) ,
\end{align}
where we set $m_0 =1$. 
The branch of $u$ on $\triangle$ is given by 
$\arg \xi_1 = \cdots =\arg \xi_n =\arg (1-\sum_{i=1}^n \xi_i)=0$, 
and those of the other chains are defined by its analytic continuations. 

For $i=1,\dots ,n$, let $S_i^{m_i}$ be the oriented circle 
going around the divisor $(\xi_i=0)$ $m_i$ times (see Figure \ref{fig:circles}). 
If the branches of $\uu$ at the starting points of $S_i$ and $S_i^{m_i}$ are same, then 
we have
\begin{align}
  S_i^{m_i} \ot \uu
  =\sum_{q=0}^{m_i-1} S_i \ot e^{\tpi \frac{\alpha_i}{m_i}q} \uu
  =\Bigg( \sum_{q=0}^{m_i-1} e^{\tpi \frac{\alpha_i}{m_i}q} \Bigg) S_i \ot \uu
  =\frac{e^{\tpi \alpha_i}-1}{e^{\tpi \frac{\alpha_i}{m_i}}-1} \cdot S_i \ot \uu
\end{align}
modulo the image of the boundary operator. Then the twisted cycle $C$ is also expressed as 
\begin{align}
  \label{eq:reguralization-expression2}
  C=\triangle \ot \uu +\sum_{\emptyset \neq I \subsetneq \{ 0,\dots ,n \}}
  \left( \prod_{i\in I} \frac{1}{e^{\tpi \alpha_i}-1} \cdot 
  \left(\left( \bigcap_{i\in I}l_i \right) \times \prod_{i\in I}S_i^{m_i} \right)\ot \uu
  \right)
\end{align}
modulo the image of the boundary operator. For $i=1,\dots ,n$ and $q \in \Z$, 
let $\tilde{S}_i^{(q)}$ be the circle in $\tX$ defined by $p_D^{-1}(S_i^{m_i})$ whose starting point is 
$x_i =\sqrt[m_i]{\varepsilon}e^{\tpi \frac{q}{m_i}}$. 
% For a twisted 1-chain $S_i^{m_i}\ot u$,
We have 
$(p_D)_{*}(\tilde{S}_i^{(q)}\ot \tu) = e^{\tpi \frac{\alpha_i}{m_i}q} \cdot S_i^{m_i} \ot \uu$ 
% \begin{align*}
%   (p_D)_{*}(\tilde{S}_i^{(q)}\ot \tu) = e^{\tpi \frac{\alpha_i}{m_i}q} \cdot S_i^{m_i} \ot \uu
%   % \quad \textrm{i.e.,} \quad 
%   % (p_M)_{*}\left( \frac{1}{e^{\tpi \alpha_i}-1} \tilde{S}_j^{(k)} \ot \tu \right) 
%   % =\frac{e^{\tpi \frac{\alpha_i}{m_i}k}}{e^{\tpi \alpha_i}-1}S_j \ot \tu ,
% \end{align*}
where the branch $\tu$ is defined by $\uu \circ p_D$. 
The preimage $p_D^{-1}(\triangle)$ is decomposed into 
$p_D^{-1}(\triangle) = \bigsqcup_{{\bf q}\in \Z^{n\times 1} /\Z \transp{D}} \tilde{\triangle}^{({\bf q})}$, 
where $\Z^{n\times 1} /\Z \transp{D}
=\Z^{n\times 1} /(\Z \diag(m_1,\dots ,m_n)) =(\Z/m_1 \Z)\times \cdots \times (\Z/m_n \Z)$ 
(column vectors) and $\tilde{\triangle}^{({\bf q})}$ is characterized by the condition 
$(\sqrt[m_1]{\varepsilon}e^{\tpi \frac{q_1}{m_1}},\dots,\sqrt[m_n]{\varepsilon}e^{\tpi \frac{q_n}{m_n}})
\in \tilde{\triangle}^{({\bf q})}$ for  
${\bf q}=\transp{(q_1,\dots ,q_n)} \in \Z^{n\times 1} /\Z \transp{D}$.  
% \begin{align*}
%   p_D^{-1}(\triangle) = \bigsqcup_{{\bf q}\in \Z^{n\times 1} /\Z \transp{D}} \tilde{\triangle}^{({\bf q})} ;\qquad 
%   (\sqrt[m_1]{\varepsilon}e^{\tpi \frac{q_1}{m_1}},\dots,\sqrt[m_n]{\varepsilon}e^{\tpi \frac{q_n}{m_n}})
%   \in \tilde{\triangle}^{({\bf q})} 
%   \quad \textrm{if} \quad {\bf q}=\transp{(q_1,\dots ,q_n)}, 
% \end{align*}
% where $\Z^{n\times 1} /\Z \transp{D}
% =\Z^{n\times 1} /(\Z \diag(m_1,\dots ,m_n)) =(\Z/m_1 \Z)\times \cdots \times (\Z/m_n \Z)$ 
% (column vectors).  
Let $\tilde{l}_j^{({\bf q})}$ be the face of $\tilde{\triangle}^{({\bf q})}$ 
that satisfies $p_D(\tilde{l}_j^{({\bf q})}) =l_j$, 
and $\tilde{S}_0^{({\bf q})}$ the connected component of the preimage $p_D^{-1} (S_0)$ 
whose starting point is in $\tilde{l}_j^{({\bf q})}$. 
% For ${\bf q}=\transp{(q_1,\dots ,q_n)}$,
We have 
$(p_D)_{*}(\tilde{\triangle}^{({\bf q})} \ot \tu) 
=\prod_{i=1}^n e^{\tpi \frac{\alpha_i}{m_i}q_i} \cdot \triangle \ot \uu$, 
$(p_D)_{*}(\tilde{l}_j^{({\bf q})} \ot \tu) 
=\prod_{i=1}^n e^{\tpi \frac{\alpha_i}{m_i}q_i} \cdot l_j \ot \uu$, and 
$(p_D)_{*}(\tilde{S}_0^{({\bf q})} \ot \tu) 
=\prod_{i=1}^n e^{\tpi \frac{\alpha_i}{m_i}q_i} \cdot S_0 \ot \uu$,  
% \begin{align*}
%   &(p_D)_{*}(\tilde{\triangle}^{({\bf q})} \ot \tu) 
%   =\prod_{i=1}^n e^{\tpi \frac{\alpha_i}{m_i}q_i} \cdot \triangle \ot \uu, \quad 
%   (p_D)_{*}(\tilde{l}_j^{({\bf q})} \ot \tu) 
%   =\prod_{i=1}^n e^{\tpi \frac{\alpha_i}{m_i}q_i} \cdot l_j \ot \uu,  \\
%   &(p_D)_{*}(\tilde{S}_0^{({\bf q})} \ot \tu) 
%   =\prod_{i=1}^n e^{\tpi \frac{\alpha_i}{m_i}q_i} \cdot S_0 \ot \uu,  
% \end{align*}
where the branches $\tu$ are defined by $\uu \circ p_D$. 

\begin{figure}[h]
  \centering
  \begin{tikzpicture}
    \draw (0,0) circle [radius=2pt];
    \draw[->, smooth, samples=100, variable=\t, domain=0:6.20] plot (\t r:{1});
    \draw (1,1) node {$S_i$};
    \draw (0,0) node[below] {$\xi_i=0$};
  \end{tikzpicture}
  \hspace{60pt}
  \begin{tikzpicture}
    \draw (0,0) circle [radius=2pt];
    \draw[->, smooth, samples=100, variable=\t, domain=0:18.84] plot (\t r:{1+(0.02*\t)});
    \draw (1.5,1) node {$S_i^{3}$};
    \draw (0,0) node[below] {$\xi_i=0$};
  \end{tikzpicture}
  \caption{$S_i=S_i^1$ and $S_i^{m_i}$}\label{fig:circles}
\end{figure}
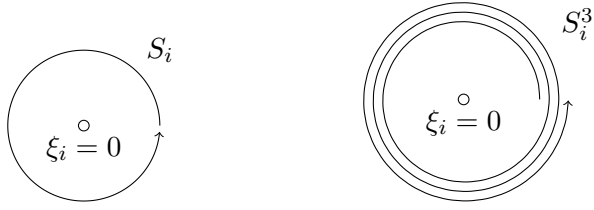

For ${\bf q}=\transp{(q_1,\dots ,q_n)} \in \Z^{n\times 1} /\Z \transp{D}$, we set 
\begin{equation}
  % \nonumber 
  \tC^{({\bf q})} 
= \frac{1}{ e^{\tpi \transp{\bf q} D^{-1} \alpha}}
  \left(\tilde{\triangle}^{({\bf q})} \ot \tu +\sum_{\emptyset \neq I \subsetneq \{ 0,\dots ,n \}}
  \left( \prod_{i\in I} \frac{1}{e^{\tpi \alpha_i}-1} \cdot 
  \left(\left( \bigcap_{i\in I}\tilde{l}_i^{({\bf q})} \right) \times \prod_{i\in I}\tilde{S}_i^{({\bf q})} \right)\ot \tu
  \right) \right) , \label{eq:def-lifted-cycle}
\end{equation}
where $\tilde{S}_i^{({\bf q})} =\tilde{S}_i^{(q_i)}$ if $i=1,\dots ,n$. It is easy to see that $\tC^{({\bf q})}$ defines an element of $\mathcal{Z}_n(\tX,p_D^{-1}\mathcal{L})$ as in \cite[\S3.2.4]{AomotoKita}. Note that $e^{\tpi \transp{\bf q} D^{-1} \alpha}=\prod_{i=1}^n e^{\tpi \frac{\alpha_i}{m_i}q_i}$. By the expression (\ref{eq:reguralization-expression2}), this cycle satisfies
$(p_D)_{*} (\tC^{({\bf q})}) =C$.
% \begin{align*}
%   (p_D)_{*} (\tC^{({\bf q})}) =C.
% \end{align*}

\begin{rem}
  Note that $\tC^{({\bf q})}$ is independent of the choice of the representative 
  ${\bf q}$ of $\Z^{n\times 1} /\Z \transp{D} =(\Z/m_1 \Z)\times \cdots \times (\Z/m_n \Z)$. 
  As an example, let us compare two cycles
  \begin{align}
    \tC^{(0)} &= 
    \tilde{\triangle}^{(0)} \ot \tu_0 
    +\frac{1}{e^{\tpi \alpha_1}-1} \tilde{S}_1^{(0)} \ot \tu_0 
    -\frac{1}{e^{\tpi \alpha_0}-1} \tilde{S}_0^{(0)} \ot \tu_0 , \\
    \tC^{(m_1)} &= \frac{1}{e^{\tpi \alpha_1} }
    \left(\tilde{\triangle}^{(m_1)} \ot \tu_1 
    +\frac{1}{e^{\tpi \alpha_1}-1} \tilde{S}_1^{(m_1)} \ot \tu_1 
    -\frac{1}{e^{\tpi \alpha_0}-1} \tilde{S}_0^{(m_1)} \ot \tu_1 
    \right)
  \end{align}
  for $n=1$. 
  Though $\tilde{\triangle}^{(0)}$ (resp. $\tilde{S}_1^{(0)}$, $\tilde{S}_0^{(0)}$) and 
  $\tilde{\triangle}^{(m_1)}$ (resp. $\tilde{S}_1^{(m_1)}$, $\tilde{S}_0^{(m_1)}$) are same as sets, 
  the branches $\tu_0$ on $\tilde{\triangle}^{(0)}$ and $\tu_1$ on $\tilde{\triangle}^{(m_1)}$ are different: 
  $\tu_1 =e^{\tpi \alpha_1} \cdot \tu_0$. 
  Thus, $\tC^{(0)}$ and $\tC^{(m_1)}$ defines the same twisted cycle. 
\end{rem}

For a twisted cycle $\sigma \in \mathcal{Z}_n(\tX ;p_D^{-1}\mathcal{L})$, 
we write $\sigma^{\vee} \in \mathcal{Z}_n(\tX ;p_D^{-1}\mathcal{L}^{\vee})$ for a twisted cycle which can be obtained by replacing $\tu$ with $\tu^{-1}=1/\tu$ in the construction of $\sigma$. 
\begin{prop}
  For ${\bf q}=\transp{(q_1,\dots ,q_n)}, {\bf q}'=\transp{(q'_1,\dots ,q'_n)} \in \Z^{n\times 1} /\Z \transp{D}$, 
  we assume $0 \leq q_i ,q'_i \leq m_i -1$ and 
  set $I({\bf q},{\bf q}')=\{ i \mid q_i=q_i' \}$. 
  Then we have 
  \begin{align}
    \label{eq:intersection-number-lift}
    \langle \tC^{({\bf q})}, \tC^{({\bf q}')\vee}\rangle_h
    =\frac{\left( 1-e^{\tpi \alpha_0} \prod_{i\in I({\bf q},{\bf q}')} e^{\tpi \alpha_i}  \right) \cdot 
    \prod_{i\not\in I({\bf q},{\bf q}')} e^{\tpi \alpha_i \{ \frac{q'_i -q_i}{m_i} \} }}
    {(1-e^{\tpi \alpha_0})\prod_{i=1}^n (1-e^{\tpi \alpha_i})} , 
  \end{align}
  where 
  \begin{align}
    \lef \frac{q'_i -q_i}{m_i} \righ =
    \begin{cases}
      \frac{q'_i-q_i}{m_i} & (q_i < q'_i) \\
      \frac{m_i+q'_i-q_i}{m_i} & (q_i > q'_i) . 
    \end{cases}
  \end{align}
\end{prop}
\begin{proof}
  First, we assume $|I({\bf q},{\bf q}')|$ is neither $0$ nor $n$. 
  By using a discussion similar to the proof of
  \cite[Proposition in p.175]{KitaYoshida2}, we obtain 
  \begin{align}
    \langle \tC^{({\bf q})}, \tC^{({\bf q}')\vee}\rangle_h
    =\frac{1}{\prod_{i=1}^n e^{\tpi \frac{\alpha_i}{m_i}(q_i-q'_i)}}
    \cdot \prod_{i\not\in I({\bf q},{\bf q}')} \frac{-1 \cdot (\textrm{difference})}{e^{\tpi \alpha_i}-1} 
    \cdot \langle C_1^{({\bf q},{\bf q}')}, C_1^{({\bf q},{\bf q}')\vee} \rangle_h ,
    % \cdot (\textrm{the self-intersection number of $C_1^{({\bf q},{\bf q}')}$}) ,
  \end{align}
  where ``$-1$'' in the middle comes from
  the local intersection multiplicity, 
  % the topological intersection number, 
  ``(difference)'' is 1 (resp. $e^{\tpi \alpha_i}$) when 
  $q_i<q'_i$ (resp. $q_i>q'_i$) which  
  expresses the difference of the branches 
  $\tu$ on $\tilde{S}_i^{(q_i)}$ and $1/ \tu$ on $\tilde{\triangle}_i^{({\bf q}')}$
  (see Figure \ref{fig:intersection-1dim}), 
  and $C_1^{({\bf q},{\bf q}')}$ is a twisted cycle (of dimension $|I({\bf q},{\bf q}')|$) 
  \begin{align}
    \bigcap_{i\not\in I({\bf q},{\bf q}')}\tilde{l}_i^{({\bf q})} \ot \tu|_{X_{{\bf q},{\bf q}'}}
    +\sum_{\emptyset \neq I \subsetneq \{ 0\} \cup I({\bf q},{\bf q}')}
    \left( \prod_{i\in I} \frac{1}{e^{\tpi \alpha_i}-1} \cdot 
    \left(\left( \bigcap_{i\in I\cup I({\bf q},{\bf q}')^c}\tilde{l}_i^{({\bf q})} \right) 
    \times \prod_{i\in I}\tilde{S}_i^{({\bf q})} \right)\ot \tu|_{X_{{\bf q},{\bf q}'}}
    \right) 
  \end{align}
  in the $|I({\bf q},{\bf q}')|$-dimensional space 
  $X_{{\bf q},{\bf q}'}=\bigcap_{i\not\in I({\bf q},{\bf q}')} 
  \{x_i=\sqrt[m_i]{\varepsilon}e^{\tpi \frac{q_i}{m_i}}\}$. 
  Since the self-intersection number
  $\langle C_1^{({\bf q},{\bf q}')}, C_1^{({\bf q},{\bf q}')\vee} \rangle_h$
  % of this twisted cycle
  is evaluated as 
  that of a $|I({\bf q},{\bf q}')|$-simplex (see \cite[Example in p.177]{KitaYoshida2}), 
  it equals to  
  \begin{align}
    \frac{1-e^{\tpi \alpha_0} \prod_{i\in I({\bf q},{\bf q}')} e^{\tpi \alpha_i} }
    {(1-e^{\tpi \alpha_0})\prod_{i\in I({\bf q},{\bf q}')} (1-e^{\tpi \alpha_i})} , 
  \end{align}
  and hence we obtain (\ref{eq:intersection-number-lift}). 
  Next, we consider the case $I({\bf q},{\bf q}')=\emptyset$. 
  In this case, we have 
  \begin{align}
    \langle \tC^{({\bf q})}, \tC^{({\bf q}')\vee}\rangle_h
    =\frac{1}{\prod_{i=1}^n e^{\tpi \frac{\alpha_i}{m_i}(q_i-q'_i)}}
    \cdot \prod_{i=1}^n \frac{-1 \cdot (\textrm{difference})}{e^{\tpi \alpha_i}-1}  
  \end{align}
  which is equivalent to (\ref{eq:intersection-number-lift}). 
  Finally, we consider the self-intersection number 
  $\langle \tC^{({\bf q})}, \tC^{({\bf q})\vee}\rangle_h$. 
  It can be evaluated
  in the same manner as that for the regularization of an $n$-simplex. 
  % as the self-intersection number of $n$-simplex $\tilde{\triangle}^{({\bf q})}$. 
  Therefore we have 
  \begin{align}
    \langle \tC^{({\bf q})}, \tC^{({\bf q})\vee}\rangle_h
    =\frac{ 1-\prod_{j=0}^n e^{\tpi \alpha_j} }
    {\prod_{j=0}^n (1-e^{\tpi \alpha_j})} ,   
  \end{align}
  and the proof is completed. 
\begin{figure}[h]
  \centering
  \begin{tikzpicture}
    \draw (0,0) circle [radius=2pt];
    \draw (0,0) node[below] {$0$};
    \draw[->] (0.866,0.5) arc [start angle = 30, end angle = 380, radius = 1];
    \draw[->] (0.866,0.5) -- ++ (30:1);
    \draw[->] (0.866*2,0.5*2) arc [start angle = -150, end angle = 200, radius = 1];
    \draw (0.866*3,0.5*3) circle [radius=2pt];
    \draw (0.866*3,0.5*3) node[below] {$\zeta_{m_i}^{q_i}$};
    \draw[->] (-0.5*0.8,0.866*0.8) arc [start angle = 120, end angle = 470, radius = 0.8];
    \draw[->] (-0.5*0.8,0.866*0.8) -- ++ (120:1.2);
    \draw[->] (-0.5*2,0.866*2) arc [start angle = -60, end angle = 290, radius = 1];
    \draw (-0.5*3,0.866*3) circle [radius=2pt];
    \draw (-0.5*3,0.866*3) node[below] {$\zeta_{m_i}^{q'_i}$};
    \draw (3,0) node {$\tilde{C}^{(q_1)}$};
    \draw (0.2,3) node {$\tilde{C}^{(q'_1)\vee}$};
    \fill [black] (-0.5,0.866) circle [radius=3pt];
  \end{tikzpicture}
  \hspace{60pt}
  \begin{tikzpicture}
    \draw (0,0) circle [radius=2pt];
    \draw (0,0) node[below] {$0$};
    \draw[->] (0.866*0.8,0.5*0.8) arc [start angle = 30, end angle = 380, radius = 0.8];
    \draw[->] (0.866*0.8,0.5*0.8) -- ++ (30:1.2);
    \draw[->] (0.866*2,0.5*2) arc [start angle = -150, end angle = 200, radius = 1];
    \draw (0.866*3,0.5*3) circle [radius=2pt];
    \draw (0.866*3,0.5*3) node[below] {$\zeta_{m_i}^{q'_i}$};
    \draw[->] (-0.5,0.866) arc [start angle = 120, end angle = 470, radius = 1];
    \draw[->] (-0.5,0.866) -- ++ (120:1);
    \draw[->] (-0.5*2,0.866*2) arc [start angle = -60, end angle = 290, radius = 1];
    \draw (-0.5*3,0.866*3) circle [radius=2pt];
    \draw (-0.5*3,0.866*3) node[below] {$\zeta_{m_i}^{q_i}$};
    \draw (3,0) node {$\tilde{C}^{(q'_1)\vee}$};
    \draw (0.2,3) node {$\tilde{C}^{(q_1)}$};
    \fill [black] (0.866,0.5) circle [radius=3pt];
  \end{tikzpicture}
  \caption{$n=1$; if $q_1<q'_1$ (left) (resp. $q_1>q'_1$ (right)), 
    $\arg (x_1)$ at $\bullet$ are same (resp. different).}\label{fig:intersection-1dim}
\end{figure}
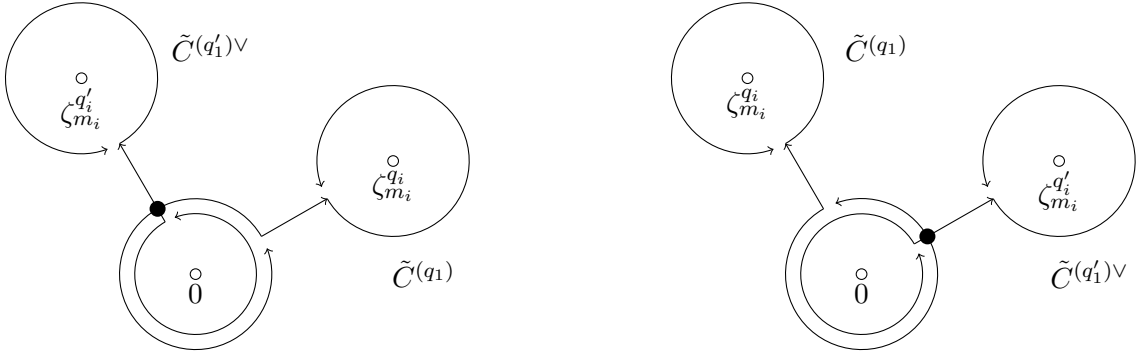
\end{proof}

\begin{prop}\label{prop:eigenvalue-diagonal}
  Let $\Ihgoto =\left( \langle \tC^{({\bf q})}, \tC^{({\bf q}')\vee}\rangle_h 
  \right)_{{\bf q}, {\bf q}' \in \Z^{n\times 1} /\Z \transp{D}} $
  be the intersection matrix. 
  For any fixed vector ${\bf k}\in \Z^{n\times 1} /\Z D=(\Z/m_1 \Z)\times \cdots \times (\Z/m_n \Z)$, 
  the column vector 
  $\left( e^{\tpi \transp{\tilde{\bf q}}D^{-1}{\bf k}}\right)_{\tilde{\bf q}\in \Z^{n\times 1} /\Z \transp{D}}$ 
  is an eigenvector of $\Ihgoto$ whose eigenvalue is 
  \begin{align}
    h_{D,{\bf k}}(\alpha_0 ,\alpha)
    &=\frac{1}{1-e^{\tpi \alpha_0}}
    \frac{1-e^{\tpi \alpha_0}\prod_{i=1}^n(e^{\tpi \frac{\alpha_i}{m_i}}\zeta_{m_i}^{k_i})}
    {\prod_{i=1}^n(1-e^{\tpi \frac{\alpha_i}{m_i}}\zeta_{m_i}^{k_i})}  \\
    &=\frac{1}{1-e^{\tpi \alpha_0 }}
      \frac{1-e^{\tpi \alpha_0}\prod_{i=1}^n e^{\tpi \transp{{\bf e}_i} D^{-1}(\alpha+{\bf k}) }}
      {\prod_{i=1}^n(1-e^{\tpi \transp{{\bf e}_i} D^{-1}(\alpha+{\bf k})} )}  .
  \end{align}
  Here, $\{{\bf e}_1,\dots,{\bf e}_n\}\subset\Z^{n\times 1}$ is the standard basis of $\Z^{n\times 1}$ and we have put $\zeta_m=e^{\frac{\tpi}{m}}$.
\end{prop}
\begin{proof}
  % The ${\bf q}$-th entry of 
  % $\Ihgoto \cdot \left( e^{\tpi \transp{{\bf q}'}M^{-1}{\bf k}}\right)_{{\bf q}'\in G_M}$ is 
  % \begin{align*}
  %   \sum_{j=1}^n \sum_{q_j=0}^{m_j-1}
  %   \frac{\left( 1-e^{\tpi \alpha_0} \prod_{i\in I({\bf q},{\bf q}')} e^{\tpi \alpha_i}  \right) \cdot 
  %   \prod_{i\not\in I({\bf q},{\bf q}')} e^{\tpi \alpha_i \lef \frac{q'_i -q_i}{m_i} \righ }}
  %   {(1-e^{\tpi \alpha_0})\prod_{i=1}^n (1-e^{\tpi \alpha_i})}
  %   e^{\tpi \transp{{\bf q}'}M^{-1}{\bf k}}
  % \end{align*}
  We also write $I_D (\alpha_0, \alpha_1 ,\dots ,\alpha_n)$ for 
  the intersection matrix $\Ihgoto$. 
  We prove the proposition by induction on $n$. 
  First, we assume $n=1$ and ${\bf k}=(k_1)$. 
  The intersection matrix is a circulant matrix expressed as
  \begin{align}
    I_{(m_1)}(\alpha_0, \alpha_1) =\frac{1}{(1-e^{\tpi \alpha_0})(1-e^{\tpi \alpha_1})}
    \left( (1-e^{\tpi (\alpha_0+\alpha_1)})E_{m_1} +
    (1-e^{\tpi \alpha_0})\sum_{q=0}^{m_1-1} e^{\tpi \frac{\alpha_1}{m_1}} Y_{m_1}^q \right), 
  \end{align}
  where $E_{m}$ is the identity matrix of size $m$ and $Y_{m}$ is an $m \times m$ matrix defined by 
  \begin{align}
    Y_{m} =
    \begin{pmatrix}
      0 & 1 & 0 & \cdots & 0 \\
      0 & 0 & 1 & \ddots &\vdots \\
      \vdots &  & \ddots & \ddots & 0\\
      0 &0 &\cdots& 0 &1 \\
      1 & 0 & \cdots & 0 &0
    \end{pmatrix}.
  \end{align}
  The column vector $\left( e^{\tpi \frac{k_1 \tilde{q}_1}{m_1}}\right)_{\tilde{q}_1 \in \Z /m_1 \Z}
  =\left( \zeta_{m_1}^{k_1 \tilde{q}_1}\right)_{\tilde{q}_1 \in \Z /m_1 \Z}$ is 
  an eigenvector of $I_{(m_1)}(\alpha_0, \alpha_1)$ whose eigenvalue is 
  \begin{align}
    &\frac{1}{(1-e^{\tpi \alpha_0})(1-e^{\tpi \alpha_1})}
      \left( (1-e^{\tpi (\alpha_0+\alpha_1)}) +
      (1-e^{\tpi \alpha_0})\sum_{q=0}^{m_1-1} e^{\tpi \frac{\alpha_1}{m_1}} \zeta_{m_1}^{k_1 q} \right) \\
    &=\frac{1}{1-e^{\tpi \alpha_0}} \cdot 
      \frac{1-e^{\tpi \alpha_0}e^{\tpi \frac{\alpha_1}{m_1}}}{1-e^{\tpi \frac{\alpha_1}{m_1}} \zeta_{m_1}^{k_1}}
      =h_{(m_1),(k_1)}(\alpha_0,\alpha_1) .
  \end{align}
  Next, let us assume that the proposition holds for $1,\dots ,n-1$. 
  We set $D'=\diag(m_1 ,\dots ,m_{n-1})$, ${\bf k}'=\transp{(k_1 ,\dots ,k_{n-1})}$,
  $I_1 = (1-e^{\tpi (\alpha_0+\alpha_n)})\cdot  I_{D'} (\alpha_0+\alpha_n , \alpha_1 ,\dots ,\alpha_{n-1})$, and 
  $I_2 = (1-e^{\tpi \alpha_0})\cdot  I_{D'} (\alpha_0 , \alpha_1 ,\dots ,\alpha_{n-1})$.
  % and 
  % \begin{align*}
  %   I_1 = (1-e^{\tpi (\alpha_0+\alpha_n)})\cdot  I_{D'} (\alpha_0+\alpha_n , \alpha_1 ,\dots ,\alpha_{n-1}) ,\quad
  %   I_2 = (1-e^{\tpi \alpha_0})\cdot  I_{D'} (\alpha_0 , \alpha_1 ,\dots ,\alpha_{n-1}) .
  % \end{align*}
  The intersection matrix is expressed as 
  \begin{align}
    I_D (\alpha_0, \alpha_1 ,\dots ,\alpha_n)
    =\frac{1}{(1-e^{\tpi \alpha_0})(1-e^{\tpi \alpha_n})}
    \left( E_{m_n}\ot I_1 
    +\left(\sum_{q=0}^{m_n-1} e^{\tpi \frac{\alpha_n}{m_n}} Y_{m_n}^q \right) \ot I_2\right) . 
  \end{align}
  Since the column vector 
  $\left( e^{\tpi \transp{\tilde{\bf q}}D^{-1}{\bf k}}\right)_{\tilde{\bf q}\in \Z^{n\times 1} /\Z \transp{D}}$
  can be regarded as a tensor product 
  $\left( e^{\tpi \frac{k_n \tilde{q}_n}{m_n}}\right)_{\tilde{q}_n \in \Z /m_n \Z}
  \ot \left( e^{\tpi \transp{\tilde{\bf q}'}{D'}^{-1}{\bf k}'}\right)_{\tilde{\bf q}'\in \Z^{n\times 1} /\Z \transp{D'}}$, 
  this is an eigenvector of $I_D (\alpha_0, \alpha_1 ,\dots ,\alpha_n)$ whose eigenvalue is 
  \begin{align}
    \nonumber
    &\frac{1}{(1-e^{\tpi \alpha_0})(1-e^{\tpi \alpha_n})}
      \Bigg( (1-e^{\tpi (\alpha_0+\alpha_n)})
      \cdot  h_{D',{\bf k}'} (\alpha_0+\alpha_n , \alpha_1 ,\dots ,\alpha_{n-1}) \\
    &\qquad \qquad 
      +\Big(\sum_{q=0}^{m_n-1} e^{\tpi \frac{\alpha_n}{m_n}} \zeta_{m_n}^{k_n q} \Big) 
      \cdot (1-e^{\tpi \alpha_0})\cdot  h_{D',{\bf k}'} (\alpha_0 , \alpha_1 ,\dots ,\alpha_{n-1})\Bigg)  \\
    &= h_{D,{\bf k}}(\alpha_0,\alpha_1,\dots ,\alpha_n) . 
  \end{align}
  Thus, we complete the proof. 
\end{proof}

By a straightforward calculation, we obtain the following corollary. 
\begin{cor}
  Let $l$ be the least common multiple of $m_1 ,\dots , m_n$. We have 
  \begin{align*}
    \det \Ihgoto  =\frac{\left( 1 - e^{\tpi l \alpha_0}
    \prod_{i=1}^n e^{\tpi \frac{l \alpha_i}{m_i}}\right)^{\frac{m_1\cdots m_n}{l}}}
    {(1 - e^{\tpi \alpha_0})^{m_1\cdots m_n} 
    \prod_{i=1}^n (1 - e^{\tpi \alpha_i})^{\frac{m_1\cdots m_n}{m_i}}} .
  \end{align*}
  In particular,
  $\{ [\tC^{({\bf q})} ] \}_{{\bf q}\in \Z^{n\times 1} /\Z \transp{D}} \subset \Homo_n(\tX ;p_D^{-1}\mathcal{L})$ 
  % the twisted cycles $\{ \tC^{({\bf q})}\}_{{\bf q}\in \Z^{n\times 1} /\Z \transp{D}}$ 
  are linearly independent if 
  none of 
  $\alpha_0,\alpha_1 ,\dots ,\alpha_n$ and 
  $\lcm(m_1,\dots m_n) \cdot\left(
    \alpha_0+\frac{\alpha_1}{m_1}+\cdots +\frac{\alpha_n}{m_n}
  \right)$ is an integer.
  % \begin{align*}
  %   \alpha_0,\alpha_1 ,\dots ,\alpha_n ,
  %   \lcm(m_1,\dots m_n) \cdot\left(
  %   \alpha_0+\frac{\alpha_1}{m_1}+\cdots +\frac{\alpha_n}{m_n}
  %   \right) \not\in \Z .
  % \end{align*}
\end{cor}

\noindent
In view of Proposition \ref{prop:purity-dim}, a set 
$\{ [\tC^{({\bf q})} ]\}_{{\bf q}\in \Z^{n\times 1} /\Z \transp{D}}$ forms a basis of $\Homo_n(\tX ;p_D^{-1}\mathcal{L})$. 

\subsection{General cases}\label{subsection:not-diagonal}
Now let us consider the general case when $M$ is not diagonal. 
We choose the diagonal matrix $D=\diag (m_1 ,\dots ,m_n)$ so that 
the auxiliary matrix $\bM=DM^{-1}$ is an integer matrix. 
We consider a covering map 
$p_D : \bX \to \X$ where
$\bX =\{ (y_1,\dots ,y_n) \in \C^n \mid y_1 \cdots y_n (1-y_1^{m_1} -\cdots -y_n^{m_n}) \neq 0\}$ and 
$p_D(y_1 ,\dots ,y_n)=(y_1^{m_1},\dots,y_n^{m_n})$,
% \begin{align*}
%   &\bX =\{ (y_1,\dots ,y_n) \in \C^n \mid y_1 \cdots y_n (1-y_1^{m_1} -\cdots -y_n^{m_n}) \neq 0\} ,\\
%   &p_D : \bX \to \X ;\quad 
%     p_D(y_1 ,\dots ,y_n)=(y_1^{m_1},\dots,y_n^{m_n}) ,
% \end{align*}
which is nothing but that discussed in \S\ref{subsection:diagonal}. 
We have non-splitting exact sequences
\begin{align}
  % \nonumber
  &0\rightarrow \Z^{n\times 1}/\Z M \overset{\bM \times}{\rightarrow}\Z^{n\times 1}/\Z D
    \rightarrow \Z^{n\times 1}/\Z\bM \rightarrow 0 \qquad \textrm{(exact)} ,  \\
  \label{eq:exact-seq-MDMtr}
  &0\rightarrow \Z^{n\times 1}/\Z \transp{\bM} \overset{\transp{M} \times}{\rightarrow}\Z^{n\times 1}/\Z \transp{D}
    \rightarrow \Z^{n\times 1}/\Z \transp{M} \rightarrow 0 \qquad \textrm{(exact)} ,
\end{align}
and the corresponding commutative diagram of Galois coverings 
\begin{align}
  \xymatrix{
  && \bX \ar[ld]_-{p_{\bM}} \ar[ddl]^-{p_D}  \ar@{}[r]|*[r]{\subset (\C^*)^n_y} &   \\
  &\tX \ar[d]_-{p_M}  \ar@{}[l]|*[l]{(\C^*)^n_x \supset} & \\
  & X  \ar@{}[l]|*[l]{(\C^*)^n_{\xi} \supset} & 
  } 
\end{align}
where $p_{\bM} :\bX \to \tX$ is defined by $p_{\bM} (y) =y^{\bM}$. 
The deck transformation group of $p_{\bM} :\bX \to \tX$ is isomorphic to $\Z^{n\times 1}/\Z \transp{\bM}$ and 
for $\tilde{\bf l} \in \Z^{n\times 1}/\Z \transp{\bM}$, 
the corresponding deck transformation is given by 
$g_1^{(\tilde{\bf l})}(y)=e^{\tpi \transp{\bM}^{-1} \tilde{\bf l}}y$. 
% \begin{align*}
%   g_1^{(\tilde{\bf l})}(y)
%   =e^{\tpi \transp{\bM}^{-1} \tilde{\bf l}}y ,\qquad \tilde{\bf l} \in \Z^{n\times 1}/\Z \transp{\bM}. 
% \end{align*}

By the method in \S\ref{subsection:diagonal}, we can construct the lifts 
$\{ \bC^{({\bf q})}\}_{{\bf q}\in \Z^{n\times 1}/\Z \transp{D}} \subset \mathcal{Z}_n(\bX ;p_D^{-1}\mathcal{L})$ 
of $C$. 
\begin{lem}
  \label{lem:lift-well-def}
  For ${\bf q}\in \Z^{n\times 1}/\Z \transp{D}$, $\tilde{\bf l} \in \Z^{n\times 1}/\Z \transp{\bM}$, 
  we have $(p_{\bM})_{*} (\bC^{({\bf q})})=(p_{\bM})_{*} (\bC^{({\bf q}+\transp{M}\tilde{\bf l})})$. 
\end{lem}
\begin{proof}
  By the definition (\ref{eq:def-lifted-cycle}) of $\bC^{({\bf q})}$, we have 
  \begin{align}
    (p_{\bM})_{*} (\bC^{({\bf q})}) 
%    &=\frac{1}{e^{\tpi \transp{\bf q} D^{-1} \bM \alpha  }}
  %    \big( p_{\bM} (\tilde{\triangle}^{({\bf q})}) 
    %  \ot \bu_1 \circ (p_{\bM}|_{{\triangle}^{({\bf q})}})^{-1} +\cdots \big) \\
    &=\frac{1}{e^{\tpi \transp{\bf q} M^{-1} \alpha  }}
      \big( p_{\bM} (\tilde{\triangle}^{({\bf q})}) 
      \ot \bu_1 \circ (p_{\bM}|_{{\triangle}^{({\bf q})}})^{-1} +\cdots \big) ,\\
    (p_{\bM})_{*} (\bC^{({\bf q}+\transp{M}\tilde{\bf l})}) 
%    &=\frac{1}{e^{\tpi \transp{({\bf q}+\transp{M}\tilde{\bf l})} D^{-1} \bM \alpha }}
%      \big( p_{\bM} (\tilde{\triangle}^{({\bf q}+\transp{M}\tilde{\bf l})}) 
%      \ot \bu_2 \circ (p_{\bM}|_{{\triangle}^{({\bf q}+\transp{M}\tilde{\bf l})}})^{-1} +\cdots \big) \\
    &=\frac{1}{e^{\tpi (\transp{\bf q} M^{-1}\alpha + \transp{\tilde{\bf l}} \alpha) }}
      \big( p_{\bM} (\tilde{\triangle}^{({\bf q}+\transp{M}\tilde{\bf l})}) 
      \ot \bu_2 \circ (p_{\bM}|_{{\triangle}^{({\bf q}+\transp{M}\tilde{\bf l})}})^{-1} +\cdots \big) ,
  \end{align}
  where the branches $\bu_1$ and $\bu_2$ of 
  $\bu (y) =u(p_D(y))=y^{\bM \alpha} (1-\sum_{i=1}^n y_i^{m_i})^{\alpha_0}$ are defined 
  in the manner of \S\ref{subsection:diagonal}. 
  Note that we have replaced $\alpha$ in (\ref{eq:def-lifted-cycle}) with $\bM \alpha$. 
  The supports of these twisted cycles are same; for example, 
  $p_{\bM} (\tilde{\triangle}^{({\bf q})}) =p_{\bM} (\tilde{\triangle}^{({\bf q}+\transp{M}\tilde{\bf l})})$ 
  as sets. 
  Since $p_{\bM}|_{{\triangle}^{({\bf q}+\transp{M}\tilde{\bf l})}} 
  =p_{\bM}|_{{\triangle}^{({\bf q})}} \circ (g_1^{(\tilde{\bf l})})^{-1}$ and 
  $\bu_2 \circ g_1^{(\tilde{\bf l})}=e^{\tpi \transp{(\transp{\bM}^{-1} \tilde{\bf l})}\bM \alpha} \cdot \bu_1
  =e^{\tpi \transp{\tilde{\bf l}} \alpha} \cdot \bu_1$, we have
  \begin{align}
    \bu_2 \circ (p_{\bM}|_{{\triangle}^{({\bf q}+\transp{M}\tilde{\bf l})}})^{-1}
    =\bu_2 \circ g_1^{(\tilde{\bf l})} \circ (p_{\bM}|_{{\triangle}^{(\tilde{\bf q})}})^{-1}
    =e^{\tpi \transp{\tilde{\bf l}} \alpha} \cdot \bu_1 \circ (p_{\bM}|_{{\triangle}^{(\tilde{\bf q})}})^{-1} . 
  \end{align}
  This implies $(p_{\bM})_{*} (\bC^{({\bf q})})=(p_{\bM})_{*} (\bC^{({\bf q}+\transp{M}\tilde{\bf l})})$.
\end{proof}
We define twisted cycles 
$\{ \tC^{(\tilde{\bf k})}\}_{\tilde{\bf k}\in \Z^{n\times 1}/\Z \transp{M}} \subset \mathcal{Z}_n(\tX ;p_M^{-1}\mathcal{L})$ by 
$\tC^{(\tilde{\bf k})} =(p_{\bM})_{*} (\bC^{(\tilde{\bf k})})$, 
% \begin{align*}
%   \tC^{(\tilde{\bf k})} =(p_{\bM})_{*} (\bC^{(\tilde{\bf k})}) \quad
%   (\tilde{\bf k} \in \Z^{n\times 1}/\Z \transp{M} ), 
%   % \qquad 
%   % \tilde{\bf k} \in \Z^{n\times 1}/\Z \transp{M} (\twoheadleftarrow \Z^{n\times 1}/\Z \transp{D}). 
% \end{align*}
where we also regard $\tilde{\bf k}$ as a representative of $\Z^{n\times 1}/\Z \transp{D}$
via the surjection $\Z^{n\times 1}/\Z \transp{D} \to \Z^{n\times 1}/\Z \transp{M}$ in (\ref{eq:exact-seq-MDMtr}). 
By Lemma \ref{lem:lift-well-def}, this definition is independent of the choice of the representative 
$\tilde{\bf k} \in \Z^{n\times 1}/\Z \transp{D}$. 
Since the lifts of $\tC^{(\tilde{\bf k})}$ with respect to $p_{\bM}$ is given by 
$\{ \bC^{(\tilde{\bf k}+\transp{M}\tilde{\bf l})} \}_{\tilde{\bf l} \in \Z^{n\times 1}/\Z \transp{\bM}}$, 
we can evaluate the intersection number $\langle \tC^{(\tilde{\bf k})}, \tC^{(\tilde{\bf k}')\vee}\rangle_h$ 
by Proposition \ref{prop:lift-intersection}:
\begin{align}
  \label{eq:intersection-number-sum-ver}
  \langle \tC^{(\tilde{\bf k}_1)} , \tC^{(\tilde{\bf k}_2)\vee} \rangle_h
  =\sum_{\tilde{\bf l} \in \Z^{n\times 1}/\Z \transp{\bM}}
  \langle \bC^{(\tilde{\bf k}_1)} , \bC^{(\tilde{\bf k}_2+\transp{\bM} \tilde{\bf l})\vee}\rangle_h .
\end{align}
Let us set $\Ihgoto =\left(\langle \tC^{(\tilde{\bf k}_1)}, \tC^{(\tilde{\bf k}_2)\vee} \rangle_h \right)_{
\tilde{\bf k}_1,\tilde{\bf k}_2\in\Z^{n\times 1}/\Z \transp{M}}$. Note that the columns and the rows of the matrix $\Ihgoto$ is indexed by the set $\Z^{n\times 1}/\Z \transp{M}$.
%%%%%%% from M-note %%%%%%%%%
\begin{prop}\label{prop:DiagonalizationTheorem1}
  For each fixed ${\bf k}\in\Z^{n\times 1}/\Z M$, the column vector 
  $\left( e^{\tpi \transp{\tilde{\bf k}}M^{-1}{\bf k}} \right)_{\tilde{\bf k} \in \Z^{n\times 1}/\Z \transp{M}}$ 
  is an eigenvector of $\Ihgoto$ whose eigenvalue is 
  \begin{align}
    h_{M,{\bf k}}(\alpha)=\frac{1}{1-e^{\tpi \alpha_0}} \cdot 
    \frac{1-e^{\tpi \alpha_0} \prod_{i=1}^n e^{\tpi p_i(\alpha+{\bf k})}}
    {\prod_{i=1}^n (1-e^{\tpi p_i(\alpha+{\bf k})})},
    % h_{M,{\bf k}}(\alpha)=\frac{1}{1-e^{\tpi \alpha_0}} \cdot 
    % \frac{1-e^{\tpi \alpha_0} \prod_{i=1}^n e^{\tpi \transp{{\bf e}_i} M^{-1}(\alpha+{\bf k})}}
    % {\prod_{i=1}^n (1-e^{\tpi \transp{{\bf e}_i} M^{-1}(\alpha+{\bf k})})}.
  \end{align}
  where for $v\in \C^{n\times 1}$, $p_i(v)$ denotes the $i$-th entry of $M^{-1}v$. 
  % where we set $p_i(v)=\transp{{\bf e}_i} M^{-1}v$ for $v\in \C^{n\times 1}$. 
\end{prop}
\begin{proof}
In view of the formula (\ref{eq:intersection-number-sum-ver}), we have 
\begin{align}
  &(\tilde{\bf k}\text{-th entry of } \Ihgoto \cdot 
    \left( e^{\tpi \transp{\tilde{\bf k}_2} M^{-1}{\bf k}}\right)_{\tilde{\bf k}_2\in\Z^{n\times 1}/\Z \transp{M}}) \\
  &=\sum_{\tilde{\bf k}_2 \in \Z^{n\times 1}/\Z \transp{M}} 
    \langle \tC^{(\tilde{\bf k})} ,\tC^{(\tilde{\bf k}_2)\vee} \rangle_h 
    \cdot e^{\tpi \transp{\tilde{\bf k}_2} M^{-1}{\bf k}}\\
  &=\sum_{\tilde{\bf k}_2 \in\Z^{n\times 1}/\Z \transp{M}} ~
    \sum_{\tilde{\bf l} \in \Z^{n\times 1}/\Z \transp{\bM}} 
    \langle \bC^{(\tilde{\bf k})} ,\bC^{(\tilde{\bf k}_2+\transp{M} \tilde{\bf l})\vee}\rangle_h
    \cdot e^{\tpi \transp{\tilde{\bf k}_2} D^{-1} \bM {\bf k}}\\
  &=\sum_{\tilde{\bf k}_2 \in \Z^{n\times 1}/\Z \transp{M}} ~
    \sum_{\tilde{\bf l} \in \Z^{n\times 1}/\Z \transp{\bM}}
    \langle \bC^{(\tilde{\bf k})}, \bC^{(\tilde{\bf k}_2+\transp{M} \tilde{\bf l})\vee} \rangle_h
    \cdot e^{\tpi \transp{(\tilde{\bf k}_2+\transp{M} \tilde{\bf l})} D^{-1} \bM {\bf k}}\\
  \label{eq:proof-eigenvalue-general}
  &=(\tilde{\bf k}\text{-th entry of } \bIhgoto \cdot 
    \left( e^{\tpi \transp{\tilde{\bf l}} D^{-1} \bM{\bf k}}\right)_{\tilde{\bf l}\in\Z^{n\times 1}/\Z \transp{D}}).
\end{align}
Here, we have used the fact 
$\transp{(\transp{M} \tilde{\bf l})}D^{-1} \bM{\bf k}=\transp{\tilde{\bf l}}{\bf k}\in \Z$ 
and put $\bIhgoto =\left(\langle \bC^{(\tilde{\bf l}_1)}, \bC^{(\tilde{\bf l}_2)\vee}\rangle_h 
\right)_{\tilde{\bf l}_1,\tilde{\bf l}_2\in\Z^{n\times 1}/\Z \transp{D}}$. 
Since $p_D^{-1}\mathcal{L}=\C y^{\bM \alpha}(1-\sum_{i=1}^ny_i^{m_i})^{\alpha_0}$, 
Proposition \ref{prop:eigenvalue-diagonal} gives 
\begin{align*}
  \bIhgoto \cdot 
  \left( e^{\tpi \transp{\tilde{\bf k}} D^{-1} \bM{\bf k}}\right)_{\tilde{\bf k}\in\Z^{n\times 1}/\Z \transp{D}}
  =h_{D,\bM{\bf k}}(\alpha_0,\bM\alpha)
  \left( e^{\tpi \transp{\tilde{\bf k}} D^{-1} \bM {\bf k}}\right)_{\tilde{\bf k}\in\Z^{n\times 1}/\Z \transp{D}}.
\end{align*}
Thus, we have
\begin{align}
  % &(\tilde{\bf k}\text{-th entry of } \Ihgoto \cdot 
  % \left( e^{\tpi \transp{\tilde{\bf k}} M^{-1}{\bf k}}\right)_{\tilde{\bf k}\in\Z^{n\times 1}/\Z \transp{M}}) \\
  (\text{\ref{eq:proof-eigenvalue-general}})
  &=h_{D,\bM {\bf k}}(\alpha_0,\bM \alpha)
    \cdot e^{\tpi \transp{\tilde{\bf k}} D^{-1} \bM {\bf k}} \\
  &=h_{D,DM^{-1}{\bf k}}(\alpha_0,DM^{-1}\alpha)
    \cdot e^{\tpi \transp{\tilde{\bf k}} M^{-1}{\bf k}}\\
  &=\frac{1}{1-e^{\tpi \alpha_0 }}
      \frac{1-e^{\tpi \alpha_0} \prod_{i=1}^n e^{\tpi \transp{{\bf e}_i} D^{-1}( DM^{-1}\alpha+DM^{-1}{\bf k})}}
      {\prod_{i=1}^n(1-e^{\tpi \transp{{\bf e}_i} D^{-1}( DM^{-1}\alpha+DM^{-1}{\bf k})} )} 
    \cdot e^{\tpi \transp{\tilde{\bf k}} M^{-1}{\bf k}} \\
  &=h_{M,{\bf k}}(\alpha)\cdot e^{\tpi \transp{\tilde{\bf k}} M^{-1}{\bf k}}.
\end{align}
\end{proof}

\begin{rem}
  The proposition above shows that the intersection matrix $\Ihgoto$ does not depend on 
  the choice of the covering indices $m_1,\dots,m_n$.
\end{rem}

\begin{cor}
  We have a closed formula
  \begin{align}
    \Ihgoto=\left(\frac{1}{|\det M|}\sum_{{\bf k}\in \Z^{n\times 1}/\Z M} 
    e^{\tpi \transp{(\tilde{\bf k}_1-\tilde{\bf k}_2)} M^{-1}{\bf k}} 
    h_{M,{\bf k}}(\alpha)\right)_{\tilde{\bf k}_1,\tilde{\bf k}_2\in\Z^{n\times 1}/\Z \transp{M}}.
  \end{align}
\end{cor}
\begin{proof}
We put $r=|\det M|$ and $U=\frac{1}{\sqrt{r}}\left( e^{\tpi \transp{\tilde{\bf k}} M^{-1}{\bf k}}
\right)_{\substack{\tilde{\bf k}\in\Z^{n\times 1}/\Z \transp{M} \\ {\bf k}\in\Z^{n\times 1}/\Z M}}$. 
Due to Schur's orthogonality, $U$ is a unitary matrix. 
By Proposition \ref{prop:DiagonalizationTheorem1}, 
we see that $\Ihgoto U = U\diag(h_{M,{\bf k}}(\alpha))_{{\bf k}\in\Z^{n\times 1}/\Z M}$. 
Thus, we have
\begin{align}
  \Ihgoto
  &=U\diag(h_{M,{\bf k}}(\alpha))_{{\bf k}\in\Z^{n\times 1}/\Z M}U^*\\
  &=\left(\frac{1}{r}\sum_{{\bf k}\in\Z^{n\times 1}/\Z M}
    e^{\tpi \transp{(\tilde{\bf k}_1-\tilde{\bf k}_2)} M^{-1}{\bf k}}
    h_{M,{\bf k}}(\alpha)\right)_{\tilde{\bf k}_1,\tilde{\bf k}_2\in\Z^{n\times 1}/\Z \transp{M}}.
\end{align}
\end{proof}

\begin{cor}
  $\Ihgoto$ is non-degenerate if for any ${\bf k}\in\Z^{n\times 1}/\Z M$ and any $i=1,\dots,n$, 
  % none of $\alpha_0,\transp{\bf e}_i M^{-1}(\alpha+{\bf k})$, and 
  none of $\alpha_0$, $p_i(\alpha+{\bf k})$, and 
  $\alpha_0+|M^{-1}(\alpha+{\bf k})|$ is an integer. Moreover, one has
\begin{align}
    \Ihgoto^{-1}=\left(\frac{1}{|\det M|}\sum_{{\bf k}\in\Z^{n\times 1}/\Z M}
    e^{\tpi \transp{(\tilde{\bf k}_1-\tilde{\bf k}_2)} M^{-1}{\bf k}}
    h_{M,{\bf k}}(\alpha)^{-1}\right)_{\tilde{\bf k}_1,\tilde{\bf k}_2\in\Z^{n\times 1}/\Z \transp{M}}.
  \end{align}
\end{cor}

\subsection{Product case and another description of lifts}\label{sec:3.4}
Let us consider the case when the local system is a product of local systems of the form discussed in the previous subsections. We consider $n+s$ complex numbers $\alpha_1,\dots,\alpha_{n+s}$. For each $l=1,\dots,s$, we assume that a natural number $N_l\geq 1$ is given, and we assume $N_1+\cdots +N_s=n$. We put $N_0=1$, and set $\s_l=\{\sum_{a=0}^{l-1}N_a,\sum_{a=0}^{l-1}N_a +1,\dots,\sum_{a=1}^lN_{a}\}$ for $l=1,\dots,s$. For any vector $\tilde{\bf k}={}^t(\tilde{k}_1,\dots,\tilde{k}_n)\in\Z^{n\times 1}$,
we write $\tilde{\bf k}(l)$ for the entries of $\tilde{\bf k}$ indexed by $\s_l$, i.e., we set $\tilde{\bf k}(l)=(\tilde{k}_i)_{i\in\s_l}$. Let $M=({\bf m}(1)\mid\dots\mid {\bf m}(n))$ be an $n\times n$ integer matrix. We put $M_l=({\bf m}(i))_{i\in \s_l}$ for $l=1,\dots,s$, $\alpha=(\alpha_1,\dots,\alpha_n)$ and we set $\mathcal{L}=\C x^\alpha\prod_{l=1}^s(1-\sum_{i\in\s_l}x^{{\bf m}(i)})^{\alpha_{n+l}}$. By setting $\alpha^\prime=\transp{(-\alpha_{n+1},\dots,-\alpha_{n+s},\alpha_1,\dots,\alpha_n)}$, we define an $(n+s)\times (n+s)$ matrix $M^\prime$ by 
\begin{equation}\label{eqn:3.56}
M^\prime=
({\bf m}^\prime(1)|\cdots|{\bf m}^\prime(n+s))=
\left(
\begin{array}{cccc|cccc|c|cccc}
1&1&\cdots&1&0&0&\cdots&0&\cdots&0&0&\cdots&0\\
\hline
0&0&\cdots&0&1&1&\cdots&1&\cdots&0&0&\cdots&0\\
\hline
&&\vdots& & &&\vdots& & \ddots& &&\vdots& \\
\hline
0&0&\cdots&0&0&0&\cdots&0&\cdots&1&1&\cdots&1\\
\hline
{\bf 0}&&M_1& &{\bf 0}& &M_2& &\cdots&{\bf 0}& &M_s&
\end{array}
\right)
\end{equation}
and $\tilde{X}=\{ x\in(\C^*)^n\mid 1-\sum_{i\in\s_l}x^{{\bf m}(i)}\neq 0 (l=1,\dots, s)\}$. As before, we say that $\alpha^\prime$ is non-resonant if for any proper subset $\Gamma\subsetneq\{ 1,\dots,n+s\}$, one has $\alpha^\prime\notin{\rm span}_{\C}\{ {\bf m}^\prime(i)\mid i\in \Gamma\}+\Z^{(n+s)\times 1}$. 
\begin{prop}\label{prop:purity-dim2}
Suppose that the parameter vector $\alpha^\prime$ is non-resonant and $\alpha_{n+1},\dots,\alpha_{n+s}\notin\Z$. Then, we have the purity
\begin{equation}
\Homo_l\left(\tX ;\mathcal{L}\right)=
\begin{cases}
0&(l\neq n)\\
\C^{|\det M|}&(l=n)
\end{cases}
\end{equation}
and the regularization condition is true, i.e., the canonical  morphism ${\rm can}:\Homo_l\left(\tX ;\mathcal{L}\right)\rightarrow\Homo_l^{lf}\left(\tX ;\mathcal{L}\right)$ is an isomorphism for any $l$.

\end{prop}

\noindent
The proof of Proposition \ref{prop:purity-dim2} is same as that of Proposition \ref{prop:purity-dim}.

We consider a covering $p_M:(\C^*)_x^n\rightarrow(\C^*)^n_{\xi}$ defined by $p_M(x)=x^M$
and take integers $m_1,\dots,m_n\in\Z$ so that $D=\diag(m_1,\dots,m_n)$ satisfies $\bar{M}=DM^{-1}\in\Z^{n\times n}$.
We define $p_{\bar{M}}:(\C^*)^n_y\rightarrow(\C^*)^n_{x}$ and $p_D:(\C^*)^n_y\rightarrow(\C^*)^n_{\xi}$ by
$p_{\bar{M}}(y)=y^{\bar{M}}$ and $p_D(y)=y^D$, respectively.
Then, for any $\tilde{\bf k}\in\Z^{n\times 1}$, one can define the twisted cycle
$\tilde{C}^{(\tilde{\bf k})}\in\mathcal{Z}_n\left( \tX;\mathcal{L}\right)$ by $\tilde{C}^{(\tilde{\bf k})}=(p_{\bM})_{*}\left( \bar{C}^{(\tilde{\bf k})}\right)$. Here, $\bar{C}^{(\tilde{\bf k})}$ is the cross product $\bar{C}^{(\tilde{\bf k})}=\bar{C}^{(\tilde{\bf k}(1))}\times \cdots\times \bar{C}^{(\tilde{\bf k}(s))}$ of $\bar{C}^{(\tilde{\bf k}(l))}\in\mathcal{Z}_{N_l}\left( (\C^*)^{\s_l}_{y_{\s_l}}\setminus\{ 1=\sum_{i\in\s_l}y_i^{m_i}\};p_{\bar{M}}^{-1}\mathcal{L}\right)$. We put $I_{h,1}=\left(\langle \tilde{C}^{(\tilde{\bf k}_1)},\tilde{C}^{(\tilde{\bf k}_2)\vee}\rangle_h\right)_{\tilde{\bf k}_1,\tilde{\bf k}_2\in\Z^{n\times 1}/\Z \transp{M}}$.

\begin{prop}\label{DiagonalizationTheorem2}
For any fixed ${\bf k}\in\Z^{n\times 1}/\Z M$, the column vector $\left( e^{2\pi\ii{}^t\tilde{\bf k}M^{-1}{\bf k}}\right)_{\tilde{\bf k}\in\Z^{n\times 1}/\Z \transp{M}}$ is an eigenvector of $I_{h,1}$ whose eigenvalue is 
\begin{equation}
h_{M,{\bf k}}(\alpha)=\prod_{l=1}^k\left(\frac{1}{1-e^{2\pi\ii\alpha_{n+l}}}\frac{1-e^{2\pi\ii\alpha_{n+l}}\prod_{i\in\s_l}e^{2\pi\ii p_i(\alpha+{\bf k})}}{\prod_{i\in\s_l}(1-e^{2\pi\ii p_i(\alpha+{\bf k})})}\right).
% h_{M,{\bf k}}(\alpha)=\prod_{l=1}^k\left(\frac{1}{1-e^{2\pi\ii\alpha_{n+l}}}\frac{1-e^{2\pi\ii\alpha_{n+l}}\prod_{i\in\s_l}e^{2\pi\ii{}^t{\bf e}_iM^{-1}(\alpha+{\bf k})}}{\prod_{i\in\s_l}(1-e^{2\pi\ii{}^t{\bf e}_iM^{-1}(\alpha+{\bf k})})}\right).
\end{equation}
\end{prop}

\noindent
The proof of Proposition \ref{DiagonalizationTheorem2} is same as that of Proposition \ref{prop:DiagonalizationTheorem1}. From Proposition \ref{DiagonalizationTheorem2}, we see that the intersection matrix $I_{h,1}$ does not depend on the choice of $D$.

Let $C^{({\bf 0})}\in \mathcal{Z}_n\left( \tX;\mathcal{L}\right)$ be the cycle constructed in \cite[Appendix 3]{MatsubaraEuler}. For any ${\bf k}\in \Z^{n\times 1}$, we set $\varphi_{\bf k}=\frac{x^{\bf k}}{\prod_{l=1}^k(1-\sum_{i\in\s_l}x^{{\bf m}(i)})}\frac{dx}{x}.$ By construction, we have identities
\begin{align}
\langle\varphi_{\bf k},C^{({\bf 0})}\rangle&:=\int_{C}x^{\alpha}\prod_{l=1}^s\Bigg(1-\sum_{i\in\s_l}x^{{\bf m}(i)}\Bigg)^{\alpha_{n+l}}\varphi_{\bf k}\\
&=\frac{1}{\det M}\int_{\prod_{l=1}^kP_{N_l}}\xi^{M^{-1}(\alpha+{\bf k})}\prod_{l=1}^s\Bigg(1-\sum_{i\in\s_l}\xi_i\Bigg)^{\alpha_{n+l}-1}\frac{d\xi}{\xi}\\
&=\frac{(2\pi\ii)^{n+s}e^{-\pi\ii\sum_{l=1}^s\alpha_{n+l}}}{\det M} \times\nonumber\\
&\ \ \ \ \frac{1}{\prod_{l=1}^s\Gamma(1-\alpha_{n+l})\Gamma\left(\sum_{i\in\s_l} p_i(\alpha+{\bf k})+\alpha_{n+l}\right)\prod_{i\in\s_l}\Gamma(p_i(\alpha+{\bf k}))}.
\end{align}

\noindent
Here, $P_k$ is the $k$-dimensional Pochhammer cycle (\cite[\S6]{Beukers}).
For any $\tilde{\bf k}\in \Z^{n\times 1}$, $C^{(\tilde{\bf k})}$ denotes the deck transform of $C^{({\bf 0})}$ along the loop $\xi\mapsto e^{2\pi\ii{}^t\tilde{\bf k}}\xi$. It is easy to check the formula
$\langle\varphi_{\bf k},C^{(\tilde{\bf k})}\rangle=e^{2\pi\ii{}^t\tilde{\bf k}M^{-1}(\alpha+{\bf k})}\langle\varphi_{\bf k},C^{({\bf 0})}\rangle$.
% \begin{equation}
% \langle\varphi_{\bf k},C^{(\tilde{\bf k})}\rangle=e^{2\pi\ii{}^t\tilde{\bf k}M^{-1}(\alpha+{\bf k})}\langle\varphi_{\bf k},C^{({\bf 0})}\rangle.
% \end{equation}

% \noindent
On the other hand, we can also integrate $\varphi_{\bf k}$ along the cycles $\tilde{C}^{(\tilde{\bf k})}$.
We write $\hat{C}^{(\tilde{\bf k})}$ for the deck transformation of $\tilde{C}^{({\bf 0})}$ along the loop $\xi\mapsto e^{2\pi\ii{}^t\tilde{\bf k}}\xi$ with respect to the covering $p_M:(\C^*)^n_x\rightarrow(\C^*)^n_{\xi}$. Then, we have
$\tilde{C}^{(\tilde{\bf k})}=e^{-2\pi\ii {}^t\tilde{\bf k}M^{-1}\alpha}\cdot \hat{C}^{(\tilde{\bf k})}$.
% \begin{equation}
% \tilde{C}^{(\tilde{\bf k})}=e^{-2\pi\ii {}^t\tilde{\bf k}M^{-1}\alpha}\cdot \hat{C}^{(\tilde{\bf k})}.
% \end{equation}
In fact, we see that the deck transform of $\bar{C}^{(\bf 0)}$ along the loop $\xi\mapsto e^{2\pi\ii{}^t\tilde{\bf k}}\xi$ with respect to the covering $p_D:(\C^*)^n_y\rightarrow(\C^*)^n_{\xi}$ is given by $e^{2\pi\ii{}^t\tilde{\bf k}D^{-1}\bar{M}\alpha}\cdot \bar{C}^{(\tilde{\bf k})}$ in view of the formula (\ref{eq:def-lifted-cycle}) and the formula $p_{\bar{M}}^{-1}\mathcal{L}=\C y^{\bar{M}\alpha}\prod_{l=1}^s(1-\sum_{i\in\s_l}y^{m_i})^{\alpha_{n+l}}$. Since we have 
$
\hat{C}^{(\tilde{\bf k})}=(p_{\bM})_{*}\left(e^{2\pi\ii{}^t\tilde{\bf k}D^{-1}\bar{M}\alpha}\cdot \bar{C}^{(\tilde{\bf k})}\right),
$
we obtain $\hat{C}^{(\tilde{\bf k})}=e^{2\pi\ii{}^t\tilde{\bf k}M^{-1}\alpha}\cdot \tilde{C}^{(\tilde{\bf k})}$. Here, we have used the identity $D^{-1}\bar{M}=M^{-1}$. We can easily check the formula
\begin{align}
\langle\varphi_{\bf k},\tilde{C}^{(\tilde{\bf k})}\rangle&:=\int_{\tilde{C}^{(\tilde{\bf k})}}x^{\alpha}\prod_{l=1}^s(1-\sum_{i\in\s_l}x^{{\bf m}(i)})^{\alpha_{n+l}}\varphi_{\bf k}\\
&=e^{-2\pi\ii{}^t\tilde{\bf k}M^{-1}\alpha}\int_{\hat{C}^{(\tilde{\bf k})}}x^{\alpha}\prod_{l=1}^s(1-\sum_{i\in\s_l}x^{{\bf m}(i)})^{\alpha_{n+l}}\varphi_{\bf k}\\
&=\frac{e^{-2\pi\ii{}^t\tilde{\bf k}M^{-1}\alpha}}{\det M}\int_{C}(e^{2\pi\ii{}^t\tilde{\bf k}}\xi)^{M^{-1}(\alpha+{\bf k})}\prod_{l=1}^s\Bigg(1-\sum_{i\in\s_l}\xi_i\Bigg)^{\alpha_{n+l}-1}\frac{d\xi}{\xi}\\
&=\frac{e^{2\pi\ii{}^t\tilde{\bf k}M^{-1}{\bf k}}}{\det M}\frac{\prod_{l=1}^s\Gamma(\alpha_{n+l})\prod_{i\in\s_l}\Gamma(p_i(\alpha+{\bf k}))}{\prod_{l=1}^s\Gamma\left(\sum_{i\in\s_l}p_i(\alpha+{\bf k})+\alpha_{n+l}\right)}\\
&=\frac{e^{2\pi\ii{}^t\tilde{\bf k}M^{-1}{\bf k}}e^{-\pi\ii |M^{-1}(\alpha+{\bf k})|}}{\prod_{l=1}^s(1-e^{-2\pi\ii\alpha_{n+l}})\prod_{i\in\s_l}(1-e^{-2\pi\ii p_i(\alpha+{\bf k})})}\langle\varphi_{\bf k},C^{({\bf 0})}\rangle, 
\end{align}
and hence we obtain a relation
$\langle\varphi_{\bf k},C^{(\tilde{\bf k})}\rangle=e^{2\pi\ii{}^t\tilde{\bf k}M^{-1}\alpha}m(\alpha+{\bf k},\tilde{\alpha})\langle\varphi_{\bf k},\tilde{C}^{(\tilde{\bf k})}\rangle$, 
where we set 
\begin{equation}
m(\alpha+{\bf k},\alpha_{n+1},\dots,\alpha_{n+k}):=m(\alpha+{\bf k},\tilde{\alpha})
:=
e^{\pi\ii|M^{-1}(\alpha+{\bf k})|}\prod_{l=1}^s(1-e^{-2\pi\ii\alpha_{n+l}})\prod_{i\in\s_l}(1-e^{-2\pi\ii p_i(\alpha+{\bf k})}).
\end{equation}
% We set 
% \begin{equation}
% m(\alpha+{\bf k},\alpha_{n+1},\dots,\alpha_{n+k}):=m(\alpha+{\bf k},\tilde{\alpha})
% :=
% e^{\pi\ii|M^{-1}(\alpha+{\bf k})|}\prod_{l=1}^s(1-e^{-2\pi\ii\alpha_{n+l}})\prod_{i\in\s_l}(1-e^{-2\pi\ii p_i(\alpha+{\bf k})}).
% \end{equation}
% Then, we have a relation
% $\langle\varphi_{\bf k},C^{(\tilde{\bf k})}\rangle=e^{2\pi\ii{}^t\tilde{\bf k}M^{-1}\alpha}m(\alpha+{\bf k},\tilde{\alpha})\langle\varphi_{\bf k},\tilde{C}^{(\tilde{\bf k})}\rangle$.
% \begin{equation}
% \langle\varphi_{\bf k},C^{(\tilde{\bf k})}\rangle=e^{2\pi\ii{}^t\tilde{\bf k}M^{-1}\alpha}m(\alpha+{\bf k},\tilde{\alpha})\langle\varphi_{\bf k},\tilde{C}^{(\tilde{\bf k})}\rangle.
% \end{equation}
Thus, if we introduce 
% a column vector $\Phi=(\varphi_{\bf k})_{{\bf k}\in\Z^{n\times 1}/\Z A}$ and
row vectors $\tilde{\bf c}=([\tilde{C}^{(\tilde{\bf k})}])_{\tilde{\bf k}\in\Z^{n\times 1}/\Z \transp{M}}$ and ${\bf c}=([C^{(\tilde{\bf k})}])_{\tilde{\bf k}\in\Z^{n\times 1}/\Z \transp{M}}$, the linear relation ${\bf c}=\tilde{\bf c}B$ is given by 
% \begin{equation}
% \langle\varphi_{\bf k},C^{(\tilde{\bf k})}\rangle=\diag(m(\alpha+{\bf k},\tilde{\alpha}))_{\bf k}\langle\varphi_{\bf k},\tilde{C}^{(\tilde{\bf k})}\rangle\diag(e^{2\pi\ii{}^t\tilde{\bf k}M^{-1}\alpha})_{\tilde{\bf k}}
% \end{equation}
% or equivalently by
$B=\left( \langle\varphi_{\bf k},\tilde{C}^{(\tilde{\bf k})}\rangle\right)_{{\bf k},\tilde{\bf k}}^{-1} 
\diag(m(\alpha+{\bf k},\tilde{\alpha}))_{\bf k}
\left( \langle\varphi_{\bf k},\tilde{C}^{(\tilde{\bf k})}\rangle\right)_{{\bf k},\tilde{\bf k}}
\diag(e^{2\pi\ii{}^t\tilde{\bf k}M^{-1}\alpha})_{\tilde{\bf k}}$.
% \begin{equation}
% B=\langle\varphi_{\bf k},\tilde{C}^{(\tilde{\bf k})}\rangle^{-1}\diag(m(\alpha+{\bf k},\tilde{\alpha}))_{\bf k}\langle\varphi_{\bf k},\tilde{C}^{(\tilde{\bf k})}\rangle\diag(e^{2\pi\ii{}^t\tilde{\bf k}M^{-1}\alpha})_{\tilde{\bf k}}.
% \end{equation}
By a direct computation, we obtain a closed formula
\begin{equation}
B=\bar{U}\diag(m(\alpha+{\bf k},\tilde{\alpha}))_{\bf k}\transp{U}\diag(e^{2\pi\ii{}^t\tilde{\bf k}M^{-1}\alpha})_{\tilde{\bf k}},
\end{equation}
where $\bar{U}$ is the complex conjugate of $U$. 
The matrix $B^\vee$ corresponding to the parameter $(-\alpha,-\tilde{\alpha})$ is given by\footnote{we have replaced ${\bf k}$ by $-{\bf k}$ to transform $\bar{U}$ into $U$ and $\transp{U}$ to $U^*$.}
\begin{align}
B^\vee&=\bar{U}\diag(m(-\alpha+{\bf k},-\tilde{\alpha}))_{\bf k}\transp{U}\diag(e^{-2\pi\ii{}^t\tilde{\bf k}M^{-1}\alpha})_{\tilde{\bf k}}\\
&=U\diag(m(-\alpha-{\bf k},-\tilde{\alpha}))_{\bf k}U^*\diag(e^{-2\pi\ii{}^t\tilde{\bf k}M^{-1}\alpha})_{\tilde{\bf k}}.
\end{align}

\begin{prop}\label{prop:3.14}
Set $I_{h,2}:=\left(\langle C^{(\tilde{\bf k}_1)},C^{(\tilde{\bf k}_2)\vee}\rangle_h\right)_{\tilde{\bf k}_1,\tilde{\bf k}_2\in\Z^{n\times 1}/\Z \transp{M}}
$ and 
\begin{align}
H_{M,{\bf k}}(\alpha,\tilde{\alpha})&:=m(\alpha+{\bf k},\tilde{\alpha})m(-\alpha-{\bf k},-\tilde{\alpha})h_{M,{\bf k}}(\alpha,\tilde{\alpha})\\
&=\prod_{l=1}^k\left(1-e^{-2\pi\ii\alpha_{n+l}}\right)\left(1-e^{2\pi\ii(\alpha_{n+l}+\sum_{i\in\s_l}p_i(\alpha+{\bf k}))}\right)\prod_{i\in\s_l}\left(1-e^{-2\pi\ii p_i(\alpha+{\bf k})}\right).
\end{align}
Then, one has a formula
\begin{equation}
I_{h,2}=\diag(e^{2\pi\ii{}^t\tilde{\bf k}M^{-1}\alpha})_{\tilde{\bf k}}U\diag(H_{M,{\bf k}}(\alpha,\tilde{\alpha}))_{\bf k}U^*\diag(e^{-2\pi\ii{}^t\tilde{\bf k}M^{-1}\alpha})_{\tilde{\bf k}}\label{MyIntersectionMatrix}.
\end{equation}
\end{prop}

\begin{proof}
We have
\begin{align}
I_{h,2}&={}^tBI_{h,1}B^\vee\\
% &=\langle \transp{\bf c},{\bf c}^\vee\rangle_h\\
% &={}^tBI_{h,1}B^\vee\\
&=\diag(e^{2\pi\ii{}^t\tilde{\bf k}M^{-1}\alpha})_{\tilde{\bf k}}U\diag(m(\alpha+{\bf k},\tilde{\alpha}))_{\bf k}U^*I_{h,1} U\diag(m(-\alpha-{\bf k},-\tilde{\alpha}))_{\bf k}U^*\diag(e^{-2\pi\ii{}^t\tilde{\bf k}M^{-1}\alpha})_{\tilde{\bf k}}\\
&=\diag(e^{2\pi\ii{}^t\tilde{\bf k}M^{-1}\alpha})_{\tilde{\bf k}}U\diag(H_{M,{\bf k}}(\alpha,\tilde{\alpha}))_{\bf k}U^*\diag(e^{-2\pi\ii{}^t\tilde{\bf k}M^{-1}\alpha})_{\tilde{\bf k}}.
\end{align}
\end{proof}

%%%%%%%%%%%%%%%%%%%%%%%%%%%%%%%%%%%%%%%%%%%%%%%%%%%%%%%%%%%%%%%%%%%%%%%%%%%%%%% 

\subsection{The proof of Theorem \ref{thm:SigmaIntersectionMatrix1}}\label{subsection:GKZcase}

We are going to prove Theorem \ref{thm:SigmaIntersectionMatrix1}. We fix a regular triangulation $T$ of $A$ given by (\ref{CayleyConfigu}) and fix a simplex $\s\in T$. In view of the discussion of \cite[\S7]{MatsubaraEuler}, we are reduced to the study of the homology intersection form of a particular twisted homology group explained below. 

We consider a local system $\C\Phi=\C \displaystyle
\prod_{l=1}^k\Big( \sum_{i\in\s^{(l)}}x^{{\bf a}^{(l)}(i)}\Big)^{-\gamma_l}x^c$, where $x=(x_1,\dots,x_n)$. We put $A_{\s^{(l)}}=({\bf a}^{(l)}(i))_{i\in\s^{(l)}}$, $\s=\s^{(1)}\cup\cdots\cup\s^{(k)}$, $A^\prime=(A_{\s^{(1)}}|\cdots|A_{\s^{(k)}})$ and 
\begin{equation}\label{CayleyConfig}
A_{\s}
=
\left(
\begin{array}{ccc|ccc|c|ccc}
1&\cdots&1&0&\cdots&0&\cdots&0&\cdots&0\\
\hline
0&\cdots&0&1&\cdots&1&\cdots&0&\cdots&0\\
\hline
 &\vdots& & &\vdots& &\ddots& &\vdots& \\
\hline
0&\cdots&0&0&\cdots&0&\cdots&1&\cdots&1\\
\hline
 &A_{\s^{(1)}}& & &A_{\s^{(2)}}& &\cdots & &A_{\s^{(k)}}& 
\end{array}
\right).
\end{equation}

\noindent
We set $X:=\{ x\in (\C^*)^n\mid \sum_{i\in\s^{(l)}}x^{{\bf a}^{(l)}(i)}\neq 0,\ l=1,\dots,k\}$.  We assume $|\s^{(l)}|\geq 1$ and $A_\s$ is a square invertible matrix of size $n+k$. We write $\s^{(l)}=\{ i_1^{(l)},\dots,i_{|\s^{(l)}|}^{(l)}\}$. We set $\bs^\prime=\{ i_1^{(l)}\}_{l=1}^k$ and $\s^\prime=\s\setminus\bs^\prime$. The set $\bs^\prime$ is naturally identified with the set $\{ 1,\dots,k\}$. We consider the preliminary reduction of the matrix $A_\s$ so that it fits in the setting of the \S\ref{sec:3.4}. We consider the row reduction of $A_{\s}$ obtained by left multiplication of 
\begin{equation}
Q=\left(
\begin{array}{c|c}
E_{k}&{\bf O}\\
\hline
-A^\prime_{\bs^\prime}&E_n
\end{array}
\right).
\end{equation}
Here, $E_k$ denotes the identity matrix of size $k$. For any matrix $B=({\bf b}(1)|\cdots|{\bf b}(m))$ and any partition $I\sqcup J=\{ 1,\dots,m\}$, we write $B$ as $B=\left( ({\bf b}(i))_{i\in I}|({\bf b}(j))_{j\in J}\right)$. We write $E_{\bs^\prime}$ for the identity matrix whose columns are labeled by the set $\bs^\prime$. We can rewrite $A_\s$ as
\begin{equation}
A_\s=\left(
\begin{array}{c|c}
E_{\bs^\prime}&S_{\s^\prime}\\
\hline
A^\prime_{\bs^\prime}&A^\prime_{\s^\prime}
\end{array}
\right)
\end{equation}
and $QA_\s$ as

\begin{equation}
QA_\s=
\left(
\begin{array}{c|c}
E_{\bs^\prime}&S_{\s^\prime}\\
\hline
{\bf O}&A_{\s^\prime,\bs^\prime}
\end{array}
\right),
\end{equation}
where we have put $A_{\s^\prime,\bs^\prime}=A^\prime_{\s^\prime}-A^\prime_{\bs^\prime}S_{\s^\prime}$. Note that the matrix $QA_\s$ takes the form of (\ref{eqn:3.56}). This reduction corresponds to the identity
\begin{equation}
\Phi=\prod_{l=1}^k\Bigg( \sum_{i\in\s^{(l)}}x^{{\bf a}^{(l)}(i)-{\bf a}^{(l)}(i^{(l)}_1)}\Bigg)^{-\gamma_l}x^{c-\sum_{l=1}^k\gamma_l{\bf a}^{(l)}(i^{(l)}_1)}.
\end{equation}

\noindent
From this expression, we easily see that the deck transformation group of the corresponding covering $(\C^*)^n\ni x\mapsto x^{A_{\s^\prime,\bs^\prime}}\in (\C^*)^n$ can be identified with $\Z^{\s^\prime\times 1}/\Z \transp{A}_{\s^\prime,\bs^\prime}$ and its dual group is $\Z^{n\times 1}/\Z A_{\s^\prime,\bs^\prime}.$ However, it is also useful to stick to the original matrix $A_\s$ to make formulae clean. Firstly, we have the natural isomorphism
\begin{equation}
\Z^{(n+k)\times 1}/\Z A_\s\overset{Q\times}{\overset{\sim}{\rightarrow}}\Z^{(n+k)\times 1}/\Z QA_\s.
\end{equation}
The right-hand side is isomorphic to $\Z^{n\times 1}/\Z A_{\s^\prime,\bs^\prime}$ via projection ${\rm pr}_1:\Z^{(n+k)\times 1}\rightarrow\Z^{n\times 1}$ truncating the first $k$ entries. Therefore, we obtain a natural isomorphism
\begin{equation}\label{Isom1}
\Z^{(n+k)\times 1}/\Z A_\s\overset{\text{pr}_1\circ Q\times}{\overset{\sim}{\rightarrow}}\Z^{n\times 1}/\Z A_{\s^\prime,\bs^\prime}.
\end{equation}

\noindent
We put 
\begin{equation}
R=
\left(
\begin{array}{c|c}
E_{\bs^\prime}&{\bf O}\\
\hline
-{}^tS_{\s^\prime}&E_{\s^\prime}
\end{array}
\right).
\end{equation}
If we write ${\rm pr}_2:\Z^{\s\times 1}\rightarrow\Z^{\s^\prime\times 1}$ for the projection truncating the entries indexed by $\bs^\prime$, we have a natural isomorphism
\begin{equation}\label{Isom2}
\Z^{\s\times 1}/\Z \transp{A}_\s\overset{\text{pr}_2\circ R\times}{\overset{\sim}{\rightarrow}}\Z^{\s^\prime\times 1}/\Z \transp{A}_{\s^\prime,\bs^\prime}.
\end{equation}

\noindent
We write any element ${\bf k}\in\Z^{(n+k)\times 1}$ as 
$
{\bf k}=
\begin{pmatrix}
{\bf k}_1\\
{\bf k}_2
\end{pmatrix}
$ with ${\bf k}_1\in\Z^{k\times 1}$ and ${\bf k}_2\in\Z^{n\times 1}$. In the same way, we can write $\tilde{\bf k}\in\Z^{\s\times 1}$ as 
$
\tilde{\bf k}=
\begin{pmatrix}
\tilde{\bf k}_{\bs^\prime}\\
\tilde{\bf k}_{\s^\prime}
\end{pmatrix}
$
with ${\bf k}_{\bs^\prime}\in\Z^{{\bs^\prime}\times 1}$ and $\tilde{\bf k}_{\s^\prime}\in\Z^{\s^\prime\times 1}$.
Then, we easily see that the identity
\begin{align}
{}^t\tilde{\bf k}A_\s^{-1}{\bf k}
&=
{}^t(R\tilde{\bf k})(QA_\s {}^tR)^{-1}(Q{\bf k})\\
&=
{}^t\tilde{\bf k}_{\bs^\prime}{\bf  k}_1+{}^t(\tilde{\bf k}_{\s^\prime}-\transp{S}_{\s^\prime}\tilde{\bf k}_{\bs^\prime})A_{\s^\prime,\bs^\prime}^{-1}({\bf k}_2-A_{\bs^\prime}{\bf k}_1)\\
&\equiv {}^t(\tilde{\bf k}_{\s^\prime}-\transp{S}_{\s^\prime}\tilde{\bf k}_{\bs^\prime})A_{\s^\prime,\bs^\prime}^{-1}({\bf k}_2-A_{\bs^\prime}{\bf k}_1)\;\; \text{(mod } \Z\text{ )}\label{eqn:3.88}
\end{align}
holds. For a moment, let us only consider representatives $[\tilde{\bf k}]\in\Z^{\s\times 1}/\Z \transp{A}_\s$ with $\tilde{\bf k}_{\bs^\prime}={\bf 0}$ and $[{\bf k}]\in\Z^{(n+k)\times 1}/\Z A_\s$ with ${\bf k}_1={\bf 0}$. Under the identifications of finite abelian groups (\ref{Isom1}) and (\ref{Isom2}), we construct a basis $\{ [C^{(\tilde{\bf k}_{\s^\prime})}]\}_{\tilde{\bf k}\in\Z^{\s\times 1}/\Z \transp{A}_\s}$ of $\Homo_n(X;\C\Phi)$ as in \S\ref{sec:3.4}. We put 
$\gamma=\transp{(\gamma_1,\dots,\gamma_k)}$ and $\delta=\transp{(\gamma_1,\dots,\gamma_k,c_1,\dots,c_n)}$. We easily see the formula
\begin{equation}\label{eqn:390}
A_\s^{-1}\delta=
\begin{pmatrix}
\gamma+S_{\s^\prime}A_{\s^\prime,\bs^\prime}^{-1}(A_{\bs^\prime}\gamma-c)\\
\hline
A_{\s^\prime,\bs^\prime}^{-1}(c-A_{\bs^\prime}\gamma)
\end{pmatrix}.
\end{equation}

\noindent
Note that $A_\s^{-1}\delta$ is a column vector of which entries are labeled by the elements of $\s$. If $i\in\s^\prime,$ (\ref{eqn:390}) tells us that $p_{\s i}(\delta)$ is equal to the $i$-th entry of the column vector $A_{\s^\prime,\bs^\prime}^{-1}(c-A_{\bs^\prime}\gamma)$. On the other hand, we have the formula
\begin{equation}
p_{\s i^{(l)}_1}(\delta)=\gamma_l-\sum_{i\in\s^{(l)},i\neq i_1^{(l)}}\Big( i\text{-th entry of }A_{\s^\prime,\bs^\prime}^{-1}(c-A_{\bs^\prime}\gamma)\Big).\label{eqn:3.90}
\end{equation}

\noindent
In view of the formula (\ref{MyIntersectionMatrix}), we put
\begin{equation}\label{eq:Eigen}
H_{A_\s,{\bf k}}(\delta)=\prod_{l:|\s^{(l)}|>1}\left(1-e^{2\pi\ii\gamma_l}\right)\prod_{i\in\s^{(l)}}\left( 1-e^{-2\pi\ii p_{\s i}(\delta+{\bf k})}\right)
\end{equation}
to obtain the basic formula 
of the intersection matrix $I_{h,3}:=\left( \langle C^{(\tilde{\bf k}_{1,\s^\prime})},C^{(\tilde{\bf k}_{2,\s^\prime})\vee} \rangle_h\right)_{\tilde{\bf k}_1,\tilde{\bf k}_2\in \Z^{\s\times 1}/\Z\transp A_\s}$ in view of (\ref{eqn:3.88}), (\ref{eqn:3.90}) and Proposition \ref{prop:3.14}. 
\begin{prop}
Under the notation above, we have
\begin{equation}\label{TheIntersectionMatrix}
I_{h,3}=\diag\left( e^{2\pi\ii{}^t\tilde{\bf k}A_\s^{-1}\delta}\right)_{\tilde{\bf k}}U\diag\left( H_{A_\s,{\bf k}}(\delta)\right)_{\bf k}U^*\diag\left( e^{-2\pi\ii{}^t\tilde{\bf k}A_\s^{-1}\delta}\right)_{\tilde{\bf k}}.
\end{equation}
\end{prop}
In the proposition above, we only considered representatives $[\tilde{\bf k}]\in\Z^{\s\times 1}/\Z \transp{A}_\s$ with $\tilde{\bf k}_{\bs^\prime}={\bf 0}$ and $[{\bf k}]\in\Z^{(n+k)\times 1}/\Z A_\s$ with ${\bf k}_1={\bf 0}$. For any element $\tilde{\bf k}\in \Z^{\s\times 1}$, we set $\Gamma^\prime_{\s,\tilde{\bf k}}:=e^{2\pi\ii{}^t\tilde{\bf k}_{\bs^\prime}\gamma}\cdot C^{(\tilde{\bf k}_{\s^\prime}-{}^tS_{\s^\prime}\tilde{\bf k}_{\bs^\prime})}$. This is nothing but the integration cycle constructed in \cite[\S6]{MatsubaraEuler}. We set $I_{\s,h}^\prime:=\left( \langle \Gamma^\prime_{\s,\tilde{\bf k}_1},\check{\Gamma}^\prime_{\s,\tilde{\bf k}_2}\rangle_h\right)_{\tilde{\bf k}_1,\tilde{\bf k}_2\in \Z^{\s\times 1}/\Z\transp A_\s}$. In view of the argument of \cite[\S7]{MatsubaraEuler}, $I_{\s,h}^\prime$ coincides with the matrix $I_{\s,h}$ in \S\ref{sec:2.3}. Let us derive a formula of $I_{\s,h}$. If an element $[{\bf k}]\in\Z^{(n+k)\times 1}/\Z A_\s$ satisfies ${\bf k}_1={\bf 0}$, one can see that $\sum_{i\in\s^{(l)}} p_{\s i}({\bf k})=k_{i_1^{(l)}}=0$ (see e.g., \cite[Lemma 6.3]{MatsubaraEuler}). Therefore, we have $I_{\s,h}=I_{h,3}$. For any element ${\bf k}\in\Z^{(n+k)\times 1}$, we formally define $H_{A_\s,{\bf k}}(\delta)$ by the formula (\ref{eq:Eigen}).
Taking into account the formula 
\begin{equation}
\langle \Gamma_{\s,\tilde{\bf k}_1},\check{\Gamma}_{\s,\tilde{\bf k}_2}\rangle_h
=
e^{2\pi\ii{}^t(\tilde{\bf k}_{1\bs^\prime}-\tilde{\bf k}_{2\bs^\prime})\gamma}\langle C^{(\tilde{\bf k}_{1\s^\prime}-{}^tS_{\s^\prime}\tilde{\bf k}_{1\bs^\prime})},C^{(\tilde{\bf k}_{2\s^\prime}-{}^tS_{\s^\prime}\tilde{\bf k}_{2\bs^\prime})\vee}\rangle_h
\end{equation}
and the formula
\begin{equation}
{}^t\tilde{\bf k}A_\s^{-1}\delta={}^t(\tilde{\bf k}_{\s^\prime}-{}^tS_{\s^\prime}\tilde{\bf k}_{\bs^\prime})A_{\s^\prime,\bs^\prime}^{-1}(c-A_{\bs^\prime}\gamma)+{}^t\tilde{\bf k}_{\bs^\prime}\gamma,
\end{equation}
% we have the final form of the intersection matrix.
% \begin{thm}
% \begin{equation}\label{SigmaIntersectionMatrix-2}
% I_{\s,h}=\diag\left( e^{2\pi\ii{}^t\tilde{\bf k}A_\s^{-1}\delta}\right)_{\tilde{\bf k}}U\diag\left( H_{A_\s,{\bf k}}(\delta)\right)_{\bf k}U^*\diag\left( e^{-2\pi\ii{}^t\tilde{\bf k}A_\s^{-1}\delta}\right)_{\tilde{\bf k}}.
% \end{equation}
% \end{thm}
we obtain the formula of the intersection matrix:
\begin{equation}\label{SigmaIntersectionMatrix-2}
I_{\s,h}=\diag\left( e^{2\pi\ii{}^t\tilde{\bf k}A_\s^{-1}\delta}\right)_{\tilde{\bf k}}U\diag\left( H_{A_\s,{\bf k}}(\delta)\right)_{\bf k}U^*\diag\left( e^{-2\pi\ii{}^t\tilde{\bf k}A_\s^{-1}\delta}\right)_{\tilde{\bf k}}.
\end{equation}
Therefore, we complete the proof of Theorem \ref{thm:SigmaIntersectionMatrix1}. 

\begin{rem}
The $(\tilde{\bf k}_1,\tilde{\bf k}_2)$-entry $(I_{\s,h})_{\tilde{\bf k}_1,\tilde{\bf k}_2}$ of the formula (\ref{SigmaIntersectionMatrix-2}) is given by
\begin{equation}\label{k1k2Component}
(I_{\s,h})_{\tilde{\bf k}_1,\tilde{\bf k}_2}=\frac{1}{r}\sum_{i=1}^re^{2\pi\ii{}^t(\tilde{\bf k}_1-\tilde{\bf k}_2)A_\s^{-1}(\delta+{\bf k}_i)}H_{A_\s,{\bf k}_i}(\delta),
\end{equation}
where $\{ [{\bf k}_1],\dots,[{\bf k}_r]\}$ is a complete set of representatives of $\Z^{(n+k)\times 1}/\Z A_\s$. Each summand is clearly independent of the choice of the representative ${\bf k}_i$. Therefore, we can conclude that the formula (\ref{k1k2Component}) holds for any choice of a complete set of representatives $\{ [{\bf k}_1],\dots,[{\bf k}_r]\}=\Z^{(n+k)\times 1}/\Z A_\s$.
\end{rem}

\end{document}